\documentclass[11 pt]{amsart}
\usepackage{latexsym,amscd,amssymb, graphicx, amsthm}  
\usepackage{young}
\usepackage{tikz}

\usepackage[margin=1in]{geometry}

\numberwithin{equation}{section}

\newtheorem{theorem}{Theorem}[section]

\newtheorem{corollary}[theorem]{Corollary}
\newtheorem{lemma}[theorem]{Lemma}
\newtheorem{conjecture}[theorem]{Conjecture}

\newtheorem{problem}[theorem]{Problem}

\newtheorem{defn}[theorem]{Definition}
\theoremstyle{definition}

\newcommand{\comaj}{{\mathrm {comaj}}}
\newcommand{\maj}{{\mathrm {maj}}}
\newcommand{\inv}{{\mathrm {inv}}}
\newcommand{\minimaj}{{\mathrm {minimaj}}}
\newcommand{\sign}{{\mathrm {sign}}}
\newcommand{\area}{{\mathrm {area}}}
\newcommand{\dinv}{{\mathrm {dinv}}}
\newcommand{\cont}{{\mathrm {cont}}}

\newcommand{\Des}{{\mathrm {Des}}}
\newcommand{\Val}{{\mathrm {Val}}}
\newcommand{\Rise}{{\mathrm {Rise}}}
\newcommand{\Stir}{{\mathrm {Stir}}}
\newcommand{\Hilb}{{\mathrm {Hilb}}}
\newcommand{\grFrob}{{\mathrm {grFrob}}}
\newcommand{\des}{{\mathrm {des}}}
\newcommand{\asc}{{\mathrm {asc}}}
\newcommand{\coinv}{{\mathrm {coinv}}}
\newcommand{\rev}{{\mathrm {rev}}}
\newcommand{\Asc}{{\mathrm {Asc}}}

\newcommand{\Ind}{{\mathrm {Ind}}}

\newcommand{\SYT}{{\mathrm {SYT}}}
\newcommand{\SSK}{{\mathrm {SSK}}}
\newcommand{\Frob}{{\mathrm {Frob}}}
\newcommand{\triv}{{\mathrm {triv}}}
\newcommand{\initial}{{\mathrm {in}}}

\newcommand{\iDes}{{\mathrm {iDes}}}
\newcommand{\shape}{{\mathrm {shape}}}

\newcommand{\symm}{{\mathfrak{S}}}

\newcommand{\wt}{{\mathrm{wt}}}

\newcommand{\CC}{{\mathbb {C}}}
\newcommand{\QQ}{{\mathbb {Q}}}
\newcommand{\ZZ}{{\mathbb {Z}}}

\newcommand{\OP}{{\mathcal{OP}}}
\newcommand{\DDD}{{\mathcal{D}}}

\newcommand{\CCC}{{\mathcal{C}}}
\newcommand{\LD}{{\mathcal {LD}}}
\newcommand{\AAA}{{\mathcal{A}}}
\newcommand{\MMM}{{\mathcal{M}}}
\newcommand{\BBB}{{\mathcal{B}}}
\newcommand{\GS}{{\mathcal{GS}}}

\newcommand{\zz}{{\mathbf {z}}}
\newcommand{\xx}{{\mathbf {x}}}
\newcommand{\II}{{\mathbf {I}}}
\newcommand{\yy}{{\mathbf {y}}}
\newcommand{\TT}{{\mathbf {T}}}
\newcommand{\mm}{{\mathbf {m}}}

\begin{document}

\title[Ordered set partitions, generalized coinvariant algebras, and the Delta Conjecture]
{Ordered set partitions, generalized coinvariant algebras, and the Delta Conjecture}

\author{James Haglund}
\address
{Department of Mathematics \newline \indent
University of Pennsylvania \newline \indent
Philadelphia, PA, 19104-6395, USA}
\email{jhaglund@math.upenn.edu}

\author{Brendon Rhoades}
\address
{Department of Mathematics \newline \indent
University of California, San Diego \newline \indent
La Jolla, CA, 92093-0112, USA}
\email{bprhoades@math.ucsd.edu}

\author{Mark Shimozono}
\address
{Department of Mathematics, MC 0123 \newline \indent
460 McBryde Hall, Virginia Tech \newline \indent
225 Stanger St. \newline \indent
Blacksburg, VA, 24601, USA}
\email{mshimo@math.vt.edu}

\begin{abstract}
The symmetric group $\symm_n$ acts on the polynomial ring $\QQ[\xx_n] = \QQ[x_1, \dots, x_n]$
by variable permutation.  The invariant ideal $I_n$ is the ideal generated by all $\symm_n$-invariant
polynomials with vanishing constant term.  The quotient $R_n = \frac{\QQ[\xx_n]}{I_n}$ is called the 
{\em coinvariant algebra}.  The coinvariant algebra $R_n$ has received a great deal of study in algebraic and 
geometric combinatorics.  We introduce a generalization 
$I_{n,k} \subseteq \QQ[\xx_n]$ of the ideal $I_n$ indexed by two positive integers $k \leq n$.
The corresponding quotient $R_{n,k} := \frac{\QQ[\xx_n]}{I_{n,k}}$ carries a graded action of $\symm_n$ and
specializes to $R_n$ when $k = n$.
We generalize many of the nice properties of $R_n$ to $R_{n,k}$.  In particular, we describe the Hilbert series
of $R_{n,k}$, give extensions of the Artin and Garsia-Stanton monomial bases of $R_n$ to $R_{n,k}$,
determine the reduced Gr\"obner basis for $I_{n,k}$ with respect to the lexicographic monomial order,
and describe the graded Frobenius series of $R_{n,k}$.
Just as the combinatorics of $R_n$ are controlled by permutations in $\symm_n$, 
we will show that the combinatorics
of $R_{n,k}$ are controlled by ordered set partitions of $\{1, 2, \dots, n\}$ with $k$ blocks.
The {\em Delta Conjecture} of Haglund, Remmel, and Wilson is a generalization of the Shuffle Conjecture
in the theory of diagonal coinvariants.  We will show that
the graded Frobenius series of $R_{n,k}$ is (up to a minor twist) the $t = 0$ specialization 
of the combinatorial side of the Delta Conjecture.
It remains an open problem to give a bigraded $\symm_n$-module $V_{n,k}$ whose Frobenius image 
is even conjecturally equal to any of the expressions in the Delta Conjecture; our module $R_{n,k}$
solves this problem in the specialization $t = 0$.
\end{abstract}

\maketitle

\section{Introduction}
\label{Introduction}

The purpose of this paper is to generalize the {\em coinvariant algebra} -- a representation 
whose combinatorics is controlled by permutations -- to a new class of graded representations
whose combinatorics will be controlled by ordered set partitions.  Our new representations
will be related to Macdonald polynomial theory in that their graded Frobenius character will
be a specialization of the combinatorial side of the {\em Delta Conjecture}
of Haglund, Remmel, and Wilson \cite{HRW}.

Let us recall the classical coinvariant algebra $R_n$.
For a positive integer $n$, 
let $\xx_n = (x_1, \dots , x_n)$ be a list of $n$ variables.
The symmetric group $\symm_n$ acts on the polynomial 
ring $\QQ[\xx_n] = \QQ[x_1, \dots, x_n]$ by variable permutation.  The corresponding
invariant subring
\begin{equation*}
\QQ[\xx_n]^{\symm_n} := \{ f \in \QQ[\xx_n] \,:\, \pi.f = f \text{ for all $\pi \in \symm_n$} \}
\end{equation*}
of {\em symmetric polynomials}
has algebraically independent generators 
$e_1(\xx_n), e_2(\xx_n), \dots, e_n(\xx_n),$
where $e_d(\xx_n)$ denotes the  {\em elementary symmetric function} of degree $d$, i.e.
$e_d(\xx_n) := \sum_{1 \leq i_1 < \cdots < i_d \leq n} x_{i_1} \cdots x_{i_d}.$

The {\em invariant ideal}
$I_n := \langle \QQ[\xx_n]^{\symm_n}_+ \rangle \subseteq \QQ[\xx_n]$ is generated by the 
set 
$\QQ[\xx_n]^{\symm_n}_+$
of symmetric polynomials with vanishing constant term. We have 
\begin{equation*}
I_n = \langle e_1(\xx_n), e_2(\xx_n), \dots, e_n(\xx_n) \rangle.
\end{equation*}
The {\em coinvariant algebra} $R_n$ is the polynomial ring $\QQ[x_1, \dots, x_n]$ modulo this ideal:
\begin{equation*}
R_n := \frac{\QQ[\xx_n]}{I_n} 
=  \frac{\QQ[\xx_n]}{\langle e_1(\xx_n), e_2(\xx_n), \dots, e_n(\xx_n) \rangle}.
\end{equation*}
The algebra $R_n$ inherits a graded action of the symmetric group $\symm_n$ from the
polynomial ring $\QQ[\xx_n]$.

The coinvariant algebra $R_n$ is among the most important representations in  combinatorics.
Let us recall some of its  properties, deferring various definitions to Section~\ref{Background}. 
We will use the usual $q$-analogs of numbers, factorials, and multinomial coefficients:

\begin{align*}
[n]_q := 1 + q + \cdots + q^{n-1} &   &[n]!_q := [n]_q [n-1]_q \cdots [1]_q  \\
{n \brack a_1, \dots , a_r}_q := \frac{[n]!_q}{[a_1]!_q \cdots [a_r]!_q} 
& &{n \brack a}_q := \frac{[n]!_q}{[a]!_q [n-a]!_q}.
\end{align*}

\begin{itemize}
\item  Artin 
 showed that the set of `sub-staircase' monomials
\begin{equation*}
\AAA_n := \{ x_1^{i_1} x_2^{i_2} \cdots x_n^{i_n} \,:\, 0 \leq i_j < j \}
\end{equation*}
descends to a basis of $R_n$ as a $\QQ$-vector space \cite{Artin}.  
In particular, we have $\dim(R_n) = n!$ and the
 Hilbert series $\Hilb(R_n; q)$ of $R_n$ is given by
 \begin{equation*}
 \Hilb(R_n; q) = [n]!_q.
 \end{equation*}
 \item  As an
 {\em ungraded} $\symm_n$-module $R_n$ is isomorphic  to the regular representation of $\symm_n$:
 \begin{equation*}
 R_n \cong_{\symm_n} \QQ[\symm_n].
 \end{equation*}
 In particular, the module $R_n$ gives a graded refinement of the multiplication action of permutations
 on themselves \cite{C}.
 \item  Lusztig (unpublished) and Stanley (see \cite[Prop. 4.11]{Stanley})
 described the {\em graded} isomorphism type of $R_n$ 
 in terms of the {\em major index} statistic on standard Young tableaux.
 In particular, if $\grFrob(R_n; q)$ is the graded Frobenius image of $R_n$, we have
 \begin{equation*}
 \grFrob(R_n; q) = \sum_{T \in \SYT(n)}  q^{\maj(T)} s_{\shape(T)}(\xx) = 
Q'_{(1^n)}(\xx;q),
 \end{equation*}
 where $Q'_{(1^n)}(\xx;q)$ is the dual Hall-Littlewood symmetric function.  By the Schensted 
 correspondence this is equivalent to
 \begin{equation*}
 \grFrob(R_n; q) = \sum_{\text{$w = w_1 \dots w_n$}} q^{\maj(w)} x_{w_1} \cdots x_{w_n}
 \end{equation*}
 where the sum is over all length $n$ words $w_1 \dots w_n$ over the alphabet of positive integers.
\end{itemize}

For two positive integers $k \leq n$, let $\OP_{n,k}$ denote the collection of ordered set partitions of 
$[n] := \{1, 2, \dots, n\}$ with $k$ blocks.  In particular, when $k = n$, ordered set partitions
are permutations and we may identify
$\OP_{n,n} = \symm_n$.  
We have
\begin{equation*}
|\OP_{n,k}| = k! \cdot \Stir(n,k),
\end{equation*}
where $\Stir(n,k)$ is the Stirling number of the second kind counting $k$-block set partitions of $[n]$.
The set $\OP_{n,k}$ carries a natural action of the group $\symm_n$
which reduces to the regular representation when $k = n$.

The  object of study in this paper is the following generalization $I_{n,k}$  of the classical  invariant ideal
$I_n  = 
\langle e_1(\xx_n), e_2(\xx_n), \dots, e_n(\xx_n) \rangle$.

\begin{defn}
Given two positive integers $k \leq n$, let $I_{n,k} \subseteq \QQ[\xx_n]$ be the ideal
\begin{equation}
I_{n,k} := \langle x_1^k, x_2^k, \dots, x_n^k, e_n(\xx_n), e_{n-1}(\xx_n), \dots, e_{n-k+1}(\xx_n) \rangle.
\end{equation}
Let $R_{n,k}$ be the corresponding quotient ring:
\begin{equation}
R_{n,k} := \frac{\QQ[\xx_n]}{I_{n,k}}.
\end{equation}
\end{defn}

Since the ideal $I_{n,k}$ is homogeneous, the quotient $R_{n,k}$ is a graded vector space.
Moreover, since $I_{n,k}$ is stable under the action of $\symm_n$, the algebra $R_{n,k}$ 
carries a graded action of $\symm_n$.
\footnote{It can also be shown that $e_n(\xx_n)$ lies in the ideal generated by
$x_1^k, \dots, x_n^k, e_{n-1}(\xx_n), \dots, e_{n-k+1}(\xx_n)$, so that its presence as a generator of $I_{n,k}$
is redundant.  We include $e_n(\xx_n)$ as a generator of $I_{n,k}$ because it will {\em not} be redundant 
as a generator in a more general family of ideals $I_{n,k,s}$ introduced in Section~\ref{Frobenius series}.}

When $k = n$,
it can be shown
\footnote{By \cite[Sec. 7.2]{Bergeron} we have $x_n^n \in I_n$ and $I_n$ is $\symm_n$-stable.}
 that for any $1 \leq i \leq n$, the variable power $x_i^n$ lies in the invariant ideal
$I_n$, 
so that $I_{n,n} = I_n$ and
$R_{n,n} = R_n$.  

At the other extreme, when $k = 1$ we have $I_{n,1} = \langle x_1, x_2, \dots, x_n \rangle$, so that 
$R_{n,1} = \frac{\QQ[\xx_n]}{\langle x_1, \dots, x_n \rangle} \cong \QQ$ is the trivial $\symm_n$-module in degree $0$.

We will prove that the modules $R_{n,k}$ extend many of the nice properties of the modules $R_n$,
where one replaces permutations in $\symm_n$ with ordered set partitions in $\OP_{n,k}$.
If $\xx = (x_1, x_2, \dots )$ is an infinite set of variables and $\ZZ[[\xx]]$ is the ring of formal power series in $\xx$
with integer coefficients, let $\rev_q: \ZZ[[\xx]][q] \rightarrow \ZZ[[\xx]][q]$ be the operator which 
reverses polynomials with respect to the variable $q$.  
Explicitly, if $f = a_d q^d + \cdots + a_1 q + a_0$ with $a_i \in \ZZ[[\xx]]$ and $a_d \neq 0$, then
$\rev_q(f) = a_0 q^d + \cdots + a_{d-1} q + a_d$.
For example, we have
\begin{equation*}
\rev_q(q^5 + 7q^3 - 8q) =  -8q^4 + 7q^2 + 1
\end{equation*}
and 
\begin{equation*}
\rev_q(e_3(\xx) q^2 + 2 e_2(\xx) q - 3) = -3q^2 + 2 e_2(\xx) q + e_3(\xx).
\end{equation*}

\begin{itemize}
\item  We have $\dim(R_{n,k}) = | \OP_{n,k} | = k! \cdot \Stir(n,k)$.  The Hilbert polynomial
$\Hilb(R_{n,k}; q)$ is given by
\begin{equation}
\Hilb(R_{n,k}; q) = \rev_q \left( [k]!_q \cdot  \Stir_q(n,k) \right),
\end{equation}
where $\Stir_q(n,k)$ is the $q$-Stirling number  (see Theorem~\ref{hilbert-series-theorem}).  
There is a generalization $\AAA_{n,k}$
of Artin's basis of $R_n$ to $R_{n,k}$ which witnesses
this identity (see Theorem~\ref{artin-basis-theorem}).
\item  As {\em ungraded} $\symm_n$-modules we have
\begin{equation}
R_{n,k} \cong_{\symm_n} \QQ[\OP_{n,k}],
\end{equation}
so that $R_{n,k}$ gives a graded version of the action of $\symm_n$ on ordered set partitions
(see Theorem~\ref{ungraded-frobenius-theorem}).
\item  The {\em graded} isomorphism type of $R_{n,k}$ can be described in terms of standard
Young tableaux (see Corollary~\ref{schur-expansion}).  
\begin{equation}
\grFrob(R_{n,k};q) = \sum_{T \in \SYT(n)}
q^{\maj(T)} {n - \des(T) - 1 \brack n-k}_q s_{\shape(T)}(\xx).
\end{equation}
In terms of the 
dual Hall-Littlewood basis  we have (see Theorem~\ref{hall-littlewood-expansion})
\begin{equation}
\grFrob(R_{n,k};q) = \rev_q \left[ \sum_{\substack{\lambda \vdash n \\ \ell(\lambda) = k}}
q^{\sum(i-1)(\lambda_i-1)} {k \brack m_1(\lambda), \dots , m_n(\lambda)}_q  Q'_{\lambda}(\xx;q) \right].
\end{equation}
We also have a combinatorial description of $\grFrob(R_{n,k};q)$ using  
Mahonian statistics on ordered multiset partitions.
\end{itemize}

Our results have connections to the theory of Macdonald polynomials.  
The {\em Delta Conjecture} of Haglund, Remmel, and Wilson \cite{HRW} is a generalization 
of the {\em Shuffle Conjecture} of Haglund, Haiman, Loehr, Remmel, and Ulyanov \cite{HHLRU}.
The Delta Conjecture asserts an equality of three quasisymmetric functions; 
see Section~\ref{Background} for details.
Two of these quasisymmetric functions are defined combinatorially and denoted $\Rise_{n,k}(\xx;q,t)$
and $\Val_{n,k}(\xx;q,t)$.

Using statistics on ordered multiset partitions, it follows from work of Wilson and Rhoades \cite{WMultiset, Rhoades} 
that
\begin{equation}
\Rise_{n,k}(\xx;q,0) = \Rise_{n,k}(\xx;0,q) = \Val_{n,k}(\xx;q,0) = \Val_{n,k}(\xx;0,q).
\end{equation}
Let $C_{n,k}(\xx;q)$ denote this common polynomial, which is known \cite{WMultiset, Rhoades} to be symmetric. 
We deform $C_{n,k}(\xx;q)$ somewhat by setting
\begin{equation}
D_{n,k}(\xx;q) := (\rev_q \circ \omega) C_{n,k}(\xx;q),
\end{equation}
where $\omega$ is the usual involution on symmetric functions sending $s_{\lambda}(\xx)$ to 
$s_{\lambda'}(\xx)$.
 We will prove that
\begin{equation}
\grFrob(R_{n,k}; q) = D_{n,k}(\xx;q).
\end{equation}
In other words, the module $R_{n,k}$ has graded Frobenius image equal to
 either of the combinatorial expressions in the Delta Conjecture at $t = 0$ (up to a twist).

 It is an open problem to determine a bigraded $\symm_n$-module which (even conjecturally) 
 has Frobenius image equal to any of the expressions in the Delta Conjecture.
 In the case $k = n$, the module of diagonal coinvariants plays this role.  
 Our result is the first theorem in this direction for general $k \leq n$.

Although our new module $R_{n,k}$ generalizes many of the nice combinatorial properties of 
the classical coinvariant module $R_n$, the proofs of these properties will be substantially different.
To see why, recall that a sequence of polynomials $f_1, \dots, f_n$ is a 
{\em regular sequence} in the  ring $\QQ[\xx_n]$ 
if $f_i$ is not a zero divisor in $\frac{\QQ[\xx_n]}{\langle f_1, \dots, f_{i-1}\rangle}$ for all $i$.
The regularity of $f_1, \dots, f_n$ immediately implies  the exactness
of 
\begin{equation*}
0 \rightarrow \frac{\QQ[\xx_n]}{\langle f_1, \dots, f_{i-1} \rangle} \xrightarrow{ \cdot f_i}
\frac{\QQ[\xx_n]}{\langle f_1, \dots, f_{i-1} \rangle} \twoheadrightarrow
 \frac{\QQ[\xx_n]}{\langle f_1, \dots, f_i \rangle} \rightarrow 0
\end{equation*}
for all $i$, and in turn (if the $f_i$ are homogeneous)
\begin{equation*}
\Hilb\left( \frac{\QQ[\xx_n]}{\langle f_1, \dots, f_n \rangle}; q \right) =
[\deg(f_1)]_q \cdots [\deg(f_n)]_q.
\end{equation*}
In particular, we have 
$\dim \left(\frac{\QQ[\xx_n]}{\langle f_1, \dots, f_n \rangle} \right) = \deg(f_1) \cdots \deg(f_n)$.
Moreover, the {\em Koszul complex}, a certain free resolution of the 
$\QQ[\xx_n]$-module
$\frac{\QQ[\xx_n]}{\langle f_1, \dots, f_n \rangle}$, is guaranteed to be exact in this case.

It can be shown that
the generators $e_1(\xx_n), e_2(\xx_n), \dots, e_n(\xx_n)$ of 
$I_n$ form a regular sequence.
The regularity of this sequence gives many of the properties of $R_n$ for free.  It is immediate that
$\dim(R_n) = n!$ and $\Hilb(R_n;q) = [n]!_q$.  Since the maps involved in the Koszul complex
commute with the action of $\symm_n$, it is readily derived that 
\begin{equation*}
\grFrob(R_n; q) = Q'_{(1^n)}(\xx;q).
\end{equation*}

This commutative algebra machinery breaks down in our setting.  The
ideal $I_{n,k}$ cannot be
generated by a regular sequence in $\QQ[\xx_n]$.  Indeed, the lack of a nice product
formula for the Stirling number $\Stir(n,k)$ makes it impossible to find a homogeneous regular
sequence $f_1, \dots, f_n$ in $\QQ[\xx_n]$ such that 
$\dim \left(\frac{\QQ[\xx_n]}{\langle f_1, \dots, f_n \rangle} \right) = \deg(f_1) \cdots \deg(f_n)
= |\OP_{n,k}|$.
Trying to modify the above program to determine the Hilbert or Frobenius image of $R_{n,k}$
is therefore hopeless.

To obtain the Hilbert series of $R_{n,k}$, we  use the theory of
Gr\"obner bases.  For a given monomial order $<$, any ideal $I \subseteq \QQ[x_1, \dots x_n]$
has a unique reduced Gr\"obner basis.  While the polynomials in this basis can have
unpredictable monomials and ugly coefficients, even for nicely presented ideals $I$, 
in our context a miracle occurs.  If we take $<$ to be the lexicographic
 term ordering, the reduced Gr\"obner basis for $I_{n,k}$ consists 
of the variable powers $x_1^k, x_2^k, \dots, x_n^k$ together with certain
(predictable) {\em Demazure characters} 
$\kappa_{\gamma}(x_n, x_{n-1}, \dots, x_1) \in \QQ[\xx_n]$
in a reversed variable set (see Theorem~\ref{reduced-groebner-basis-theorem}).
The polynomials $\kappa_{\gamma}$ are  characters of  indecomposable polynomial representations
of the Borel subgroup $B \subseteq GL_n(\CC)$ of upper triangular matrices; their appearance 
as Gr\"obner basis  elements of $I_{n,k}$ is mysterious to the authors.

Our Gr\"obner basis for $I_{n,k}$ generalizes known results on the reduced Gr\"obner basis of the classical invariant
ideal $I_n$.
In particular (see, for example, Sturmfels \cite[Thm. 1.2.7]{Sturmfels} or
Bergeron \cite[Sec. 7.2]{Bergeron}),
the reduced Gr\"obner basis for $I_n$ with respect to lexicographical order is 
\begin{equation*}
\{ h_i(x_i, x_{i+1}, \dots , x_n) \,:\, 1 \leq i \leq n \},
\end{equation*}
where $h_i$ is the  homogeneous symmetric function of degree $i$.
Moreover, the leading term of $h_i(x_i, x_{i+1}, \dots, x_n)$ is $x_i^i$, so that the initial ideal
of $I_n$ is generated by $x_1, x_2^2, \dots, x_n^n$.
Fomin, Gelfand, and Postnikov \cite[Prop. 12.1]{FGP} obtain a $q$-analog of this result 
for the quantum cohomology of the complete flag variety.
The Demazure characters $\kappa_{\gamma}(x_n, x_{n-1}, \dots, x_1) \in \QQ[\xx_n]$ in our 
Gr\"obner basis will reduce to the polynomials
$h_i(x_i, x_{i+1}, \dots, x_n)$ when $k = n$.

The Gr\"obner basis described above will allow us to derive our generalization of the Artin basis
of $R_n$.  
A {\em shuffle} of two sequences $(a_1, \dots, a_r)$ and $(b_1, \dots, b_s)$ is an interleaving
$(c_1, \dots, c_{r+s})$ of these two sequences which preserves the relative orders of the $a$'s and the $b$'s.
An {\em $(n,k)$-staircase} is a shuffle of the two sequences
$(0, 1, \dots, k-1)$ and $(k-1, k-1, \dots, k-1)$, where there
are $n-k$ copies of $k-1$.  For example, the $(5,3)$-staircases are
the shuffles of $(0,1,2)$ and $(2,2)$:
\begin{equation*}
(0,1,2,2,2), (0,2,1,2,2), (0,2,2,1,2), (2,0,1,2,2), (2,0,2,1,2), \text{ and } (2,2,0,1,2).
\end{equation*}
We will prove that the set of monomials
\begin{equation*}
\AAA_{n,k} = 
\{ x_1^{a_1} \cdots x_n^{a_n} \,:\, \text{$(a_1, \dots, a_n)$ is componentwise $\leq$ some $(n,k)$-staircase} \}
\end{equation*}
descends to a basis of $R_{n,k}$.
Since the only $(n,n)$-staircase is $(0,1,\dots,n-1)$, we get $\AAA_{n,n} = \AAA_n$.
Although it is not obvious at this point, the number of monomials in $\AAA_{n,k}$ is $|\OP_{n,k}|$.

The careful study of the basis $\AAA_{n,k}$ will give us our expression for the Hilbert series
of $R_{n,k}$.   We will also derive a generalization $\GS_{n,k}$ of the {\em Garsia-Stanton basis}
of $R_n$ whose combinatorics is governed by a major index-like statistic on ordered 
set partitions \cite{Garsia, GS}.
This will also turn out to give the {\em ungraded} module isomorphism
$R_{n,k} \cong \QQ[\OP_{n,k}]$.

To determine the {\em graded} Frobenius image of $R_{n,k}$, we use a recursive method 
dating back to the work of Garsia and Procesi on the graded isomorphism type of the cohomology
of Springer fibers \cite{GP}.
In particular, we apply the fact that two symmetric functions with equal constant terms
are equal if and only if their images 
under the operator $e_j(\xx)^{\perp}$ coincide for all $j \geq 1$.
On the algebraic side, this will involve a generalization $R_{n,k,s}$ of the rings
$R_{n,k}$ which satisfy $R_{n,k,k} = R_{n,k}$.

The rest of the paper is organized as follows. 
 In {\bf Section~\ref{Background}} we will present definitions related to ordered set partitions, symmetric functions,
  Demazure characters, and Gr\"obner bases.
  In {\bf Section~\ref{Some polynomial identities}}
  we will prove a variety of identities involving  polynomials and symmetric functions which will be crucial
  in the analysis of $R_{n,k}$ in the following sections.
  In {\bf Section~\ref{Hilbert series}} we will prove our formula for $\Hilb(R_{n,k};q)$ and give a
  generalization of the Artin basis to $R_{n,k}$.
  We will also describe the reduced Gr\"obner basis of $I_{n,k}$ with respect to the lexicographic monomial ordering.
  In {\bf Section~\ref{Garsia-Stanton}} we will give our extension $\GS_{n,k}$ of the Garsia-Stanton basis
  to $R_{n,k}$.
  In {\bf Section~\ref{Frobenius series}} we will derive the graded Frobenius image of $R_{n,k}$.
  We make concluding remarks in {\bf Section~\ref{Conclusion}}.

\section{Background}
\label{Background}

\subsection{Ordered Set Partitions}
Let $\pi = \pi_1 \dots \pi_n \in \symm_n$ be a permutation written in one-line notation.  The {\em descent set}
$\Des(\pi)$ and  {\em ascent set} $\Asc(\pi)$ are given by
\begin{center}
$\begin{array}{cc}
\Des(\pi) := \{1 \leq i \leq n-1 \,:\, \pi_i > \pi_{i+1}\}, &
\Asc(\pi) := \{1 \leq i \leq n-1 \,:\, \pi_i < \pi_{i+1}\}.
\end{array}$
\end{center}
We let $\des(\pi) := |\Des(\pi)|$ and $\asc(\pi) := |\Asc(\pi)|$ be the number of descents or ascents of $\pi$.
The {\em major index} of $\pi$ is $\maj(\pi) := \sum_{i \in \Des(\pi)} i$.
An {\em inversion} in $\pi$ is a pair $i < j$ with $\pi_i > \pi_j$; we let $\inv(\pi)$ be the number of inversions
in $\pi$.

An {\em ordered set partition} of size $n$ is a set partition of $[n]$ equipped with a total order on
its blocks.  For example,
\begin{equation*}
\sigma = \{2, 4\} \prec \{6\} \prec \{1, 3, 5\}
\end{equation*}
is an ordered set partition of size $6$ with $3$ blocks.  

We will write ordered set partitions in two ways.
The first
 denotes separation between
blocks with bars and writes letters in an increasing fashion within blocks, so that the above ordered
set partition may be written more succinctly as
\begin{equation*}
\sigma = (24 \mid 6 \mid 135).
\end{equation*}

We will sometimes use stars to indicate connectives relating elements in an ordered set partition in such
a way that letters are increasing within starred segments.  Our example ordered set partition can
then be expressed
\begin{equation*}
\sigma = 2_* 4 \,\, 6 \,\, 1_*3_*5.
\end{equation*}
An {\em ascent starred permutation} is a pair $(\pi, S)$ where 
$\pi \in \symm_n$ and $S \subseteq \Asc(\pi)$.  Our star notation gives an identification
\begin{equation*}
\OP_{n,k} = \{ \text{ascent starred permutations $(\pi, S)$} \,:\, \pi \in \symm_n \text{ and } |S| = n-k \}.
\end{equation*}
Our example ordered set partition becomes
\begin{equation*}
\sigma = (246135, \{1,4,5\}).
\end{equation*}

Let $\OP_{n,k}$ denote the collection of ordered set partitions of size $n$ with $k$ blocks.
We have $|\OP_{n,k}| = k! \cdot \Stir(n,k)$, where $\Stir(n,k)$ is the (signless) Stirling number
of the second kind counting $k$-block set partitions of $[n]$.

For $\sigma = (B_1 \mid \cdots \mid B_k) \in \OP_{n,k}$, an {\em inversion} in $\sigma$ 
is a pair of letters $i < j$ such that $i$ is minimal in $B_m$, $j \in B_{\ell}$, and $\ell < m$.
Let $\inv(\sigma)$ denote the number of inversions of $\sigma$; this is the usual inversion
statistic when $k = n$ and $\sigma$ is a permutation.
For example, the inversions in $(24 \mid 6 \mid 135)$ are the pairs 
$(1,2), (1,4), (1,6)$ and $\inv(24 \mid 6 \mid 135) = 3$.

It will be convenient to consider the statistic on ordered set partitions which is complementary
to the inversion statistic.  Given $\sigma \in \OP_{n,k}$, define
\begin{align}
\coinv(\sigma) &:= \max \{ \inv(\sigma') \,:\, \sigma' \in \OP_{n,k} \} - \inv(\sigma) \\
&= (n-k)(k-1) + {k \choose 2} - \inv(\sigma),
\end{align}
where the second equality follows from $\inv$ being (uniquely) maximized on
\begin{equation*}
(k (k+1) \cdots n \mid k-1 \mid \cdots \mid 2 \mid 1) \in \OP_{n,k}.
\end{equation*}

In a private communication to the authors, A. Wilson noted that for any $\sigma \in \OP_{n,k}$, $\coinv (\sigma)$ is the
numbers of pairs $(a,b)$, $1\le a<b \le n$, such that:
\begin{equation}
\label{coinvdef}
\begin{cases}
\text{at least one of $a$ and $b$ is minimal in its block in $\sigma$,} \\
\text{$a$ and $b$ are in different blocks, and} \\
\text{if $a$'s block is to the right of $b$'s block, then
 only $b$ is minimal in its block.}
 \end{cases}
\end{equation}
For example, in the ordered set partition $(45\mid 167\mid 23)$ the pairs that satisfy the above condition are
$12$, $13$, $34$, $46$, and $47$.
We leave it as an exercise for the interested reader to verify that this description of $\coinv$ is equivalent to 
our definition $\coinv (\sigma) = (n-k)(k-1)+{k \choose 2} - \inv (\sigma)$.

Given $\sigma = (\pi,S) \in \OP_{n,k}$ represented as an ascent-starred permutation,
we define the {\em major index} $\maj(\sigma)$ as follows.  For $1 \leq i \leq n$, let 
$i^c := (n-i+1)$ and let $\pi^c \in \symm_n$ be the permutation 
$\pi^c = \pi_1^c \dots \pi_n^c$.  Define
\begin{equation}
\maj(\sigma) := \maj(\pi^c) - \sum_{i \in S} |\Asc(\pi) \cap \{i, i+1, \dots , n-1\}|.
\end{equation}
An equivalent version of this major index statistic appears in \cite[p. 12]{RW}.
For example, we have
\begin{equation*}
\maj(2_* 4 \,\, 6 \,\, 1_*3_*5) = \maj(531642) - (1+2+4) = (1+2+4+5)-(1+2+4) = 5.
\end{equation*}

Just as in the case of $\inv$, we will consider the complementary statistic to $\maj$ on 
$\OP_{n,k}$.  Given $\sigma \in \OP_{n,k}$, define
\begin{align}
\comaj(\sigma) &:= \max \{ \maj(\sigma') \,:\, \sigma' \in \OP_{n,k} \} - \maj(\sigma) \\
&= (n-k)(k-1) + {k \choose 2} - \maj(\sigma),
\end{align}
where the second equality comes from the fact that $\maj$ is (uniquely) maximized 
on
\begin{equation*}
1 \,\, 2 \,\, \dots \,\, (k-1) \,\, k_*(k+1)_* \dots  (n-1)_*n \in \OP_{n,k}.
\end{equation*}

The {\em $q$-Stirling numbers} $\Stir_q(n,k)$ are defined by the recursion
\begin{equation}
\Stir_q(n,k) = \Stir_q(n-1,k-1) + [k]_q \cdot \Stir_q(n-1,k)
\end{equation}
and the initial condition 
$\Stir_q(1,k) = \begin{cases} 1 & k = 1 \\ 0 & k > 1. \end{cases}$
Steingr\'imsson \cite{Stein} and Remmel-Wilson \cite{RW} proved
that the product
$[k]!_q \cdot \Stir_q(n,k)$ is the generating function of $\inv$ and $\maj$ on $\OP_{n,k}$:
\begin{equation}
\sum_{\sigma \in \OP_{n,k}} q^{\inv(\sigma)} = \sum_{\sigma \in \OP_{n,k}} q^{\maj(\sigma)}
= [k]!_q \cdot \Stir_q(n,k).
\end{equation}
Any statistic on $\OP_{n,k}$ which shares this distribution is called {\em Mahonian}.
Reversing in $q$, we get
\begin{equation}
\sum_{\sigma \in \OP_{n,k}} q^{\coinv(\sigma)} = \sum_{\sigma \in \OP_{n,k}} q^{\comaj(\sigma)}
= \rev_q([k]!_q \cdot \Stir_q(n,k)).
\end{equation}

Recall that a {\em (weak) composition} $\gamma = (\gamma_1, \dots, \gamma_n)$ is a sequence 
of nonnegative integers.  We write $|\gamma| := \gamma_1 + \cdots + \gamma_n$,
write $\gamma \models |\gamma|$, and say that 
$\gamma$ has $n$ parts, or $\ell(\gamma) = n$.
We let $\gamma^* = (\gamma_n, \dots, \gamma_1)$ be the reverse of $\gamma$.

An {\em ordered multiset partition} is  a finite sequence $\mu = (M_1 \mid \cdots \mid M_k)$ of nonempty
finite
sets of positive integers.  We say that $\mu$ has {\em size} $|M_1| + \cdots + |M_k|$ and {\em $k$ blocks}.
For example, 
\begin{equation*}
\mu = (247 \mid 15 \mid 145)
\end{equation*}
is an ordered multiset partition of size $8$ with $3$ blocks.  Note that the elements in an ordered multiset 
partition are {\em sets}; we do not allow repeated letters within blocks.  

The {\em content} of an ordered 
multiset partition $\mu$ is the composition $\cont(\mu) = (\cont(\mu)_1, \cont(\mu)_2, \dots)$, where
$\cont(\mu)_i$ the multiplicity of $i$ as a letter in $\mu$.  If $\mu$ is the ordered multiset partition
above, we have $\cont(\mu) = (2,1,0,2,2,0,1)$.  For any composition $\gamma$, let
$\OP_{\gamma,k}$ be the collection of ordered multiset partitions of content $\gamma$ with $k$ blocks.
When $\gamma = (1^n)$, we recover the notion of an ordered set partition.

The definition of
 $\inv$ on ordered set partitions extends verbatim to ordered multiset partitions.
 There is also an extension of the $\maj$ statistic to ordered multiset partitions
 (which are viewed in this context as {\em descent} starred words); see \cite{RW, WMultiset}.
 Remmel-Wilson and Wilson defined two other statistics on ordered multiset
 partitions called $\dinv$ and $\minimaj$.  
 We will not use the statistics $\maj, \dinv,$ and $\minimaj$
 on ordered multiset partitions explicitly in our work; see \cite{RW, WMultiset} for their definitions.

\subsection{Symmetric functions}  
Our notation for symmetric functions is standard; see  \cite{Macdonald}.

A {\em partition} $\lambda$ of $n$ is a weakly decreasing sequence 
$\lambda = (\lambda_1 \geq \cdots \geq \lambda_k)$ of positive integers with
$\sum \lambda_i = n$.  We write $\lambda \vdash n$ to mean that $\lambda$ is a partition
of $n$ and $\ell(\lambda) = k$ for the number of parts of $\lambda$.  
We  denote the multiplicity of $i$ as a part of $\lambda$ by $m_i(\lambda)$.
Given two partitions $\lambda, \mu$ we say $\lambda \leq \mu$ in {\em dominance order}
if $\lambda_1 + \cdots  + \lambda_i \leq \mu_1 + \cdots + \mu_i$ for all $i$.

Given $\lambda \vdash n$, the (English) {\em Ferrers diagram} of $\lambda$ 
consists of $\lambda_i$ left-justified boxes in row $i$, for all $i$.  The Ferrers diagram of 
$(4,2,2) \vdash 8$ is shown below.
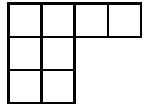
\begin{small}
\begin{center}
\begin{Young}
 & & & \cr
  & \cr
   & \cr 
\end{Young}
\end{center}
\end{small}
The {\em transpose} $\lambda'$ of $\lambda$ is the partition of $n$ whose Ferrers diagram
is obtained from the Ferrers diagram of $\lambda$ by reflecting across the line
$y = x$.  For example, we have $(4,2,2)' = (3,3,1,1)$.

For $\lambda \vdash n$, a {\em standard Young tableau} of shape $\lambda$ is a filling of the 
boxes of $\lambda$ with $1, 2, \dots, n$ which is increasing down columns and across rows.  
One possible standard Young tableau of shape $(4,2,2)$ is
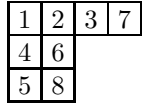
\begin{small}
\begin{center}
\begin{Young}
1 & 2 & 3 & 7 \cr
4  & 6 \cr
 5  & 8 \cr 
\end{Young}
\end{center}
\end{small}

Let $\SYT(n)$ denote the set of standard Young tableaux with $n$ boxes.  For $T \in \SYT(n)$,
let $\shape(T) \vdash n$ denote the shape of $T$.  Given $T \in \SYT(n)$, a {\em descent}
of $T$ is a letter $i$ which appears in a higher row than $i+1$ in $T$.  
Let $\des(T)$ denote the number of descents of $T$ and $\maj(T)$ denote the sum of the descents
of $T$.
For example, the descents
in the tableau $T$ above are $3,4,$ and $7$, so that $\des(T) = 3$ and 
$\maj(T) = 3 + 4 + 7 = 14$.

Let $\xx = (x_1, x_2, \dots )$ be an infinite set of variables
and let $\Lambda$ denote the ring of symmetric functions in $\xx$ with coefficients in
the field $\QQ(q,t)$.  
The ring $\Lambda = \bigoplus_{n \geq 0} \Lambda_n$ is graded by degree and the dimension
of the graded piece $\Lambda_n$ equals the number of partitions of $n$.

For $\lambda \vdash n$, let 
\begin{center}
$\begin{array}{ccccc}
m_{\lambda}(\xx), &
e_{\lambda}(\xx), & h_{\lambda}(\xx), & s_{\lambda}(\xx), &  \widetilde{H}_{\lambda}(\xx;q,t)
\end{array}$
\end{center}
be the corresponding {\em monomial},
{\em elementary},
{\em (complete) homogeneous},
{\em Schur},
and {\em modified Macdonald symmetric functions}.
Here we are using $\xx$ to refer to an infinite set of variables to avoid confusion with the 
finite variable set $\xx_n$.  
Let 
$\omega$ be the algebra involution on $\Lambda$ defined by $\omega(s_{\lambda}(\xx)) = s_{\lambda'}(\xx)$.

The {\em Hall-Littlewood polynomials} $P_{\lambda}(\xx; q)$ are 
the basis of $\Lambda$ defined implicitly in terms of the Schur basis by the formula
\begin{equation*}
s_{\lambda}(\xx) = \sum_{\mu} K_{\lambda, \mu}(q) P_{\mu}(\xx;q),
\end{equation*}
where $K_{\lambda,\mu}(q)$ is the {\em Kostka-Foulkes polynomial} (i.e., the generating function of charge
on semistandard tableaux of shape $\lambda$ and content $\mu$).  The {\em dual Hall-Littlewood polynomials}
$Q'_{\lambda}(\xx; q)$ are defined by
\begin{equation*}
Q'_{\mu}(\xx;q) = \sum_{\mu} K_{\lambda,\mu}(q) s_{\lambda}(\xx).
\end{equation*}
In particular, the polynomials $Q'_{\mu}(\xx;q)$ are Schur positive.

Let $\langle \cdot, \cdot \rangle$ be the {\em Hall inner product} on $\Lambda$ with respect to which the 
Schur basis is orthogonal.  
For any symmetric function $F(\xx) \in \Lambda$, let 
$F(\xx)^{\perp}$ denote the linear
 operator on $\Lambda$ which is adjoint to multiplication by $F(\xx)$ with 
respect to the Hall inner product.
In other words, we have
\begin{equation}
\langle F(\xx)^{\perp} G(\xx), H(\xx) \rangle = \langle G(\xx), F(\xx)H(\xx) \rangle
\end{equation}
for all symmetric functions $G(\xx), H(\xx) \in \Lambda$.

The irreducible representations of the symmetric group $\symm_n$ are naturally labeled by 
partitions of $n$.  Given a partition $\lambda \vdash n$, 
let $S^{\lambda}$ be the corresponding irreducible
representation of $\symm_n$.

Let $V$ be a finite-dimensional $\symm_n$-module. Then $V$ is completely
decomposable and  we may write 
\begin{equation*}
V \cong \bigoplus_{\lambda \vdash n} (S^{\lambda})^{\oplus c_{\lambda}}
\end{equation*}
as a direct sum of irreducible representations for some nonnegative integers $c_{\lambda}$.  
The {\em Frobenius character}
$\Frob(V) \in \Lambda_n$ is the symmetric function
\begin{equation}
\Frob(V) := \sum_{\lambda \vdash n} c_{\lambda} s_{\lambda}(\xx).
\end{equation}

For example, let $\lambda \vdash n$ and consider the {\em Young subgroup}
\begin{equation*}
\symm_{\lambda} := \symm_{\lambda_1} \times \symm_{\lambda_2} \times \cdots
\end{equation*}
of the symmetric group $\symm_n$.  The corresponding left coset representation 
$\Ind_{\symm_{\lambda}}^{\symm_n}(\triv) = \QQ[\symm_n/\symm_{\lambda}]$ of $\symm_n$ 
has Frobenius image
$\Frob:  \QQ[\symm_n/\symm_{\lambda}]  \mapsto h_{\lambda}(\xx)$.

Let $V = \bigoplus_{d \geq 0} V_d$ be a graded vector space in which each graded piece
$V_d$ is finite-dimensional.  The {\em Hilbert series} of $V$ is the power series in $q$ given by
\begin{equation}
\Hilb(V;q) := \sum_{d \geq 0} \dim(V_d) q^d.
\end{equation}
If $V$ carries a graded action of $\symm_n$, the {\em graded Frobenius character} is 
\begin{equation}
\grFrob(V;q) := \sum_{d \geq 0} \Frob(V_d) q^d.
\end{equation}

\subsection{The Delta Conjecture}

Our results are related to the {\em Delta Conjecture} arising in the theory of Macdonald polynomials.
Let us briefly review this conjecture of Haglund, Remmel, and Wilson \cite{HRW}.
To state the Delta Conjecture, we will need some definitions.

For any symmetric function $F(\xx)$, we define the linear transformation
\begin{equation}
\Delta'_F: \Lambda_n \rightarrow \Lambda_n
\end{equation}
to be the Macdonald eigenoperator 
$\Delta'_F: \widetilde{H}_{\mu}(\xx;q,t) \mapsto F[B_{\mu}(q,t) - 1] \cdot \widetilde{H}_{\mu}(\xx;q,t)$.
Here we are using the plethystic notation
\begin{equation}
F[B_{\mu}(q,t) - 1] = F( \dots , q^{(i-1)} t^{(j-1)} , \dots ),
\end{equation}
where $(i,j)$ range over all matrix coordinates $\neq (1,1)$ of cells in the 
Ferrers diagram of $\mu$ (and all other variables in $F(\xx)$ are set to $0$).

For example, if $\mu = (3,2) \vdash 5$, we have
\begin{equation*}
\Delta'_F: \widetilde{H}_{\mu}(\xx;q,t) \mapsto F(q,q^2,t,qt) \widetilde{H}_{\mu}(\xx;q,t),
\end{equation*}
corresponding to the filling
\begin{center}
\begin{Young}
$\cdot$ & $q$ & $q^2$ \cr
$t$ & $qt$ 
\end{Young}.
\end{center}

A {\em Dyck path} of size $n$ is a lattice path from $(0,0)$ to $(n,n)$ consisting of north and east steps
which stays weakly above $y = x$.    We number the rows of any Dyck path $D$ as $1, 2, \dots, n$ from bottom 
to top.
A {\em labeled Dyck path} is a Dyck path with (not necessarily unique) positive integers assigned to its north steps
such that these labels strictly increase going up columns.
We let $\DDD_n$ denote the set of Dyck paths of size $n$ and 
$\LD_n$ denote the collection of labeled Dyck paths of size $n$.
Given $P \in \LD_n$, let $\xx^P \in \QQ[[\xx]]$ be the monomial
in which the power of $x_i$ is the multiplicity of $i$ as a label of $P$.  Also let $\ell_i(P)$ denote the label
of $P$ in row $i$.

For $D \in \DDD_n$ and $1 \leq i \leq n$, let $a_i(D)$ be the number of full squares between $D$ and the 
line $y = x$ in row $i$.  Let $\area(D) := a_1(D) + \cdots + a_n(D)$ be the {\em area} of $D$.  If $P \in \LD_n$
is a labeled Dyck path and $D(P) \in \DDD_n$ is its underlying Dyck path, we set
$a_i(P) := a_i(D(P))$ and $\area(P) := \area(D(P))$.

Let $P \in \LD_n$ be a labeled Dyck path and $1 \leq i \leq n$.  Define $d_i(P)$ by 
\begin{align*}
d_i(P) := &|\{i < j \leq n \,:\, a_i(P) = a_j(P), \ell_i(P) < \ell_j(P)\}|  \\
&+ |\{i < j \leq n \,:\, a_i(P) = a_j(P) + 1, \ell_i(P) > \ell_j(P)\}|.
\end{align*}
The dinv statistic is  $\dinv(P) := d_i(P) + \cdots + d_n(P)$.
The {\em contractible valleys} of $P$ are
\begin{align*}
\Val(P) := &\{2 \leq i \leq n \,:\, a_i(P) < a_{i-1}(P)\} \\ &\cup \{2 \leq i \leq n \,:\, a_i(P) = a_{i-1}(P), \ell_i(P) > \ell_{i-1}(P)\}.
\end{align*}

As an example of these concepts, let $P \in \LD_5$ be the labeled Dyck path shown below.
We have $\area(P) = 2, \dinv(P) = 4, \Val(P) = \{4,5\},$ and $\xx^P = x_1 x_2^2 x_3 x_6$.

\begin{center}
\begin{tikzpicture}[scale=0.6]
(8,0) rectangle +(5,5);
\draw[help lines] (8,0) grid +(5,5);
\draw[dashed] (8,0) -- +(5,5);
\coordinate (prev) at (8,0);
\draw [color=black, line width=2] (8,0)--(8,2)--(9,2)--(9,3)--(11,3)--(11,4)--(12,4)--(12,5)--(13,5);

\draw (8,0) node [scale=0.5, circle, draw,fill=black]{};
\draw (13,5) node [scale=0.5, circle, draw,fill=black]{};

\node[align=left,scale=1.4] at (8.5,0.5) {3};
\node[align=left,scale=1.4] at (8.5,1.5) {6};
\node[align=left,scale=1.4] at (9.5,2.5) {2};
\node[align=left,scale=1.4] at (11.5,3.5) {1};
\node[align=left,scale=1.4] at (12.5,4.5) {2};

\node[align=left,scale=1.4] at (14.5,5.5) {$i$};
\node[align=left,scale=1.4] at (14.5,4.5) {$5$};
\node[align=left,scale=1.4] at (14.5,3.5) {$4$};
\node[align=left,scale=1.4] at (14.5,2.5) {$3$};
\node[align=left,scale=1.4] at (14.5,1.5) {$2$};
\node[align=left,scale=1.4] at (14.5,0.5) {$1$};

\node[align=left,scale=1.4] at (16,5.4) {$a_i$};
\node[align=left,scale=1.4] at (16,4.5) {$0$};
\node[align=left,scale=1.4] at (16,3.5) {$0$};
\node[align=left,scale=1.4] at (16,2.5) {$1$};
\node[align=left,scale=1.4] at (16,1.5) {$1$};
\node[align=left,scale=1.4] at (16,0.5) {$0$};

\node[align=left,scale=1.4] at (17.5,5.5) {$d_i$};
\node[align=left,scale=1.4] at (17.5,4.5) {$0$};
\node[align=left,scale=1.4] at (17.5,3.5) {$1$};
\node[align=left,scale=1.4] at (17.5,2.5) {$1$};
\node[align=left,scale=1.4] at (17.5,1.5) {$2$};
\node[align=left,scale=1.4] at (17.5,0.5) {$0$};
 \end{tikzpicture}
\end{center}

\begin{conjecture} \cite{HRW}
{\em (The Delta Conjecture)}
For positive integers $k \leq n$,
\begin{align}
\label{riseform}
\Delta'_{e_{k-1}} e_n(\xx) 
&= \{z^{n-k}\} \left[ \sum_{P \in \mathcal{LD}_n} q^{\mathrm{dinv}(P)} t^{\mathrm{area}(P)}
\prod_{i: a_i(P) > a_{i-1}(P)} \left( 1 + z/t^{a_i(P)} \right) \xx^P \right] \\
&= \{z^{n-k}\} \left[ \sum_{P \in \mathcal{LD}_n} q^{\mathrm{dinv}(P)} t^{\mathrm{area}(P)}
\prod_{i \in \mathrm{Val}(P)} \left( 1 + z/q^{d_i(P) + 1} \right) \xx^P \right].
\end{align}
Here the operator $\{z^{n-k}\}$ extracts the coefficient of $z^{n-k}$.
\end{conjecture}

When $k = n$, we have $\Delta'_{e_{n-1}} e_n(\xx) = \nabla e_n(\xx)$, 
where $\nabla$ is the Bergeron-Garsia
Macdonald eigenoperator, and the Delta Conjecture reduces to the Shuffle Conjecture of 
Haglund, Haiman, Loehr, Remmel, and Ulyanov \cite{HHLRU}.
Haiman proved that $\nabla e_n(\xx)$ is the bigraded Frobenius image of the diagonal coinvariant algebra \cite{Haiman}.
If we set $t = 0$, we get that $\nabla e_n(\xx) |_{t = 0}$ is the graded Frobenius image of the 
classical coinvariant algebra $R_n$:
$\grFrob(R_n; q) = \nabla e_n(\xx) |_{t = 0}.$

For general $k \leq n$,
there is not even a conjectural bigraded $\symm_n$-module whose bigraded Frobenius image 
equals any of the expressions in the Delta Conjecture.  We will prove that (up to minor modification),
our rings $R_{n,k}$ provide such a module in the specialization $t = 0$ of either combinatorial 
expression in the Delta Conjecture.

Let $\Rise_{n,k}(\xx;q,t)$ and $\Val_{n,k}(\xx;q,t)$ denote the middle and right
sides of the Delta Conjecture, respectively.
\footnote{Our conventions are `off by one' from those in \cite{HRW} and elsewhere -- our 
$\Rise_{n,k}(\xx;q,t)$ is their $\Rise_{n,k-1}(\xx;q,t)$, etc.}
By the work of Remmel-Wilson, Wilson, and Rhoades \cite{RW, WMultiset, Rhoades} we have 
\begin{equation*}
\Rise_{n,k}(\xx;q,0) = \Rise_{n,k}(\xx;0,q) = \Val_{n,k}(\xx;q,0) = \Val_{n,k}(\xx;0,q).
\end{equation*}
Let $C_{n,k}(\xx;q)$ denote this common symmetric function 
\footnote{While this paper was under review, Garsia, Haglund, Remmel, and Yoo \cite{GHRY}
proved that we have the additional equalities
$\Delta'_{e_{k-1}} e_n(\xx)|_{t = 0} = \Delta'_{e_{k-1}} e_n(\xx)|_{q = 0, t = q} = 
C_{n,k}(\xx;q)$, so we can also define $C_{n,k}(\xx;q)$ in terms of the operator
$\Delta'_{e_{k-1}}$.}
and set 
\begin{equation}
D_{n,k}(\xx;q) := \rev_q \circ \omega [C_{n,k}(\xx;q)].
\end{equation}

The symmetric function $C_{n,k}(\xx;q)$ is related to ordered multiset partitions as follows.  If $\gamma$ is 
any composition, let $\xx^{\gamma} = x_1^{\gamma_1} x_2^{\gamma_2} \cdots$.  
Haglund, Remmel, and Wilson proved the following formulas \cite[Prop. 4.1]{HRW}:
\begin{equation}
\label{fourways}
\begin{split}
\Rise_{n,k}(\xx;q,0) = \sum_{|\gamma|=n} \sum_{\mu \in \OP_{\gamma,k}} q^{\dinv(\mu)} \xx^{\gamma} \hspace{0.2in}
\Rise_{n,k}(\xx;0,q) = \sum_{|\gamma|=n} \sum_{\mu \in \OP_{\gamma,k}} q^{\maj(\mu)} \xx^{\gamma} \\ 
\Val_{n,k}(\xx;q,0) = \sum_{|\gamma|=n} \sum_{\mu \in \OP_{\gamma,k}} q^{\inv(\mu)} \xx^{\gamma} \hspace{0.2in}
\Val_{n,k}(\xx;0,q) = \sum_{|\gamma|=n} \sum_{\mu \in \OP_{\gamma,k}} q^{\minimaj(\mu)} \xx^{\gamma}.
\end{split}
\end{equation}

By the {\em standardization} 
of a word $w$ of $n$ positive integers, we mean the unique permutation $\pi \in \symm_n$ which satisfies
$w_i < w_j$ iff $\pi _i <\pi _j$.  For example, the standardization of $131$ is $132$.
For $\pi \in \symm_n$, let $\iDes (\pi)$ denote the descent set of $\pi ^{-1}$.

For $\mu$ an ordered  multiset partition, let  the {\em reading word}
$\text{rword}(\mu)$ be the word obtained 
from $\mu$ by reading along ``diagonals"
(the $m$th diagonal consists of all elements which are the $m$th largest in their block, left to right), larger $m$ first.
For example, if
$\mu = (247\mid 1\mid 35\mid 3)$, then $\text{rword}(\mu) = 7452133$.  
This choice of a reading word guarantees that if $\mu$ is an ordered multiset partition, and the standardization of 
$\text{rword}(\mu)$ equals  $\text{rword}(\sigma)$ for some $\sigma \in \OP_{n,k}$, then
a pair of elements in $\mu$ form an inversion pair iff the corresponding pair in $\sigma$ do as well.  
For example, $\text{rword}(247\mid 26\mid 6)=746226$,
which standardizes to the permutation $634125$, and this is the reading word of 
$(136 \mid  24 \mid 5)$.  The inversion pairs of $(136  \mid 24 \mid 5)$ are $(5,6),(2,6)$ and $(2,3)$,  while those of $(247\mid 26\mid 6)$  are 
the second $6$ and the $7$, the second $2$ and the $7$, and the second $2$ and the $4$.

Let $F_{n,D}(\xx)$ denote the Gessel fundamental quasisymmetric function corresponding to the
descent set $D \subseteq \{1,2,\ldots ,n-1\}$.  One way of defining $F_{n,D}(\xx)$ is the sum of the $\xx$-weights of all words
which standardize to a permutation $\pi$ with $\iDes (\pi)=D$.
See \cite[pp. 99-101]{Hagbook} for background on standardization and Gessel fundamental quasisymmetric functions (there denoted 
$Q_{n,D}(\xx)$).
Thus an equivalent way of writing the third expression for $C_{n,k}(\xx ,q)$ from (\ref{fourways}) is
\begin{align}
\label{ourway}
\Val_{n,k}(\xx;q,0) = \sum_{\sigma \in \OP_{n,k}} q^{\inv(\sigma)}  F_{n,\iDes (\text{rword}(\sigma))}(\xx).
\end{align}

Let $G(\xx)$ be a symmetric function expressable in the form
\begin{align}
\label{rev1}
G(\xx) = \sum_{\pi \in \symm_n} c(\pi) F_{n,\iDes (\pi)}(\xx),
\end{align}
where the sum is over the symmetric group and the $c(\pi)$ are independent of $\xx$. Then
\begin{align}
\label{reverse2}
\omega G(\xx) = \sum_{\pi \in \symm_n} c(\pi) F_{n,\iDes ( \text{revword}( \pi) )}(\xx),
\end{align}
where $\text{revword}(\sigma)$ is the word obtained by reversing $\text{rword}(\sigma)$.   This follows easily if $G$ is a Schur function, using the
well-known decomposition of a Schur function into Gessel fundamentals \cite{Stanley}, and hence for general $G$ since the Schur
functions form a basis for the ring of symmetric functions.  Combining this with (\ref{ourway}) gives
\begin{align}
\label{new}
D_{n,k}(\xx;q) =  
\sum_{\sigma \in \OP_{n,k}} q^{\coinv (\sigma)} F_{n,\iDes (\text{revword}(\sigma))}(\xx).
\end{align}

\subsection{Demazure characters} 
To any  composition 
$\gamma = (\gamma_1, \dots, \gamma_n)$
we have a {\em (augmented) skyline diagram} consisting of 
columns of heights $\gamma_1, \dots, \gamma_n$ augmented with the basement which 
reads, from left to right, $n, n-1, \dots, 1$.
For example, the skyline diagram of $(3,0,1,3)$ is shown below.
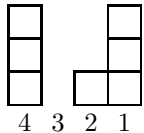
\begin{small}
\begin{center}
\begin{Young}
 &, & ,& \cr 
 &, & ,& \cr
 &, & & \cr
 ,4 & ,3 & ,2 & ,1 \cr
\end{Young}
\end{center}
\end{small}

Let $\gamma = (\gamma_1, \dots, \gamma_n)$ be a  composition and
let $1\leq i < j \leq n$ 
index a pair
of columns of $\gamma$.  A {\em type A triple} is a set of three cells $a, b, c$ of the form
$(i,k), (j,k), (i,k-1)$ of the skyline diagram such that $\gamma_i \geq \gamma_j$.  
A {\em type B triple} is a set of three cells $a, b, c$ of the form
$(j,k+1), (i,k), (j,k)$ where $\gamma_i < \gamma_j$.  
These two situations are shown schematically
below.  
(Note that basement cells are allowed to be members of triples.)
\begin{small}
\begin{center}
\begin{Young}
$a$ &,& ,\dots &,& $b$ \cr
$c$ &, &,  \cr  ,\cr
, type A ($\gamma_i \geq \gamma_j$)
\end{Young}
\hspace{1in}
\begin{Young}
, &,& , &, & $a$ \cr
$b$ &, &, \dots &, & $c$  \cr  ,\cr
, type B ($\gamma_i < \gamma_j$)
\end{Young}
\end{center} 
\end{small}
If $a, b, c$ are positive integers coming from a filling, 
the corresponding triple is called a {\em coinversion triple}
if $a \leq b \leq c$.  Otherwise the triple is called an {\em inversion triple}.

Let $\gamma$ be a  composition with $\ell(\gamma) = n$.  A 
{\em semistandard skyline filling} (SSK) of shape $\gamma$ 
is a filling of the skyline diagram of 
$\gamma$ with positive integers such that
\begin{enumerate}
\item  the entries decrease weakly up each column (including the basement) and
\item  every triple (including those involving basement cells) is an inversion triple.
\end{enumerate}
An example of a SSK of shape $(3,0,1,3)$ is shown below.
\begin{small}
\begin{center}
\begin{Young}
3 &, & ,& 1 \cr 
 3 &, & ,& 1 \cr
 4 &, & 2 & 1 \cr
 ,4 & ,3 & ,2 & ,1 \cr
\end{Young}
\end{center}
\end{small}
Let $\SSK(\gamma)$ be the set of SSK of shape $\gamma$.

Let $\gamma$ be a  composition with $\ell(\gamma) = n$.  The {\em Demazure character}
is the polynomial
$\kappa_{\gamma}(\xx_n) \in \QQ[\xx_n]$ given by
\begin{equation}
\kappa_{\gamma}(\xx_n) = \sum_{T \in \SSK(\gamma^*)} 
x_1^{\#\text{ of 1s in $T$}} x_2^{\# \text{ of 2s in $T$}} \cdots.
\end{equation}
Note that the Demazure character labeled by $\gamma$ is the generating function of SSK
of shape given by the {\em reverse} composition $\gamma^*$.
For example, the  SSK shown above contributes $x_1^3 x_2 x_3^2 x_4$ to 
$\kappa_{(3,1,0,3)}(\xx_4)$.
While the original definition of the Demazure character was not combinatorial,
we will take this combinatorial reformulation (due to Mason \cite{Mason}) as our definition.

Demazure characters are rarely symmetric polynomials.  In fact,
the collection of all Demazure characters
$\{ \kappa_{\gamma}(\xx_n) \,:\, \text{$\gamma$ a weak composition} \}$ forms 
a basis of the polynomial ring $\QQ[\xx_n]$.  
Let $B \subseteq GL_n(\CC)$ be the subgroup of upper triangular matrices.
The Demazure character
$\kappa_{\gamma}(\xx_n)$ is the trace of the diagonal matrix
$\mathrm{diag}(x_1, \dots , x_n)$ acting on the indecomposable polynomial representation of 
$B$ indexed by $\gamma$.

It will be convenient  to consider Demazure characters in a reversed set of variables.
We denote by $\xx_n^* = (x_n, \dots, x_1)$ our list of $n$ variables in reverse order.
Hence, if $f(\xx_n) = f(x_1, \dots, x_n) \in \QQ[\xx_n]$ is any polynomial, we set
$f(\xx_n^*) := f(x_n, \dots, x_1)$.  
In particular, for any composition $\gamma$ with $\ell(\gamma) = n$
we have the {\em reverse Demazure character}
\begin{equation*}
\kappa_{\gamma}(\xx_n^*) = \kappa_{\gamma}(x_n, \dots, x_1).
\end{equation*}
Note that if $f(\xx_n) \in \QQ[\xx_n]^{\symm_n}$ is a symmetric polynomial, we have
$f(\xx_n) = f(\xx_n^*)$.

Let us mention a recursive construction of the Demazure characters.  
If $\gamma = (\gamma_1 \geq \cdots \geq \gamma_n)$ is a {\em dominant} composition,
we have $\kappa_{\gamma}(\xx_n) = x_1^{\gamma_1} \cdots x_n^{\gamma_n}$.  In general,
suppose that $\gamma = (\gamma_1, \dots , \gamma_n)$ is obtained
from $\gamma' = (\gamma'_1, \dots, \gamma'_n)$ by swapping 
$\gamma'_i > \gamma'_{i+1}$.  Then
\begin{equation*}
\kappa_{\gamma}(\xx_n) = \frac{1 - s_i}{x_i - x_{i+1}} [x_i \cdot \kappa_{\gamma'}(\xx_n)].
\end{equation*}
Here $s_i$ acts on polynomials by interchanging $x_i$ and $x_{i+1}$, so that 
$\frac{1 - s_i}{x_i - x_{i+1}}$ is the {\em divided difference operator}
and 
$\frac{1 - s_i}{x_i - x_{i+1}} \cdot x_i$ is the {\em isobaric divided difference operator}.
Although this recursive construction (the {\em Demazure Character Formula}) could be used 
to prove some of our results, Mason's \cite{Mason} 
combinatorial interpretation of Demazure characters in terms of SSK
will be crucial in our work.

\subsection{Gr\"obner bases}
A total order $<$ on the monomials in the polynomial ring $\QQ[\xx_n]$ is called a 
{\em monomial order} if 
\begin{enumerate}
\item for any monomial $m$ we have $1 \leq m$, and
\item for any monomials $m, m',$ and $m''$, we have that $m < m'$ implies $m \cdot m'' < m' \cdot m''$.
\end{enumerate}
The {\em lexicographic} monomial order $<_{lex}$ is defined by 
$x_1^{a_1} \cdots x_n^{a_n} <_{lex} x_1^{b_1} \cdots x_n^{b_n}$ if there is an index $i$ such
that $a_1 = b_1, \dots, a_{i-1} = b_{i-1},$ and $a_i < b_i$.

Let $<$ be a monomial order on $\QQ[\xx_n]$.
For any nonzero polynomial $f \in \QQ[\xx_n]$, let 
$\initial_{<}(f)$ be the leading monomial of $f$ with respect to $<$.
For an ideal $I \subseteq \QQ[\xx_n]$, let $\initial_<(I) = \langle \initial_<(f) \,:\, f \in I - \{0\} \rangle$ be the  ideal
generated by the leading monomials of nonzero polynomials in $I$.

A finite collection $G = \{g_1, \dots, g_r\}$ of nonzero polynomials in $I$ is a
{\em Gr\"obner basis} if 
$\initial_<(I) = \langle \initial_<(g_1), \dots, \initial_<(g_r) \rangle$;
this immediately implies that $I = \langle g_1, \dots, g_r \rangle$.  The Gr\"obner basis
$G$ is called {\em reduced} if 
\begin{enumerate}
\item  the coefficient of $\initial_<(g_i)$ is $1$ for all $i$, and
\item  for $i \neq j$, no monomial in $g_j$ is divisible by $\initial_<(g_i)$.
\end{enumerate}
For a fixed monomial order, every ideal $I \subseteq \QQ[\xx_n]$ has a 
unique reduced Gr\"obner basis.

Gr\"obner bases are helpful in constructing linear bases of quotient rings.
In particular, let $I \subseteq \QQ[\xx_n]$ be an ideal and let 
$G = \{g_1, \dots, g_r\}$ be a Gr\"obner basis of $I$ with respect to some monomial order $<$.
A monomial $m \in \QQ[\xx_n]$ is called a {\em standard monomial} if 
$m \nmid \initial_<(f)$  for all $f \in I - \{0\}$.  Equivalently, a monomial $m$ is standard 
if $m \nmid \initial_<(g_i)$ for $1 \leq i \leq r$.  The collection
\begin{equation*}
\{m + I \,:\, \text{$m$ a standard monomial} \}
\end{equation*}
gives a vector space basis for the quotient ring 
$\frac{\QQ[\xx_n]}{I}$; this is called the {\em standard monomial basis}.

\section{Some polynomial identities}
\label{Some polynomial identities}

In this section we  prove a family of identities involving polynomials and symmetric functions
which will serve as lemmata for the analysis of the quotient ring $R_{n,k}$.
The first of these is a vanishing property satisfied by certain alternating products of 
elementary and homogeneous symmetric function evaluations.

\begin{lemma}
\label{vanishing-lemma}
Let $k \leq n$, let $\alpha_1, \dots, \alpha_k \in \QQ$ be distinct rational numbers, and let 
$\beta_1, \dots, \beta_n \in \QQ$ be rational numbers with the property that
$\{ \beta_1, \dots, \beta_n \} = \{\alpha_1, \dots , \alpha_k\}$.
For any $n-k+1 \leq r \leq n$ we have 
\begin{equation}
\label{vanishing-equation}
\sum_{j = 0}^r (-1)^j e_{r-j}(\beta_1, \dots, \beta_n) h_j(\alpha_1, \dots, \alpha_k) = 0.
\end{equation}
\end{lemma}

\begin{proof}
The left hand side of Equation~\ref{vanishing-equation} is the coefficient of $t^r$ in the power series
\begin{equation}
\label{factor-equation}
\frac{\prod_{i = 1}^n (1 + t \beta_i)}{\prod_{i = 1}^k (1 + t \alpha_j)}.
\end{equation}
By assumption, every term in the denominator cancels a term in the numerator, so
this power series is in fact a polynomial in $t$ of degree $n-k$.
If $r > n-k$, the left hand side of Equation~\ref{vanishing-equation} therefore equals zero.
\end{proof}

The next lemma will be used to show that certain reverse Demazure characters
$\kappa_{\gamma}(\xx_n^*)$ lie in the ideal $I_{n,k}$.  In order to state this lemma, we will
introduce a family of `skip' objects related to a set $S \subseteq [n]$.

\begin{defn}
Let $S = \{s_1 < \cdots < s_m \} \subseteq [n]$ be a set.
\begin{enumerate}
\item  The {\em skip monomial} $\xx(S) \in \QQ[\xx_n]$ is the monomial
\begin{equation*}
\xx(S) := x_{s_1}^{s_1} x_{s_2}^{s_2 - 1} \cdots x_{s_m}^{s_m - m + 1}.
\end{equation*}
\item  The {\em skip composition} $\gamma(S) = (\gamma_1, \dots, \gamma_n)$ 
is the weak composition of length $n$  defined by
\begin{equation*}
\gamma_i := \begin{cases}
0 & i \notin S \\
s_j - j + 1 & i = s_j \in S.
\end{cases}
\end{equation*}
\end{enumerate}
\end{defn}

For example, let $n = 8$ and take $S = \{2,3,6,8\}$.  We have
$\xx(S) = x_2^2 x_3^2 x_6^4 x_8^5 \in \QQ[\xx_8]$ and
$\gamma(S) = (0,2,2,0,0,4,0,5)$.  In general, the support of the monomial $\xx(S)$ is the set 
$S$ and $\gamma(S)$ is the exponent vector of $\xx(S)$.  The terminology here comes from the 
fact that the power of the lowest variable in $\xx(S)$ is the variable index, and powers of higher
variables increase according to how many indices the set $S$ skips.  
Skip monomials will be crucial in our study of $R_{n,k}$.

Let $\gamma = (\gamma_1, \dots, \gamma_n)$ be a weak composition  and 
consider a collection $\rho$ of cells lying immediately above the columns of the 
skyline diagram of $\gamma$.  The collection $\rho$ will be called {\em right biased (RB) for $\gamma$} if
\begin{enumerate}
\item there is at most one cell in $\rho$ on top of any column of $\gamma$, and
\item among the columns of $\gamma$ with a fixed height $h$, the cells of $\rho$ are right-justified.
Note that $0$ is a possible value for $h$.
\end{enumerate} 
For example, consider the composition $\gamma = (0,4,0,3,3,0,2,0,0)$.  The collection
of cells marked with $\circ$ on the left is RB while the collection of $\circ$ cells on the right is not RB.
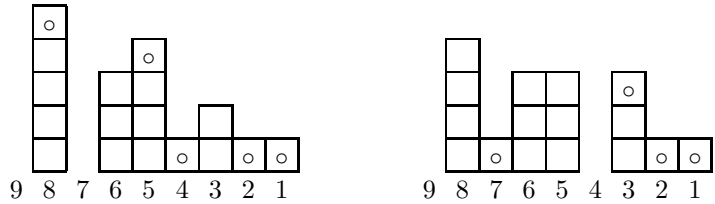
\begin{small}
\begin{center}
\begin{Young}
 , & $\circ$ & , &, &, &, &, &, &, \cr
 , & &, & , & $\circ$ &, &, &, & ,\cr
 , & &, & & &, &, &, &, \cr
 , & &, & & &, & &, &, \cr
 , & &, & & & $\circ$ & & $\circ$ & $\circ$ \cr
  ,9 & ,8 & ,7 & ,6 & ,5 & ,4 & ,3& ,2 & ,1 
\end{Young}
\hspace{0.5 in}
\begin{Young}
 , & &, & , & , &, &, &, & ,\cr
 , & &, & & &, &$\circ$ &, &, \cr
 , & &, & & &, & &, &, \cr
 , & &$\circ$ & & & , & & $\circ$ & $\circ$ \cr
  ,9 & ,8 & ,7 & ,6 & ,5 & ,4 & ,3& ,2 & ,1 
\end{Young}
\end{center}
\end{small}

The dual Pieri rule describes how to expand products of the form $e_d(\xx) s_{\lambda}(\xx)$ 
in the Schur basis of symmetric functions.
The following Demazure version of the dual Pieri rule is due to Haglund, Luoto, Mason,
and van Willigenburg.  It describes how to expand products of the form
$e_d(\xx_n) \kappa_{\gamma}(\xx_n)$ in the Demazure character basis of polynomials.
If $\rho$ is a RB collection of cells for $\gamma$, let $\gamma \cup \rho$ be the composition
which is the set theoretic union of the skyline diagram of $\gamma$ and the collection 
of cells $\rho$.

\begin{theorem}
\label{demazure-dual-pieri}
\cite[Thm. 6.1, $\lambda = (1^d)$]{HLMV}  
Let $\gamma = (\gamma_1, \dots, \gamma_n)$ be a weak composition and $d \geq 0$.
We have
\begin{equation}
e_d(\xx_n) \kappa_{\gamma}(\xx_n) = \sum_{|\rho| = d} \kappa_{(\gamma \cup \rho)}(\xx_n),
\end{equation}
where the sum is over all RB collections of cells $\rho$ of size $d$.
\end{theorem}

We will add RB collections to the skyline diagrams of compositions;  we 
now describe collections of cells we will remove.  Let 
$\gamma = (\gamma_1, \dots, \gamma_n)$ be a weak composition.  A collection $\lambda$ 
of non-basement
cells in the skyline diagram of $\gamma$ will be called {\em left leaning (LL) for $\gamma$} if
\begin{enumerate}
\item  the cells of $\lambda$ are top-justified within any column,
\item  there is at most one cell of $\lambda$ in any row, and
\item  it is impossible to move a cell of $\lambda$ to the left in such a way that  (1) remains satisfied.
\end{enumerate}
For example, if $\gamma = (0,4,0,3,3,0,2,0,0)$ the collection of cells marked with $\bullet$ 
on the left is LL while the collection of $\bullet$ cells on the right is not LL.
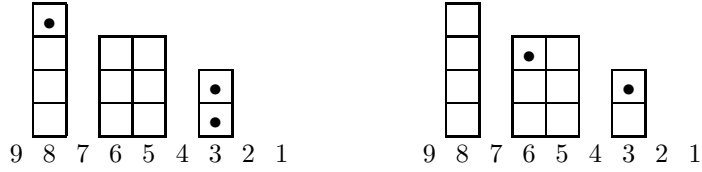
\begin{small}
\begin{center}
\begin{Young}
 , & $\bullet$  &, & , & , &, &, &, & ,\cr
 , & &, & & &, & , &, &, \cr
 , & &, & & &, & $\bullet$ &, &, \cr
 , & &, & & & , & $\bullet$ & , & , \cr
  ,9 & ,8 & ,7 & ,6 & ,5 & ,4 & ,3& ,2 & ,1 
\end{Young}
\hspace{0.5in}
\begin{Young}
 , & &, & , & , &, &, &, & ,\cr
 , & &, & $\bullet$  & &, & , &, &, \cr
 , & &, & & &, & $\bullet$ &, &, \cr
 , & &,  & & & , & & , & , \cr
  ,9 & ,8 & ,7 & ,6 & ,5 & ,4 & ,3& ,2 & ,1 
\end{Young}
\end{center}
\end{small}
If $\lambda$ is any LL collection of cells for $\gamma$, let $\gamma - \lambda$ be the 
composition whose skyline diagram is the set theoretic difference of $\gamma$ and $\lambda$.

We will need a little bit more notation to state our lemma.  For any  composition
$\gamma$, let $\overline{\gamma}$ denote the {\em decremented} composition 
obtained by subtracting $1$ from every nonzero part of $\gamma$:
\begin{equation*}
\overline{\gamma}_i := \begin{cases}
\gamma_i - 1 & \gamma_i > 0 \\
0 & \gamma_i = 0.
\end{cases}
\end{equation*}  
In particular, if $\gamma(S)$ is a skip composition, we have 
$|\overline{\gamma(S)}| = |\gamma(S)| - |S|$.  Also, for any collection of cells $\nu$,
let $|\nu|$ denote the number of cells in $\nu$.

\begin{lemma}
\label{demazure-identity}
Let $k \leq n$ and let $S \subseteq [n]$ with $|S| = n-k+1$.  Let $\gamma(S)$ be the corresponding
skip monomial and consider the reversal $\gamma(S)^*$ and the decremented reversal
$\overline{\gamma(S)^*}$.  

We have the identity
\begin{equation}
\label{magical-demazure-equation}
\kappa_{\gamma(S)^*}(\xx_n) = \sum_{\lambda} (-1)^{|\lambda|}
\kappa_{\overline{\gamma(S)^*} - \lambda}(\xx_n) e_{n-k+1+|\lambda|}(\xx_n),
\end{equation}
where the sum is over all LL collections $\lambda$ for $\overline{\gamma(S)^*}$.

Reversing variables and applying the symmetry of $e_r(\xx_n)$ we get
\begin{equation}
\label{magical-demazure-equations-friend}
\kappa_{\gamma(S)^*}(\xx_n^*) = \sum_{\lambda} (-1)^{|\lambda|}
\kappa_{\overline{\gamma(S)^*} - \lambda}(\xx_n^*) e_{n-k+1+|\lambda|}(\xx_n).
\end{equation}
\end{lemma}

The proof of this lemma is involved, but worth it.  
Equation~\ref{magical-demazure-equations-friend} will be our tool for proving 
that   reverse Demazure characters 
 lie in  various ideals.

\begin{proof}
It suffices to prove Equation~\ref{magical-demazure-equation}.  To do this, we will introduce
a sign-reversing involution on a set of combinatorial objects called `partisan skylines'.

A {\em partisan skyline} $\nu$ is a skyline diagram of a composition whose 
non-basement
cells have one
of four
labels
\begin{center}
\begin{Young}
 \cr
\end{Young} \hspace{0.3in}
\begin{Young}
$\circ$ \cr
\end{Young} \hspace{0.3in}
\begin{Young}
$\bullet$ \cr
\end{Young} \hspace{0.3in}
\begin{Young}
$\odot$ \cr
\end{Young}
\end{center}
satisfying the following conditions.  For any collection of labels $L$, we let 
$\nu(L)$ denote the collection of cells in $\nu$ which have labels in $L$.
\begin{enumerate}
\item Reading any column of $\nu$ from bottom to top gives a sequence of labels
$\square$,
followed by a sequence of labels $\odot$,
followed by a sequence of labels  $\bullet$,
followed by a sequence of labels $\circ$.  Any of these four sequences could be empty.
If a column contains the label $\circ$, it is the lone non-$\square$ box in its column.
\item   We have $\nu(\square, \bullet, \odot) = \overline{\gamma(S)^*}$.
\item   We have $|\nu(\square, \circ, \odot)| = |\gamma(S)^*|$.
\item  The collection of cells $\nu(\bullet, \odot)$  is LL 
for $\overline{\gamma(S)^*}$.
\item  The collection of cells $\nu(\circ, \odot)$
 is RB for $\nu(\square)$.
\end{enumerate}
The {\em sign} of a partisan skyline $\nu$ is
\begin{equation*}
\sign(\nu) = (-1)^{|\nu(\bullet, \odot)|}
\end{equation*}
and the {\em weight} 
of $\nu$ is the composition 
\begin{equation*}\wt(\nu) = \nu(\square, \circ, \odot).
\end{equation*}

Partisan skylines combinatorially encode the right hand side of 
Equation~\ref{magical-demazure-equation}.  By Theorem~\ref{demazure-dual-pieri},
a typical term in this expression is obtained
by choosing a LL collection $\lambda$ for $\overline{\gamma(S)^*}$ to remove (at the cost of a sign),
and then adding a RB collection $\rho$ of cells to the resulting composition with
$|\rho| = n-k+1+|\lambda| = |\gamma(S)^*| + |\lambda|$.
Cells which are removed are labeled $\bullet$, cells which are added are labeled $\circ$, and 
cells which are removed and then added are labeled $\odot$.
Said differently, we have
\begin{equation}
\label{lemma-3.4-one}
\sum_{\lambda} (-1)^{|\lambda|}
\kappa_{\overline{\gamma(S)^*} - \lambda}(\xx_n) e_{n-k+1+|\lambda|}(\xx_n) = 
\sum_{\nu} \sign(\nu) \kappa_{\wt(\nu)}(\xx_n),
\end{equation}
where the sum is over all partisan skylines $\nu$.

As an example of this construction, consider the case $n = 9$,
$k = 5$, and $S = \{1,3,5,6,9\}$.  
Then $\gamma(S) = (1,0,2,0,3,3,0,0,5), \overline{\gamma(S)} = (0,0,1,0,2,2,0,0,4)$, and
the skyline diagram of 
$\overline{\gamma(S)^*}$ is shown below.
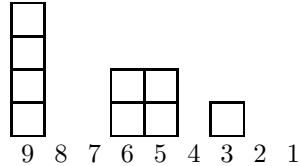
\begin{small}
\begin{center}
\begin{Young}
 &, &, &, &, &, &, &, &, \cr
 &, &, & ,&, &, &, &, &, \cr
  &, & , & & &, &, &, &, \cr
  &, &, & & & , & & , & , \cr
  ,9 & ,8 & ,7 & ,6 & ,5 & ,4 & ,3& ,2 & ,1 
\end{Young}
\end{center}
\end{small}
Three possible partisan skylines are shown in (\ref{partisans}) below.
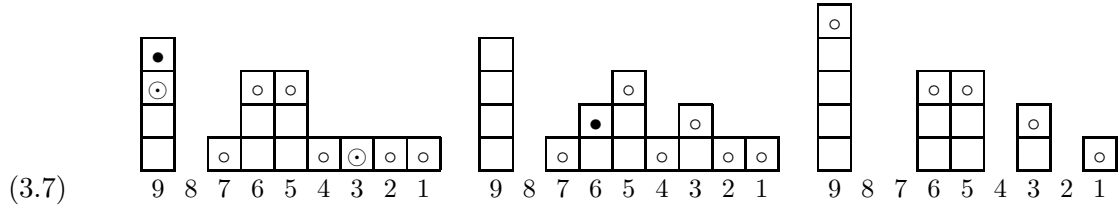
\begin{equation}
\label{partisans}
\begin{small}
\begin{Young}
\bullet &, &, &, &, &, &, &, &, \cr
\odot &, &, & \circ & \circ &, &, &, &, \cr
  &, & , & & &, &,  &, &, \cr
  &, & \circ & & & \circ &\odot  & \circ & \circ \cr
  ,9 & ,8 & ,7 & ,6 & ,5 & ,4 & ,3& ,2 & ,1 
  \end{Young}
  \hspace{0.2in}
\begin{Young}
 &, &, &, &, &, &, &, &, \cr
 &, &, & , & \circ &, &, &, &, \cr
  &, & , & \bullet & &, & \circ &, &, \cr
  &, & \circ &   & &  \circ & & \circ & \circ \cr
  ,9 & ,8 & ,7 & ,6 & ,5 & ,4 & ,3& ,2 & ,1 
\end{Young}
\hspace{0.2in}
\begin{Young}
\circ &, &, &, &, &, &, &, &, \cr
 &, &, &, &, &, &, &, &, \cr
 &, &, & \circ & \circ &, &, &, &, \cr
  &, & , & & &, & \circ &, &, \cr
  &, &, & & & , & & , & \circ \cr
  ,9 & ,8 & ,7 & ,6 & ,5 & ,4 & ,3& ,2 & ,1 
\end{Young}
\end{small}
\end{equation}
From left to right, the weights of these partisan skylines are
$(3,0,1,3,3,1,1,1,1), (4,0,1,1,3,1,2,1,1),$ and $(5,0,0,3,3,0,2,0,1) = \gamma(S)^*$.
Also from left to right, the signs of these partisan skylines are
$(-1)^3 = -1, (-1)^1 = -1,$ and $(-1)^0 = 1$.  If $\nu$ is the partisan skyline on the left of 
(\ref{partisans}), then $\nu(\square, \odot, \circ), \nu(\square),$ and $\nu(\odot, \circ)$ are 
(from left to right):
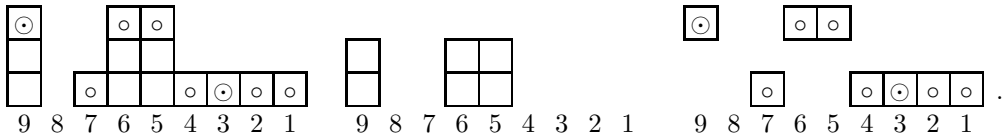
\begin{equation*}
\begin{small}
\begin{Young}
\odot &, &, & \circ & \circ &, &, &, &, \cr
  &, & , & & &, &,  &, &, \cr
  &, & \circ & & & \circ &\odot  & \circ & \circ \cr
  ,9 & ,8 & ,7 & ,6 & ,5 & ,4 & ,3& ,2 & ,1 
  \end{Young}
  \hspace{0.2in}
  \begin{Young}
  &, & , & & &, &,  &, &, \cr
  &, & , & & &, & ,  & , & , \cr
  ,9 & ,8 & ,7 & ,6 & ,5 & ,4 & ,3& ,2 & ,1 
  \end{Young}
  \hspace{0.2in}
  \begin{Young}
\odot &, &, & \circ & \circ &, &, &, &, \cr
 , &, & , &, & ,&, &,  &, &, \cr
 , &, & \circ &, &, & \circ &\odot  & \circ & \circ &,. \cr
  ,9 & ,8 & ,7 & ,6 & ,5 & ,4 & ,3& ,2 & ,1 
  \end{Young}
\end{small}
\end{equation*}

A column in a partisan skyline $\nu$ is declared {\em frozen} if it contains a cell labeled
$\circ$.  There is a unique partisan skyline in which all nonempty columns are frozen: one simply 
places a single cell labeled $\circ$ on top of every nonempty column in the skyline diagram
of $\overline{\gamma(S)^*}$, together with the appropriate number of cells labeled 
$\circ$ in the rightmost empty columns of  $\overline{\gamma(S)^*}$.  This is the
{\em completely frozen} partisan skyline $\nu_0$; an example is shown on the right 
(\ref{partisans}).  The completely frozen partisan skyline 
$\nu_0$ always satisfies $\sign(\nu_0) = 1$ and $\wt(\nu_0) = \gamma(S)^*$.

We wish to define a sign-reversing and
weight-preserving involution $\iota$ on the set of partisan skylines which
are not completely frozen.   To do this, we will use

{\bf Claim 1:} {\em If $\nu$ is any partisan skyline, then
$\nu$ does not contain a nonempty column consisting only of cells labeled $\bullet$.}

\begin{proof} {\em (of Claim 1)}
Write $S = \{s_1 < \cdots < s_{n-k+1} \}$, let $1 \leq i \leq n-k+1$, and 
suppose the nonempty column $C$ in position $n - s_i + 1$ of $\nu$ consisted only of cells 
labelled $\bullet$.  Since the column $C$ contains $s_i - i$ cells, there must be at least 
$(n-k+1) + (s_i - i)$ cells in $\nu$ labeled $\circ$.  
Let $e$ be the number of empty columns strictly to the right of $C$ in
$\overline{\gamma(S)^*}$.
By considering the number of nonempty 
columns in $\overline{\gamma(S)^*}$ besides $C$, Conditions (3) and (5)
in the definition of a partisan skyline force at least $e+1$
of these $\circ$ cells to be at height $1$.  But then Condition (5) forces the bottom box of 
$C$ to be labeled $\circ$, which is a contradiction.  
\end{proof}

We are ready to define our involution $\iota$.  Let $\nu$ be a partisan skyline which is 
not completely frozen.  
For from left to right, write the unfrozen nonempty columns of $\nu$ as 
$C_1, \dots, C_m$.  For $1 \leq i \leq m$, let $h_i$ be the number of boxes
in $C_i$ labeled $\square$ or $\odot$.  
Let $h = \min \{ h_i \,:\, 1 \leq i \leq m \}$ and let $C$ be the {\em leftmost}
column among $C_1, \dots, C_m$ with $h$ boxes labeled 
$\square$ or $\odot$.
By Claim 1, we have $h > 0$.
The highest cell in $C$ which is not labeled $\bullet$
is labeled either $\odot$ or 
$\square$.
If this cell is labeled $\odot$, change its label to
$\square$.
If this cell is labeled
$\square$,
change its label to $\odot$.  Let $\iota(\nu)$ be the resulting diagram.
As an example of the map $\iota$, the left and middle partisan skylines of
(\ref{partisans}) are sent to
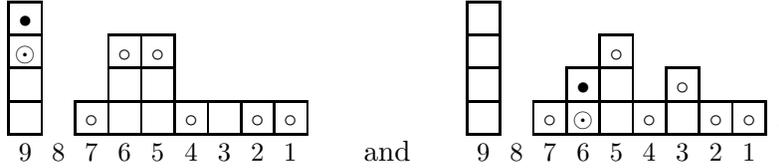
\begin{small}
\begin{center}
\begin{Young}
$\bullet$ &, &, &, &, &, &, &, &, \cr
$\odot$ &, &, & $\circ$ & $\circ$ &, &, &, &, \cr
  &, & , & & &, &,  &, &, \cr
  &, & $\circ$ & & & $\circ$ &    & $\circ$ & $\circ$ \cr
  ,9 & ,8 & ,7 & ,6 & ,5 & ,4 & ,3& ,2 & ,1 
  \end{Young}
  \hspace{0.2in} \begin{normalsize}and\end{normalsize} \hspace{0.2in}
\begin{Young}
 &, &, &, &, &, &, &, &, \cr
 &, &, & , & $\circ$ &, &, &, &, \cr
  &, & , & $\bullet$ & &, & $\circ$ &, &, \cr
  &, & $\circ$ & $\odot$   & &  $\circ$ & & $\circ$ & $\circ$ &,, \cr
  ,9 & ,8 & ,7 & ,6 & ,5 & ,4 & ,3& ,2 & ,1 
\end{Young}
\end{center}
\end{small}
respectively (we have $h = 1$ in either of these examples).
To give another example (where $h = 2$ this time) the map $\iota$ 
interchanges
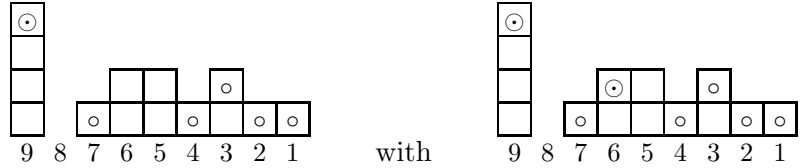
\begin{small}
\begin{center}
\begin{Young}
$\odot$ &, &, &, &, &, &, &, &, \cr
   &, &, & , & , &, &, &, &, \cr
  &, & , & & &, & $\circ$  &, &, \cr
  &, & $\circ$ & & & $\circ$ &    & $\circ$ & $\circ$ \cr
  ,9 & ,8 & ,7 & ,6 & ,5 & ,4 & ,3& ,2 & ,1 
  \end{Young}
  \hspace{0.2in} \begin{normalsize} with \end{normalsize} \hspace{0.2in}
\begin{Young}
$\odot$ &, &, &, &, &, &, &, &, \cr
   &, &, & , & , &, &, &, &, \cr
  &, & , & $\odot$ & &, & $\circ$  &, &, \cr
  &, & $\circ$ & & & $\circ$ &    & $\circ$ & $\circ$ \cr
  ,9 & ,8 & ,7 & ,6 & ,5 & ,4 & ,3& ,2 & ,1 
  \end{Young}
\end{center}
\end{small}

{\bf Claim 2:}
{\em Let $\nu$ be a partisan skyline which is not completely frozen. Then $\iota(\nu)$ is also 
a partisan skyline which is not completely frozen.  In fact, the sets of not frozen columns
of $\nu$ and $\iota(\nu)$ coincide.}

\begin{proof} {\em (of Claim 2)}
Let $C$ be the column which differs between $\nu$ and $\iota(\nu)$ and let 
$(i, j)$ be the coordinates the cell whose label was changed.
In the three examples above we have $(i, j) = (7,1), (4,1)$, and $(4,2)$, respectively.
Conditions (1), (2), and (3) are obviously satisfied
by $\iota(\nu)$.  

Condition (4) is clearly preserved if $\iota$ changes the label of the cell $(i, j)$ from
$\odot$ to $\square$.
The only way Condition (4) could fail is if $\iota$ changes the label of 
$(i, j)$ from
$\square$ to  $\odot$, resulting in a collection of cells 
$\iota(\nu)(\odot, \bullet)$
which is not LL for $\iota(\nu)(\square, \odot, \bullet) = \nu(\square, \odot, \bullet)$.
If $\iota(\nu)(\odot, \bullet)$ is not LL for $\iota(\nu)(\square,\odot,\bullet)$ then
either 
\begin{verse}
(i) there exists a cell $(i',j-1)$ in $\nu$ with $i' > i$ containing the label
$\odot$ or $\bullet$, or \\ (ii) 
there exists a cell $(i'',j+1)$ in $\nu$ with $i'' < i$ containing the label 
$\odot$ or $\bullet$.
\end{verse}

Suppose there exists a cell $(i', j-1)$ in $\nu$ with $i' > i$ containing the label
$\odot$ or $\bullet$.  The column $C'$ containing $i'$ is not frozen and has 
strictly fewer boxes labeled $\square$ or $\odot$ than $C$.  This contradicts
the definition of $\iota$.

Suppose there exists a cell $(i'',j+1)$ in $\nu$ with $i'' < i$ containing the label 
$\odot$ or $\bullet$.  If $(i'',j+1)$ contains the label $\bullet$, then 
the column $C''$ containing $(i'',j+1)$ would have no more boxes labeled
$\square$ or $\odot$ than $C$; since $C''$ is to the left of $C$, this contradicts
the definition of $\iota$.  If $(i'',j+1)$ contains the label $\odot$, then 
since $(i,j)$ contains the label $\square$ in $\nu$, the collection
$\nu(\circ, \odot)$ is not RB for $\nu(\square)$, contradicting the assumption
that $\nu$ is a partisan skyline.  We conclude that 
$\iota(\nu)(\bullet, \odot)$ is LL for $\overline{\gamma(S)}^*$, so that 
Condition (4) holds for $\iota(\nu)$.

Let us consider Condition (5).  If $\iota$ changes the label of $(i,j)$ from $\odot$ to $\square$,
suppose the collection $\iota(\nu)(\circ, \odot) = \nu(\circ, \odot) - \{(i,j)\}$ were not RB for 
$\iota(\nu)(\square)$.  There must then exist  some cell $(i'', j+1)$ in $\nu$ with $i'' < i$ which is labeled
with $\circ$ or $\odot$.  Since the cell $(i, j)$ has the label $\odot$ in $\nu$ and no two
cells of $\nu$ in the same row can have label $\bullet$ or $\odot$, the cell $(i'', j)$ in $\nu$ must be labeled
$\square$.  However, this violates the assumption that $\nu(\bullet, \odot)$ is LL for $\nu(\square, \bullet, \odot)$
(because we could have moved the $\odot$ in position $(i, j)$ of $\nu$ to the (leftward) position
$(i'', j)$). We conclude that 
$\iota(\nu)(\circ, \odot)$ is RB for $\nu(\square)$, so that Condition (5) holds for $\iota(\nu)$ in this case.

Finally, consider Condition (5) in the case where $\iota$ changes the label of $(i,j)$
from $\square$ to $\odot$.  Suppose the collection
$\iota(\nu)(\circ, \odot) = \nu(\circ, \odot) \cup \{(i,j)\}$ were not RB for $\iota(\nu)(\square)$.  Then there 
must be some $i' > i$ such that $(i',j)$ is not a cell in $\nu(\square,\circ,\odot)$ but $(i',j-1)$ is a cell
in $\nu(\square)$.  In particular, the column $C'$ of $\nu$ containing $(i',j-1)$ is not frozen.
If $j > 1$, then $C'$ is also nonempty, 
and contains strictly fewer boxes labeled $\square$ or $\odot$ than $C$,
contradicting the definition of $\iota$.
If $j = 1$, the fact that $\nu$ is partisan forces {\em every} coordinate $(i', j)$ with $i' > i$ to be a cell
of $\nu(\square, \circ)$, again a contradiction. 

We have demonstrated that $\iota(\nu)$ is a partisan skyline.
It is clear that $\iota$ does not freeze any columns.  This completes the proof of Claim 2.
\end{proof}

By Claim 2 $\iota$ is a well defined operator on the set of partisan skylines which
are not completely frozen.  Moreover, 
 the map $\iota$ does not affect the total number of boxes labeled $\square$
 or $\odot$ in any column (it merely swaps an $\odot$ for a $\square$ or vice
 versa).  It is also evident that $\iota$ does not freeze any columns.
 This implies that 
$\iota$ is an involution.  Since $\nu(\square, \circ, \odot) = \iota(\nu)(\square, \circ, \odot)$,
the map $\iota$ preserves weight.  Since $\iota$ changes the parity of the set
$|\nu(\bullet, \odot)|$, the map $\iota$ reverses sign.  The right hand side of 
Equation~\ref{lemma-3.4-one} becomes
\begin{equation}
\sum_{\nu} \sign(\nu) \kappa_{\wt(\nu)}(\xx_n) =
\sign(\nu_0) \kappa_{\wt(\nu_0)}(\xx_n) =
\kappa_{\gamma(S)^*}(\xx_n)
\end{equation}
and the proof is complete.
\end{proof}

Our proof of Lemma~\ref{demazure-identity} was combinatorial and relied
on a sign-reversing involution.   Since the polynomials involved in 
Equation~\ref{magical-demazure-equation} are characters of indecomposable representations
of the Borel subgroup $B \subseteq GL_n(\CC)$, Lemma~\ref{demazure-identity} suggests
the existence of a long exact sequence whose terms are (tensor products of) the corresponding
modules.  It may be interesting to find a direct algebraic proof of Lemma~\ref{demazure-identity}.

In order to determine the reduced Gr\"obner basis of the ideal $I_{n,k}$, we will need
to study the monomials appearing in the reverse Demazure characters
$\kappa_{\gamma(S)^*}(\xx_n^*)$.  Our lemma in this direction is as follows.  It states that the 
leading terms of these Demazure characters are skip monomials.

\begin{lemma}
\label{reduced-demazure-lemma}
Let $k \leq n$ and let $S \subseteq [n]$ satisfy $|S| = n-k+1$.  Let $<$ be lexicographic order.  We have
\begin{equation*}
\initial_<(\kappa_{\gamma(S)^*}(\xx_n^*)) = \xx(S).
\end{equation*}
Moreover, for any $1 \leq i \leq n$ we have 
\begin{equation*}
x_i^{\max(S)- n + k + 1} \nmid m
\end{equation*} 
for any monomial $m$ appearing
in $\kappa_{\gamma(S)^*}(\xx_n^*)$.
Finally, if $T \subseteq [n]$ satisfies $|T| = n-k+1$ and $T \neq S$, then $\xx(S) \nmid m$
for any monomial $m$ appearing in $\kappa_{\gamma(T)^*}(\xx_n^*)$.
\end{lemma}

\begin{proof}
We  use the SSK model for Demazure characters.  Given $S$, the {\em big filling} of
$\gamma(S)$ is obtained by letting the entries in any column equal the entry in the basement 
of that column.  For example, if $(n, k) = (5,3)$ and $S = \{2,3,5\}$, then 
$\gamma(S) = (0,2,2,0,3)$ and the big filling of $\gamma(S)$ is shown
below.
\begin{center}
\begin{small}
\begin{Young}
  , &, &, &, &1 \cr
 , &4 & 3 & , &1 \cr
 , &4 & 3 & , &1 \cr
 ,5 & ,4 & ,3 & ,2 & , 1
\end{Young}
\end{small}
\end{center}
It can be observed that the big filling is SSK.  The big filling above contributes 
$x_1^3 x_3^2 x_4^2$ to $\kappa_{(0,2,2,0,3)}(\xx_5)$, and hence
$x_2^2 x_3^2 x_5^3 = \xx(235)$ to $\kappa_{(0,2,2,0,3)}(\xx_5^*)$.  In general, the big filling
of $\gamma(S)$ contributes $\xx(S)$ to $\kappa_{\gamma(S)^*}(\xx_n^*)$.  

The {\em reverse content} of any SSK $F$ is the vector
\begin{equation*}
\mathrm{revcont}(F) := (\# \text{ of $n$'s in $F$}, \dots,  \# \text{ of $2$'s in $F$}, \# \text{ of $1$'s in $F$}).
\end{equation*}
From the condition
that SSK are weakly increasing up columns  it follows that the big filling of $\gamma(S)$ 
is the unique filling of $\gamma(S)$ with the lexicographically largest reverse content vector.  
This immediately implies that $\initial_<(\kappa_{\gamma(S)^*}(\xx_n^*)) = \xx(S)$.  

Let $m$ be a monomial appearing in $\kappa_{\gamma(S)^*}(\xx_n^*)$ and let $1 \leq i \leq n$.
If $x_{n-i+1}^{\max(S)-n+k+1} \mid m$,
 there would be a SSK $F$ of shape $\gamma(S)$ containing at least $(\max(S) - n + k+1)$
copies of $i$.  By the definition of $\gamma(S)$, the highest column in the skyline diagram of 
$\gamma(S)$ has height $(\max(S) - n + k)$.  This implies that $F$ contains two copies of $i$ in the 
same row.   If these two
copies of $i$ occur in columns of the same height, the triple of cells in $F$
\begin{small}
\begin{center}
\begin{Young}
$i$& ,&, $\cdots$ & , & $i$  \cr
$j$ &, &, &, & ,
\end{Young}
\end{center}
\end{small}
forms  a type A coinversion triple, a contradiction.  If these copies occur in 
columns which are not of the same height, the triple of cells in $F$
\begin{small}
\begin{center}
\begin{Young}
,& ,& ,& , & $j$  \cr
$i$ &, &, $\cdots$ &, & $i$
\end{Young}
\end{center}
\end{small}
forms a type B coinversion triple, again a contradiction.  We conclude
that $x_{n-i+1}^{\max(S)-n+k+1} \nmid m$.

Finally, suppose $T \subseteq [n]$ satisfies $|T| = n-k+1$ and $\xx(S) \mid m$ for some 
monomial $m$ appearing in $\kappa_{\gamma(T)^*}(\xx_n^*)$.  Let $F$ be a SSK of shape 
$\gamma(T)$ which contributes $m$ to $\kappa_{\gamma(T)^*}(\xx_n^*)$.  Since $\xx(S) \mid m$ and 
$F$ decreases weakly up columns, for all $i$
we have
\begin{equation*}
\# \text{ of cells in columns $1, 2, \dots, i$ of $\gamma(T)$ } \geq
\# \text{ of cells in columns $1, 2, \dots, i$ of $\gamma(S)$}.
\end{equation*}
The definition of skip compositions forces $S = T$.
\end{proof}

To determine the graded Frobenius series of $R_{n,k}$, we will need to carefully study the symmetric
function operators $e_j(\xx)^{\perp}$.  The basic tool we will use is as follows.

\begin{lemma}
\label{e-perp-determines-symmetric-function}
Let $F(\xx), G(\xx) \in \Lambda$ be symmetric functions with equal constant
terms.  We have $F(\xx) = G(\xx)$ if and only if 
$e_j(\xx)^{\perp} F(\xx) = e_j(\xx)^{\perp} G(\xx)$ for all $j \geq 1$.
\end{lemma}

\begin{proof}
The forward direction is obvious.  For the reverse implication, observe that for any partition $\lambda$
and $j \geq 1$
we have
\begin{align}
\langle F(\xx), e_j(\xx) e_{\lambda}(\xx) \rangle &= \langle e_j(\xx)^{\perp} F(\xx), e_{\lambda}(\xx) \rangle \\
&= \langle e_j(\xx)^{\perp} G(\xx), e_{\lambda}(\xx) \rangle \\
&= \langle G(\xx), e_j(\xx) e_{\lambda}(\xx) \rangle,
\end{align}
so that $F(\xx) = G(\xx)$.
\end{proof}

To use Lemma~\ref{e-perp-determines-symmetric-function} to prove 
$\grFrob(R_{n,k}; q) = D_{n,k}(\xx;q)$, we will need to determine the images of 
$\grFrob(R_{n,k}; q)$ and
$D_{n,k}(\xx;q)$ under the operator $e_j(\xx)^{\perp}$ for $j \geq 1$.    The case of
$\grFrob(R_{n,k}; q)$ will be handled in Section~\ref{Frobenius series} after we develop
a better understanding of the algebraic combinatorics of  $R_{n,k}$.  The 
case of $D_{n,k}(\xx;q)$ can be treated now.

\begin{lemma}
\label{e-perp-image}
We have 
\begin{equation}
\label{e-perp-equation}
e_j(\xx)^{\perp} D_{n,k}(\xx;q) = 
q^{{j \choose 2}} {k \brack j}_q \cdot 
\sum_{m = \max(1,k-j)}^{\min(k,n-j)}  
q^{(k-m) \cdot (n-j-m)}  {j \brack k-m}_q  D_{n-j,m}(\xx;q).
\end{equation}
\end{lemma}

\begin{proof}

General facts about Gessel fundamental quasisymmetric functions                 
and the superization of symmetric functions \cite{HHLRU}, \cite[pp. 99-101]{Hagbook} imply that (\ref{new}) is equivalent to
the statement that for any (strong) composition 
$\beta = (\beta _1, \ldots , \beta _p)$ of $n$ into positive parts
\begin{align}
\label{snew}
\langle D_{n,k}(\xx;q), e_{\beta _1}(\xx) \cdots e_{\beta _p}(\xx) \rangle =  
\sum_{\sigma \in \OP_{n,k} \atop \text{revword$(\sigma)$ is a $\beta$-shuffle}  } q^{\coinv (\sigma)}, 
\end{align}
where the sum is over all $\sigma$ whose reverse reading word $\text{revword}(\sigma)$ is a shuffle of the decreasing sequences 
$$
(\beta _{1} ,\ldots ,2,1),(\beta _1 + \beta _2, \ldots , \beta _1 +1), \ldots ,(n,n-1,\ldots , \beta _1+\ldots +\beta _{p-1}+1).
$$
Furthermore, letting $\beta_1=j$, this sum also equals
$\langle e_j^{\perp}(\xx) D_{n,k}(\xx;q), e_{\beta _2}(\xx) \cdots e_{\beta _p}(\xx) \rangle$.

We now give a combinatorial interpretation for the right-hand-side of (\ref{e-perp-equation}) 
(with the parameter $m$ there replaced by $k-r$).  
Given $0\le r \le k-1$ and 
$\sigma \in \OP_{n-j,k-r}$, let 
$T$ be a way of adding the $j$ elements of the set $\{n-j+1,\ldots ,n\}$ (hereafter referred to as {\em big} letters) 
to the blocks of
$\sigma$ to form an ordered set partition $\sigma ^{\prime}$ of $n$ with $k$ blocks, 
in such a way that the big letters occur in 
$\text{revword}(\sigma ^{\prime})$ in reverse order $n,n-1,\ldots n-j+1$.  For example, if $\sigma = (246\mid 15\mid 3)$ and $j=3$, one 
choice for $\sigma ^{\prime}$ would be $(2467 \mid  8 \mid  15 \mid  9 \mid  3)$, since then $\text{revword}(\sigma)= 391825467$.  Note that
the conditions on the big letters force each block to have at most one big letter, and that $T$ must have $r$ blocks consisting
of a single big letter.

Call the letters from $\{1,2,\ldots ,n-j\}$ {\em small} letters.  Furthermore let big letters which are minimal (non-minimal) in 
their block be denoted by $minb$ ($nminb$) letters, respectively, and
small letters which are minimal (non-minimal) in 
their block be denoted by $mins$ ($nmins$) letters, respectively.  Recalling the conditions (\ref{coinvdef}), in $T$
each $minb$ letter forms a coinversion pair with each 
$nmins$ letter, giving $r(n-j-(k-r))$ coinversions of this type.  Each of the $minb$ letters also form a coinversion pair with
each of the $nminb$ letters (giving $r(j-r)$ pairs), and with each of the $minb$ letters to its left (yielding
${r \choose 2}$ pairs).   In addition, each $minb$ letter forms
a coinversion pair with each $mins$ letter to its left;  as we sum over all ways of 
interleaving the $r$ blocks of $minb$ letters with the other $k-r$ blocks, leaving everything else in $T$ fixed, we generate a 
factor of ${k \brack r}_q$ from these pairs. 

Next note that each $nminb$ letter forms a coinversion pair with each $mins$ letter to its left.  It follows that if we ignore the $r$ blocks of
$minb$ letters, and sum over all ways to distribute the $nminb$ letters amongst the $k-r$ blocks with a $mins$ letter,
we generate a factor of $q^{{j-r \choose 2}}{k-r \brack j-r}_q$.   Since
\begin{equation*}
{r \choose 2} + r(j-r) + {j-r \choose 2} = {j \choose 2},
\end{equation*}
we have
\begin{align}
\label {rrr}
\sum_{T} q^{\text{coinv}(T)} =
q^{\text{coinv}(\sigma)} q^{r(n-j-k+r)+{j \choose 2}} {k \brack r}_q {k-r \brack j-r}_q\\
\notag
= q^{\text{coinv}(\sigma)} q^{r(n-j-k+r)+{j \choose 2}} {k \brack j}_q {j \brack r}_q.
\end{align} 
If we sum (\ref{rrr}) over all $\sigma \in \OP_{n-j,k-r}$, multiplying by the appropriate Gessel fundamental,
and also sum over $r$, we get the right hand side of (\ref{e-perp-equation}).  
Taking the scalar product of the resulting symmetric function 
with $e_{\beta_2}(\xx) \cdots e_{\beta _p}(\xx)$, again using the results on superization and (\ref{snew}), we get the
scalar product of the left hand side of (\ref{e-perp-equation}) with 
$e_{\beta_2}(\xx) \cdots e_{\beta _p}(\xx)$.
\end{proof}

\section{Hilbert series and generalized Artin basis}
\label{Hilbert series}

\subsection{The point sets $Y_{n,k}$}
Our  strategy for determining the Hilbert and Frobenius series of $R_{n,k}$ is inspired from the 
work of Garsia and Procesi \cite{GP}.  Garsia and Procesi used this method to determine
the Hilbert and Frobenius series of the cohomology of the Springer fiber $R_{\lambda}$
for $\lambda \vdash n$.  We  recall the method,  then apply it to our situation.

Let $Y \subset \QQ^n$ be any finite set of points and let $\II(Y) \subseteq \QQ[\xx_n]$ be the ideal of 
polynomials which vanish on $Y$:
\begin{equation}
\II(Y) = \{ f \in \QQ[\xx_n] \,:\, f(\yy) = 0 \text{ for all } \yy \in Y\}.
\end{equation}
We can think of the quotient $\QQ[Y] = \frac{\QQ[\xx_n]}{\II(Y)}$ as the ring of polynomial
functions $Y \longrightarrow \QQ$, and in particular we have
\begin{equation}
\dim \left( \frac{\QQ[\xx_n]}{\II(Y)} \right) = |Y|.
\end{equation}
Moreover, if $Y$ is stable under the action of any finite subgroup $G \subseteq GL_n(\QQ)$, we have
an isomorphism of $G$-modules 
\begin{equation}
 \frac{\QQ[\xx_n]}{\II(Y)}  \cong_G \QQ[Y].
\end{equation}

The ideal $\II(Y)$ is almost never homogeneous.  To produce a homogeneous ideal,
let $\tau(f)$ be the top degree component of any nonzero
polynomial $f \in \QQ[\xx_n]$.  That is, if 
$f = f_d + f_{d-1} + \cdots + f_0$ with each $f_i$ homogeneous of degree $i$ and $f_d \neq 0$,
we define $\tau(f) := f_d$.  Let $\TT(Y) \subseteq \QQ[\xx_n]$ be the ideal generated by the top degree
components of all of the polynomials in $\II(Y)$:
\begin{equation}
\TT(Y) := \langle \tau(f) \,:\, f \in \II(Y) - \{0\} \rangle.
\end{equation}

The ideal $\TT(Y)$ is homogeneous by definition.  It can also be shown (see \cite{GP}) that any 
set of homogeneous polynomials in $\QQ[\xx_n]$ which descends to a basis of 
$\frac{\QQ[\xx_n]}{\II(Y)}$ also descends to a basis of
$\frac{\QQ[\xx_n]}{\TT(Y)}$.  In particular, we have
\begin{equation}
\dim \left( \frac{\QQ[\xx_n]}{\TT(Y)} \right) = \dim \left( \frac{\QQ[\xx_n]}{\II(Y)} \right) = |Y|.
\end{equation}
Moreover, if $Y$ is stable under the action of a finite subgroup $G \subseteq GL_n(\QQ)$,
we have an isomorphism of $G$-modules
\begin{equation}
\frac{\QQ[\xx_n]}{\TT(Y)}  \cong_G \frac{\QQ[\xx_n]}{\II(Y)}  \cong_G \QQ[Y].
\end{equation}

Our strategy is as follows.
\begin{enumerate}
\item  Find a finite point set $Y_{n,k} \subset \QQ^n$ which is stable under the action
of $\symm_n$ such that there is a $\symm_n$-equivariant bijection between $Y_{n,k}$
and $\OP_{n,k}$.
\item  Prove that $I_{n,k} \subseteq \TT(Y_{n,k})$ by showing that the generators of
$I_{n,k}$ arise as
top degree components 
of  polynomials $f \in \II(Y_{n,k})$ vanishing on $Y_{n,k}$.
\item  Prove that 
\begin{equation*}
\dim(R_{n,k}) = \dim \left( \frac{\QQ[\xx_n]}{I_{n,k}} \right) \leq |\OP_{n,k}| =
\dim \left( \frac{\QQ[\xx_n]}{\TT(Y_{n,k})}  \right)
\end{equation*}
and use the relation $I_{n,k} \subseteq \TT(Y_{n,k})$ to conclude that
$I_{n,k} = \TT(Y_{n,k})$.
\end{enumerate}
The execution of this strategy will show that $\dim(R_{n,k}) = |\OP_{n,k}|$ and 
$R_{n,k} \cong \QQ[\OP_{n,k}]$ as ungraded $\symm_n$-modules.

To achieve Step 1 of our strategy, we introduce the following point sets.

\begin{defn}
Fix distinct rational numbers $\alpha_1, \dots, \alpha_k \in \QQ$.  Let $Y_{n,k} \subset \QQ^n$
be the set of points 
with coordinates occurring in $\{\alpha_1, \dots, \alpha_k\}$ such that each $\alpha_i$ appears
at least once.  In other words,
\begin{equation*}
Y_{n,k} := \{(y_1, \dots, y_n) \in \QQ^n \,:\, \{ \alpha_1, \dots , \alpha_k \} = \{y_1, \dots, y_n \} \}.
\end{equation*}
\end{defn}

There is a bijection from $Y_{n,k}$ and $\OP_{n,k}$ obtained by sending 
$(y_1, \dots, y_n)$ to $(B_1 | \cdots | B_k)$, where $B_i = \{ j \,:\, y_j = \alpha_i \}$.  For example, 
we have
\begin{equation*}
(\alpha_3, \alpha_2, \alpha_2, \alpha_1, \alpha_3) \leftrightarrow (4 \mid 23 \mid 15).
\end{equation*}
It is evident that this bijection is $\symm_n$-equivariant, which implies
\begin{equation}
\QQ[Y_{n,k}] \cong_{\symm_n} \QQ[\OP_{n,k}].
\end{equation}

Lemma~\ref{vanishing-lemma} allows us to achieve Step 2 of our strategy right away.

\begin{lemma}
\label{I-contained-in-T}
We have $I_{n,k} \subseteq \TT(Y_{n,k})$.
\end{lemma}

\begin{proof}
For $n-k+1 \leq r \leq n$, Lemma~\ref{vanishing-lemma} guarantees that
\begin{equation}
\sum_{j = 0}^r (-1)^j h_j(\alpha_1, \dots, \alpha_k) e_{r-j}(\xx_n) \in \II(Y_{n,k}).
\end{equation}
Taking the top degree component, we get
$e_r(\xx_n) \in \TT(Y_{n,k})$.
Moreover, since the coordinates of points in $Y_{n,k}$ lie in the set $\{\alpha_1, \dots, \alpha_k\}$,
for $1 \leq i \leq n$ we have
\begin{equation}
(x_i - \alpha_1) \cdots (x_i - \alpha_k) \in \II(Y_{n,k}).
\end{equation}
Taking the top degree component, we get
$x_i^k \in \TT(Y_{n,k})$.  We conclude that $I_{n,k} \subseteq \TT(Y_{n,k})$.
\end{proof}

Step 3 of our strategy will require more work.  
We aim to prove the dimension inequality
$\dim \left( \frac{\QQ[\xx_n]}{I_{n,k}} \right) \leq |\OP_{n,k}|$.
To do so, we will use Gr\"obner theory.

\subsection{Skip and nonskip monomials}
Let $<$ be the lexicographic monomial
 order.  We know that 
\begin{equation}
\dim \left( \frac{\QQ[\xx_n]}{I_{n,k}} \right) =
\dim \left( \frac{\QQ[\xx_n]}{\initial_<(I_{n,k})} \right).
\end{equation}
Moreover, we know that a basis for $\frac{\QQ[\xx_n]}{\initial_<(I_{n,k})}$ is given by
\begin{equation}
\{ m + I_{n,k} \,:\, \text{$m \in \QQ[\xx_n]$ is a monomial and $\initial_<(f) \nmid m$ for all
$f \in I_{n,k} - \{0\}$} \}.
\end{equation}
Our aim is to bound the size of this set above by $|\OP_{n,k}|$.  To do this, we will
calculate $\initial_<(f)$ for some strategically chosen $f \in I_{n,k}$.

\begin{lemma}
\label{skip-monomials-in-initial}
Let $k \leq n$.  For any $S \subseteq [n]$ of size $|S| = n-k+1$, the reverse Demazure character
$\kappa_{\gamma(S)^*}(\xx_n^*)$ satisfies
$\kappa_{\gamma(S)^*}(\xx_n^*) \in I_{n,k}$.  In particular, we have 
\begin{equation*}
\xx(S) \in \initial_<(I_{n,k}).
\end{equation*}
Moreover, for $1 \leq i \leq n$ we have
\begin{equation*}
x_i^k \in \initial_<(I_{n,k}).
\end{equation*}
\end{lemma}

\begin{proof}
Lemma~\ref{demazure-identity} 
(and in particular Equation~\ref{magical-demazure-equations-friend})
implies that $\kappa_{\gamma(S)^*}(\xx_n^*) \in I_{n,k}$.
By Lemma~\ref{reduced-demazure-lemma}, taking the lexicographic leading term
shows that
\begin{equation*}
\initial_<(\kappa_{\gamma(S)^*}(\xx_n^*)) = \xx(S) \in \initial_<(I_{n,k}).
\end{equation*}
Since $x_i^k \in I_{n,k}$, we have $x_i^k \in \initial_<(I_{n,k}).$
\end{proof}

It will turn out that the leading monomials furnished by
Lemma~\ref{skip-monomials-in-initial} are all we need.
We give the collection of monomials which are not divisible by any of these monomials
a name.

\begin{defn}
\label{nonskip-definition}
Let $k \leq n$.  A monomial $m \in \QQ[\xx_n]$ is 
{\em $(n,k)$-nonskip} if 
\begin{enumerate}
\item  $\xx(S) \nmid m$ for all $S \subseteq [n]$ with $|S| = n-k+1$ and
\item  $x_i^k \nmid m$ for all $1 \leq i \leq n$.
\end{enumerate}
Let $\MMM_{n,k}$ denote the set of all $(n,k)$-nonskip monomials in $\QQ[\xx_n]$.
\end{defn}

There is some redundancy in Definition~\ref{nonskip-definition}.  If $n \in S$ and 
$|S| = n-k+1$, the skip monomial $\xx(S)$ contains the variable power $x_n^k$.
It is therefore enough to consider those subsets $S$ with $|S| = n-k+1$
and $n \notin S$.

We will prove that $|\MMM_{n,k}| = |\OP_{n,k}|$.  This will imply that 
$R_{n,k} \cong \QQ[\OP_{n,k}]$.
We will also show that the degree statistic on $\MMM_{n,k}$ is equidistributed with $\coinv$
on $\OP_{n,k}$.  This will imply that
$\Hilb(R_{n,k}; q) = \rev_q( [k]!_q \cdot \Stir_q(n,k) )$.  Finally, we will show 
that $\MMM_{n,k} = \AAA_{n,k}$, proving that generalized Artin monomials descend to a basis
of $R_{n,k}$.

The proofs of the claims in the last paragraph will involve a combinatorial analysis of skip
and nonskip monomials.  A key observation regarding skip monomials and divisibility is as follows.

\begin{lemma}
\label{skip-monomial-union}
Let $m \in \QQ[\xx_n]$ be a monomial and let $S, T \subseteq [n]$.  If 
$\xx(S) \mid m$ and $\xx(T) \mid m$, then $\xx(S \cup T) \mid m$.
\end{lemma}

\begin{proof}
Let $i \in S$.  Then
\begin{align*}
\text{exponent of $x_i$ in $\xx(S \cup T)$} 
&= i - | (S \cup T) \cap \{1, 2, \dots, i-1\} | \\
&\leq i - |S \cap \{1, 2, \dots, i-1\}| \\
&= \text{exponent of $x_i$ in $\xx(S)$.} 
\end{align*}
Similarly, if $j \in T$ then the exponent of $x_j$ in $\xx(S \cup T)$ is $\leq$ the exponent of 
$x_j$ in $\xx(T)$.  The lemma follows.
\end{proof}

An immediate useful consequence of Lemma~\ref{skip-monomial-union} is:

\begin{lemma}
\label{canonical-skip-union}
Let $m \in \QQ[\xx_n]$ be a monomial and let $r$ be a positive integer such that 
$\xx(S) \mid m$ for some $S \subseteq [n]$ with $|S| = r$ but there does not exist 
$T \subseteq [n]$ with $|T| > r$ such that $\xx(T) \mid m$.

Then there exists a {\em unique} $S \subseteq [n]$ with $|S| = r$ such 
that $\xx(S) \mid m$.
\end{lemma}

\begin{proof}
If there were two such sets $S \neq S'$, Lemma~\ref{skip-monomial-union} would force
$\xx(S \cup S') \mid m$, contrary to the assumption on $r$.
\end{proof}

For any subset $S = \{i_1, \dots , i_r\} \subseteq [n]$, let 
$\mm(S) := x_{i_1} \cdots x_{i_r} \in \QQ[\xx_n]$ be the product of the corresponding variables.
For example, we have $\mm(245) = x_2 x_4 x_5$.
The following existence-uniqueness type result will be crucial in bijecting 
$\MMM_{n,k}$ with $\OP_{n,k}$.

\begin{lemma}
\label{canonical-skip}
Let $k \leq n$ and let $m \in \MMM_{n,k}$.
There exists a unique
set $S \subseteq [n]$ with $|S| = n-k$ such that
\begin{enumerate}
\item  $\xx(S) \mid (\mm(S) \cdot m)$, and
\item  $\xx(U) \nmid (\mm(S) \cdot m)$ for all $U \subseteq [n]$ with $|U| = n-k+1$.
\end{enumerate}
\end{lemma}

\begin{proof}
We start with uniqueness.  Suppose $S = \{s_1 < \cdots < s_{n-k} \}$ and
$T = \{t_1 < \cdots < t_{n-k} \}$ were two distinct such sets.  Let $r$ be such that
$s_1 = t_1, \dots , s_{r-1} = t_{r-1},$ and $s_r \neq t_r$.  Without loss of generality we have
$s_r < t_r$.  Let $U = \{s_1 < \cdots < s_r < t_r < t_{r+1} < \cdots < t_{n-k} \}$.  Since
$\xx(S) \mid (\mm(S) \cdot m)$ and $\xx(T) \mid (\mm(T) \cdot m)$, it follows that
$\xx(U) \mid (\mm(S) \cdot m)$ and $|U| = n-k+1$.

To prove existence, consider the collection $\CCC$ of sets
\begin{equation}
\CCC := \{ S \subseteq [n] \,:\, |S| = n-k \text{ and } \xx(S) \mid (\mm(S) \cdot m) \}.
\end{equation}
The collection $\CCC$ is nonempty; indeed, we have
$\{1, 2, \dots, n-k\} \in \CCC$.  Let $S_0 \in \CCC$ be the lexicographically {\em final}
subset in $\CCC$.  We argue that $\mm(S_0) \cdot m$ satisfies Condition 2 in the statement
of the lemma, thus finishing the proof of the lemma.

Let $T \subseteq [n]$ have size $|T| = n-k+1$.  Working towards a contradiction, suppose
$\xx(T) \mid (\mm(S_0) \cdot m)$.  If there were an element $t \in T$ with $t < \min(S_0)$, then
the relation $\xx(S_0) \mid (\mm(S_0) \cdot m)$ would imply
$\xx(\{t\} \cup S_0) \mid m$, a contradiction.  Since $|T| > |S|$, there exists an element $t_0 \in T - S_0$
with $t_0 > \min(S_0)$.  Write 
$S_0 \cup \{t_0\} = \{s_1 < \cdots < s_r < t_0 < s_{r+1} < \cdots < s_{n-k} \}$.  If we define
$S'_0$ to be the set 
\begin{equation*}
S'_0 := \{s_1 < \cdots s_{r-1} < t_0 < s_{r+1} < \cdots < s_{n-k} \},
\end{equation*}
then $S'_0$ would come after $S_0$ in lexicographic order but we would have
$S'_0 \in \CCC$, contradicting our choice of $S_0$.
\end{proof}

To see how Lemma~\ref{canonical-skip} and its proof work, consider $(n,k) = (7,4)$.
It can be checked that $m = x_1^3 x_4^2 x_5^3 x_6^3 \in \MMM_{7,4}$.
We have
\begin{equation*}
\CCC = \{ S \subseteq [7] \,:\, |S| = 7-4 \text{ and } \xx(S) \mid (\mm(S) \cdot m) \} 
= \{123, 124, 125, 126,  145, 146, 156\}.
\end{equation*}
However, we also have 
\begin{center}
$\begin{array}{ccc}
\xx(1234) \mid (\mm(123) \cdot m), & \xx(1245) \mid (\mm(124) \cdot m), & \xx(1245) \mid (\mm(125) \cdot m), \\
\xx(1245) \mid (\mm(126) \cdot m), & \xx(1456) \mid (\mm(145) \cdot m), & \xx(1456) \mid (\mm(146) \cdot m).
\end{array}$
\end{center}
If $U \subseteq [7]$ and $|U| = 4$, then $\xx(U) \nmid (\mm(156) \cdot m)$, so that $156$ is the set $S$ guaranteed
by Lemma~\ref{canonical-skip}.  Observe that $S$ is the lexicographically final set in $\CCC$
and that $\mm(156) \cdot m = x_1^4 x_4^2 x_5^4 x_6^4 \notin \MMM_{7,4}$ (because there are variable powers
$\geq 4$).

\subsection{Nonskip monomials, ordered set partitions, and stars and bars}  Let $P_{n,k}(q)$ be the 
generating function for the $\coinv$ statistic on $\OP_{n,k}$ and let $M_{n,k}(q)$ be the 
generating function for the degree statistic on $\MMM_{n,k}$:
\begin{equation}
P_{n,k}(q) := \sum_{\sigma \in \OP_{n,k}} q^{\coinv(\sigma)}, \hspace{0.5in}
M_{n,k}(q) := \sum_{m \in \MMM_{n,k}} q^{\deg(m)}.
\end{equation}
We will prove that $P_{n,k}(q) = M_{n,k}(q)$ by 
finding a recursive bijection 
$\Psi: \OP_{n,k} \longrightarrow \MMM_{n,k}$ with the property that
for any $\sigma \in \OP_{n,k}$ we have
$\coinv(\sigma) = \deg(\Psi(\sigma))$.

At the level of ordered set partitions, the recursion on which $\Psi$ is based is as follows.
Suppose we are given an ordered set partition $\sigma \in \OP_{n,k}$ and we wish to 
build an ordered set partition of size $n+1$.  There are two ways we could insert $n+1$ into
$\sigma = (B_1 \mid \cdots \mid B_k)$:
\begin{enumerate}
\item  insert $n+1$ into one of the $k$ blocks $B_1, \dots, B_k$, or
\item  insert  $n+1$ as a singleton block in one of the $k+1$ spaces on either side of $B_1, \dots, B_k$.
\end{enumerate}
Following Remmel and Wilson \cite{RW}, we will call an insertion of Type 1 a {\em star insertion}
and an insertion of Type 2 a {\em bar insertion}.  
Performing a star insertion results in an element of $\OP_{n+1,k}$ 
(thus adding a star to the star model of $\sigma$)
whereas performing a 
bar insertion results in an element of $\OP_{n+1,k+1}$
(thus adding a bar to the bar model of $\sigma$).

For example, consider $\sigma = (5 \mid 1 4 6 \mid 2 3  \mid 7) \in \OP_{7,4}$.  The 
four possible star insertions
of $8$, together with their effect on $\inv$, are as follows.

\begin{center}
$\begin{array}{cccc}
(5 \mid 146 \mid 23 \mid 7 {\bf 8}) & 
(5 \mid 146 \mid 23 {\bf 8} \mid 7 ) &
(5 \mid 146 {\bf 8} \mid 23  \mid 7 ) &
(5 {\bf 8} \mid 146 \mid 23  \mid 7 )  \\
\inv + 0 & \inv + 1 & \inv + 2 & \inv + 3
\end{array}$
\end{center}
Rather than $\inv$ itself, we are  interested in the complementary statistic $\coinv$.  In our example, 
the effect of star insertion on $\coinv$ is as follows.  

\begin{center}
$\begin{array}{cccc}
(5 \mid 146 \mid 23 \mid 7 {\bf 8}) & 
(5 \mid 146 \mid 23 {\bf 8} \mid 7 ) &
(5 \mid 146 {\bf 8} \mid 23  \mid 7 ) &
(5 {\bf 8} \mid 146 \mid 23  \mid 7 )  \\
\coinv + 3 & \coinv + 2 & \coinv + 1 & \coinv + 0
\end{array}$
\end{center} 
In general,
if $\widehat{\sigma}$ is the result of star inserting 
$n$ into the block $B_i$ of $\sigma = (B_1 \mid \cdots \mid B_k) \in \OP_{n,k}$, we have
\begin{equation}
\label{star-insertion-equation-osp}
\coinv(\widehat{\sigma}) = \coinv(\sigma) + (i-1).
\end{equation}

On the other hand, consider using bar insertion to insert $8$ into
$\sigma = (5 \mid 1 4 6 \mid 2 3  \mid 7) \in \OP_{7,4}$.  The five possible ways to do this, together
with their effect on $\inv$, are as follows.

\begin{center}
$\begin{array}{ccccc}
(5 \mid 146 \mid 23 \mid 7 \mid {\bf 8}) & 
(5 \mid 146 \mid 23 \mid {\bf 8} \mid 7 ) &
(5 \mid 146 \mid {\bf 8} \mid 23  \mid 7 ) &
(5 \mid {\bf 8} \mid 146 \mid 23  \mid 7 ) &
({\bf 8} \mid 5  \mid 146 \mid 23  \mid 7 )  \\
\inv + 0 & \inv + 1 & \inv + 2 & \inv + 3 & \inv + 5
\end{array}$
\end{center}
The effect of these bar insertions on $\coinv$ is shown below.

\begin{center}
$\begin{array}{ccccc}
(5 \mid 146 \mid 23 \mid 7 \mid {\bf 8}) & 
(5 \mid 146 \mid 23 \mid {\bf 8} \mid 7 ) &
(5 \mid 146 \mid {\bf 8} \mid 23  \mid 7 ) &
(5 \mid {\bf 8} \mid 146 \mid 23  \mid 7 ) &
({\bf 8} \mid 5  \mid 146 \mid 23  \mid 7 )  \\
\coinv + 7 & \coinv + 6 & \coinv + 5 & \coinv + 4 & \coinv + 3
\end{array}$
\end{center}
In general, if $\widehat{\sigma}$ is the result of bar inserting the singleton block
$\{n\}$ in the space between
$B_{i-1}$ and $B_i$ of $\sigma = (B_1 \mid \cdots \mid B_k) \in \OP_{n,k}$, we have
\begin{equation}
\label{bar-insertion-equation-osp}
\coinv(\widehat{\sigma}) = \coinv(\sigma) + (n-k) + (i-1).
\end{equation}

Equations~\ref{star-insertion-equation-osp} and \ref{bar-insertion-equation-osp}
imply that the generating function $P_{n,k}(q)$ of $\coinv$ on $\OP_{n,k}$ satisfies the recursion
\begin{equation}
\label{p-recursion-equation}
P_{n,k}(q) = [k]_q \cdot (P_{n-1,k}(q) + q^{n-k} P_{n-1,k-1}(q)).
\end{equation}
Together with the initial conditions
\begin{equation}
\label{p-initial-condition-equation}
 \begin{cases}
P_{1,1}(q) = 1 \\
P_{n,k}(q) = 0 & n < k,
\end{cases}
\end{equation}
this determines $P_{n,k}(q)$ completely.
The  idea for constructing our bijection 
$\Psi: \OP_{n,k} \rightarrow \MMM_{n,k}$ is to
show combinatorially that the degree generating function $M_{n,k}(q)$ on $\MMM_{n,k}$
also satisfies the recursion of Equation~\ref{p-recursion-equation}.

The base case of our map $\Psi$ is given by the unique function
$\Psi: \OP_{1,1} \rightarrow \MMM_{1,1}$ mapping $(1) \mapsto 1$.

In general, suppose $\sigma = (B_1 \mid \cdots \mid B_k) \in \OP_{n,k}$ 
and let $\overline{\sigma}$ be the ordered set partition
of size $n-1$
obtained by erasing $n$ from $\sigma$, as well as the block containing $n$ if that block is a singleton.
Then $\overline{\sigma} \in \OP_{n-1,k}$ or  $\overline{\sigma} \in \OP_{n-1,k-1}$ 
according to whether $\{n\}$ is a singleton block
of $\sigma$.

We inductively assume that the portions of the function $\Psi$ given by
$\Psi: \OP_{n-1,k} \rightarrow \MMM_{n-1,k}$ and
$\Psi: \OP_{n-1,k-1} \rightarrow \MMM_{n-1,k-1}$ have already been defined
(so that in particular $\Psi(\overline{\sigma})$ is defined).  We define $\Psi(\sigma)$ as follows,
according to whether $\overline{\sigma} \in \OP_{n-1,k}$ or $\overline{\sigma} \in \OP_{n-1,k-1}$.

If $\overline{\sigma} \in \OP_{n-1,k}$, then $\{n\}$ is not a singleton block of 
$\sigma = (B_1 \mid \cdots \mid B_k)$ and $\sigma$ arises from $\overline{\sigma}$ from
star insertion. 
There exists $0 \leq i \leq k-1$ such that $n \in B_{i+1}$.
Define
\begin{equation}
\Psi(\sigma) := \Psi(\overline{\sigma}) \cdot x_n^i.
\end{equation}

If $\overline{\sigma} \in \OP_{n-1,k-1}$, then $\{n\}$ is a singleton block of $\sigma$, so that
$\sigma$ arises from $\overline{\sigma}$ by bar insertion.
We have $\Psi(\overline{\sigma}) \in \MMM_{n-1,k-1}$.
Invoking Lemma~\ref{canonical-skip}, let $S \subseteq [n-1]$ be the unique subset 
with $|S| = n-k$ such that $\xx(S) \mid (\Psi(\overline{\sigma})  \cdot \mm(S))$ but 
$\xx(U) \nmid (\Psi(\overline{\sigma})  \cdot \mm(S))$ for all $U \subseteq [n-1]$ with
 $|U| = n-k+1$.  Since $\{n\}$ is a singleton block of $\sigma$, there exists 
 $0 \leq i \leq k-1$ such that $\{n\} = B_{i+1}$.  Define 
 \begin{equation}
\Psi(\sigma) := \Psi(\overline{\sigma}) \cdot \mm(S) \cdot x_n^i.
\end{equation}

The map $\Psi$ is  best understood with an example.  Let 
$\sigma \in \OP_{8,5}$ be the ordered set partition
$\sigma = (5 \mid 146 \mid 8 \mid 23 \mid 7 )$.   We calculate $\Psi(\sigma)$ by starting 
with $\Psi: (1) \mapsto 1 \in \MMM_{1,1}$ and repeatedly inserting larger letters to build up 
to $\sigma$.

\begin{small}
\begin{center} 
$\begin{array}{c | c | c | c | c}
\sigma & \text{star/bar insertion?}  &S  &i  &\Psi(\sigma) \\ \hline
({\bf 1}) \in \OP_{1,1} & & & & 1 \in \MMM_{1,1}  \\
(1 \mid {\bf 2}) \in \OP_{2,2} & \text{bar} & \emptyset & 1 & 1 \cdot \mm(\emptyset) \cdot x_2^1 = x_2 \in \MMM_{2,2} \\
(1 \mid 2 {\bf 3}) \in \OP_{3,2} & \text{star} &  & 1  &   x_2 \cdot x_3^1 = x_2 x_3 \in \MMM_{3,2} \\
(1 {\bf 4} \mid 2 3) \in \OP_{4,2} & \text{star} &  & 0  &  x_2 x_3 \cdot x_4^0 = x_2 x_3 \in \MMM_{4,2} \\
({\bf 5} \mid 1  4 \mid 2 3) \in \OP_{5,3} & \text{bar} & 23 & 0  &  x_2 x_3 \cdot \mm(23) \cdot x_5^0 = 
x_2^2 x_3^2 \in \MMM_{5,3}   \\
( 5 \mid 14{\bf 6} \mid 2 3) \in \OP_{6,3} & \text{star} &  & 1  &   x_2^2 x_3^2 \cdot x_6^1 = x_2^2 x_3^2 x_6 \in \MMM_{6,3} \\
( 5 \mid 146 \mid 23 \mid {\bf 7}) \in \OP_{7,4} & \text{bar} & 123 & 3  &   x_2^2 x_3^2 x_6 \cdot \mm(123) \cdot x_7^3 =
x_1 x_2^3 x_3^3 x_6 x_7^3 \in \MMM_{7,4}  \\
( 5 \mid 146 \mid {\bf 8} \mid 23 \mid 7) \in \OP_{8,5} & \text{bar} & 123 & 2  &   
x_1 x_2^3 x_3^3 x_6 x_7^3 \cdot \mm(123) \cdot x_8^2 = x_1^2 x_2^4 x_3^4 x_6 x_7^3 x_8^2 \in \MMM_{8,5} 
\end{array}$
\end{center}
\end{small}
We conclude that $\Psi: (5 \mid 146 \mid 8 \mid 23 \mid 7 ) \mapsto x_1^2 x_2^4 x_3^4 x_6 x_7^3 x_8^2$.

\begin{lemma}
\label{psi-well-defined}
The above procedure gives a well-defined function $\Psi: \OP_{n,k} \rightarrow \MMM_{n,k}$.
\end{lemma}

\begin{proof}
This is certainly true for $(n,k) = (1,1)$.  Fix indices $k \leq n$ and assume that we have
well-defined functions $\Psi: \OP_{n-1,k} \rightarrow \MMM_{n-1,k}$
and $\Psi: \OP_{n-1,k-1} \rightarrow \MMM_{n-1,k-1}$.  Consider an ordered set partition 
$\sigma = (B_1 \mid \cdots \mid B_k) \in \OP_{n,k}$.
We need to check that the above procedure defining $\Psi(\sigma)$ actually yields an monomial
in $\MMM_{n,k}$.

{\bf Case 1:}  {\em $\{n\}$ is not a singleton block of $\sigma$.} 

In this case $\Psi(\sigma) = \Psi(\overline{\sigma}) \cdot x_n^i$,
where $n \in B_{i+1}$.  In particular, we have $x_n^k \nmid \Psi(\sigma)$.
Since $\Psi(\overline{\sigma}) \in \MMM_{n-1,k}$, we also know that 
$x_j^k \nmid \Psi(\sigma)$ for $1 \leq j \leq n-1$.  

Let $T \subseteq [n]$ satisfy $|T| = n-k+1$ and 
suppose $\xx(T) \mid \Psi(\sigma)$.  Let $T' := T  -  \{ \max(T) \}$, so that
$T' \subseteq [n-1]$ and $|T'| = n-k$.
Since $T' \subseteq [n-1]$ and 
$\xx(T) \mid \Psi(\sigma)$, we have
$\xx(T') \mid \Psi(\overline{\sigma})$, contradicting the assumption that
$\Psi(\overline{\sigma}) \in \MMM_{n-1,k}$.   We conclude that $\xx(T) \nmid \Psi(\sigma)$,
so that $\Psi(\sigma) \in \MMM_{n,k}$.

{\bf Case 2:}  {\em $\{n\}$ is a singleton block of $\sigma$.}

In this case $\Psi(\sigma) = \Psi(\overline{\sigma}) \cdot \mm(S) \cdot x_n^i$,
where $B_{i+ 1} = \{n\}$ and $S \subseteq [n-1]$ is as above.  By construction
$x_n^k \nmid \Psi(\sigma)$.  Since $\Psi(\overline{\sigma}) \in \MMM_{n-1,k-1}$, we know that
$x_j^{k-1} \nmid \Psi(\overline{\sigma})$ for all $1 \leq j \leq n-1$, so that 
$x_j^k \nmid \Psi(\sigma)$.

Let $T \subseteq [n]$ satisfy $|T| = n-k+1$ and 
suppose $\xx(T) \mid \Psi(\sigma)$.  
Our choice of $S$ (and Lemma~\ref{canonical-skip}) guarantee that
$T \not\subseteq [n-1]$.  
On the other hand, if $n \in T$, the variable power $x_n^k$ would appear in $\xx(T)$, implying the 
contradiction $x_n^k \mid \Psi(\sigma)$.
We conclude that
$\xx(T) \nmid \Psi(\sigma)$, so that
$\Psi(\sigma) \in \MMM_{n,k}$.
\end{proof}

The construction of $\Psi$ suggests the construction of its inverse.  Given an 
$(n,k)$-nonskip monomial $m = x_1^{i_1} \cdots x_{n-1}^{i_{n-1}} x_n^{i} \in \MMM_{n,k}$,
consider the monomial
$m' = x_1^{i_1} \cdots x_{n-1}^{i_{n-1}}$ obtained by erasing the last variable $x_n$.
Either $m'$ is divisible by some skip monomial $\xx(S)$ with $|S| = n-k$ or it is not.  If so,
then $S$ is unique and we have $\frac{m'}{\mm(S)} \in \MMM_{n-1,k-1}$.
If not, then we have $m' \in \MMM_{n-1,k}$.  Either way, we recursively have a size $n-1$ ordered 
set partition.  To get a size $n$ ordered set partition, use this branching structure to determine
whether to insert $n$ as a singleton block and use the exponent $i$  of $x_n$
to determine where to insert $n$ from left to right.

\begin{theorem}
\label{psi-is-bijection}
The function $\Psi: \OP_{n,k} \rightarrow \MMM_{n,k}$ is a bijection with the property that
$\coinv(\sigma) = \deg(\Psi(\sigma))$ for all $\sigma \in \OP_{n,k}$.

In particular, we have $P_{n,k}(q) = M_{n,k}(q)$ and $|\OP_{n,k}| = |\MMM_{n,k}|$.
\end{theorem}

\begin{proof}
We recursively define the inverse $\Phi$ to the function $\Psi$.  When $(n,k) = (1,1)$, there is 
only one choice: let $\Phi: \MMM_{n,k} \rightarrow \OP_{n,k}$ be the unique assignment
$\Phi: 1 \mapsto (1)$.

In general, fix positive $k \leq n$ and assume inductively that 
$\Phi: \MMM_{n-1,k} \rightarrow \OP_{n-1,k}$ and 
$\Phi: \MMM_{n-1,k-1} \rightarrow \OP_{n-1,k-1}$ have already been defined.
Let $m = x_1^{i_1} \cdots x_{n-1}^{i_{n-1}} x_n^i \in \MMM_{n,k}$ and define
$m' := x_1^{i_1} \cdots x_{n-1}^{i_{n-1}}$.  Either $m' \in \MMM_{n-1,k}$
or $m' \notin \MMM_{n-1,k}$.

If $m' \in \MMM_{n-1,k}$, then $\Phi(m') = (B_1 \mid \cdots \mid B_k)$ is an
ordered set partition of size $n-1$.  To define $\Phi(m) = \Phi(m' \cdot x_n^i)$, we 
simply star insert $n$ to the block $B_{i+1}$:
\begin{equation}
\Phi(m) = \Phi(m' \cdot x_n^i) := (B_1 \mid \cdots \mid B_{i+1} \cup \{n\} \mid \cdots \mid B_k) \in \OP_{n,k}.
\end{equation}

If $m' \notin \MMM_{n-1,k}$, then there exists a subset $S \subseteq [n-1]$ such that 
$|S| = n-k$ and $\xx(S) \mid m'$.  Since $m = m' \cdot x_n^i \in \MMM_{n,k}$,
Lemma~\ref{canonical-skip-union} guarantees that the set $S$ is unique.  Since 
$\xx(S) \mid m'$, we have $\mm(S) \mid m'$ and the quotient $\frac{m'}{\mm(S)}$ is a monomial.

{\bf Claim:}  $\frac{m'}{\mm(S)} \in \MMM_{n-1,k-1}$.  

Since $m \in \MMM_{n,k}$, we know that $\xx(T) \nmid \frac{m'}{\mm(S)}$ for all $T \subseteq [n-1]$
with $|T| = n-k+1$.  Let $1 \leq j \leq n-1$.  We need to show $x_j^{k-1} \nmid \frac{m'}{\mm(S)}$.
If $j \in S$ this is immediate from the fact that $x_j^k \nmid m'$.  If $j \notin S$ and 
$x_j^{k-1} \mid \frac{m'}{\mm(S)}$, then $x_j^{k-1} \mid m'$ and
$\xx(S \cup \{j\}) \mid m'$, a contradiction to the assumption
that $m = m' \cdot x_n^i \in \MMM_{n,k}$.   This finishes the proof of the Claim.

By the Claim, we recursively have an ordered set partition
$\Phi \left( \frac{m'}{\mm(S)} \right) = (B_1 \mid \cdots \mid B_{k-1})$ of size $n-1$.
To define $\Phi(m) = \Phi(m' x_n^i)$, we bar insert a singleton block $\{n\}$ 
to the left of
$B_{i+1}$
\begin{equation}
\Phi(m) := (B_1 \mid \cdots \mid B_{i} \mid n \mid B_{i+1} \mid \cdots \mid B_{k-1}) \in \OP_{n,k}.
\end{equation}

As with $\Psi$, the map $\Phi$ is  best understood with an example.  Consider the monomial
$m = x_1^2 x_2^4 x_3^4 x_6 x_7^3 x_8^2 \in \MMM_{8,5}$.  To calculate 
$\Phi(m) \in \OP_{8,5}$, we write down  the following table.  We show entries in the right hand column in
bold to indicate that no more letters can be added to their block.

\begin{small}
\begin{center}
$\begin{array}{ c | c | c | c | c | c | c | c}
m & m' & (n,k) & m' \in \MMM_{n-1,k}?  & S & \frac{m'}{\mm(S)}  & i & \Phi(m) \\ \hline
x_1^2 x_2^4 x_3^4 x_6 x_7^3 x_8^2 & 
x_1^2 x_2^4 x_3^4 x_6 x_7^3 & (8,5) &
\text{no} & 123 & x_1 x_2^3 x_3^3 x_6 x_7^3 & 2 & 
( \cdot \mid \cdot  \mid {\bf 8} \mid \cdot   \mid \cdot)  \\
x_1 x_2^3 x_3^3 x_6 x_7^3 & 
x_1 x_2^3 x_3^3 x_6 & (7,4) &
\text{no} & 123 &  x_2^2 x_3^2 x_6  & 3 & 
( \cdot \mid \cdot  \mid {\bf 8} \mid \cdot   \mid {\bf 7})  \\
x_2^2 x_3^2 x_6 & 
x_2^2 x_3^2  & (6,3) &
\text{yes} &  &    & 1 & 
( \cdot \mid 6  \mid {\bf 8} \mid \cdot   \mid {\bf 7})  \\
x_2^2 x_3^2 & x_2^2 x_3^2 & (5,3) & \text{no} & 23 & x_2 x_3 & 0 & 
({\bf 5}  \mid 6  \mid {\bf 8} \mid \cdot   \mid {\bf 7}) \\
x_2 x_3 & x_2 x_3 & (4,2) & \text{yes} & & & 0 &
({\bf 5} \mid 46 \mid {\bf 8} \mid \cdot \mid {\bf 7}) \\
x_2 x_3 & x_2 & (3,2) & \text{yes} &  & &1 &
({\bf 5} \mid 46 \mid {\bf 8} \mid 3 \mid {\bf 7}) \\
x_2 & 1 & (2,2) & \text{no} & \emptyset & 1 & 1 &
({\bf 5} \mid 46 \mid {\bf 8} \mid {\bf 23} \mid {\bf 7}) \\
1 & 1 & (1,1) & \text{no} & \emptyset & 1 & 0 &
({\bf 5} \mid {\bf 146} \mid {\bf 8} \mid {\bf 23} \mid {\bf 7}) 
\end{array}$
\end{center}
\end{small}

The rules for proceeding from one row of this table to the next are as follows.  
\begin{itemize}
\item
Define $m$
to be the monomial $m'$ from the above row
(if the answer to the query in the above row was `yes') or 
$\frac{m'}{\mm(S)}$ from the above row
(if the answer to the query in the above row was `no').
\item 
Define the $(n,k)$-column entry to be $(n-1,k)$ from the above row
(if the answer to the query in the above row was `yes') or 
$(n-1,k-1)$ from the above row 
(if the answer to the query in the above row was `no').
\item  Using $(n,k)$ in the current row, define the monomial $m'$ by the 
relation $m = m' \cdot x_n^i$.
\item  Record the value of $i$.
\item  Check whether $m' \in \MMM_{n-1,k}$.
\item  If $m' \notin \MMM_{n-1,k}$, let $S \subseteq [n-1]$ be the unique set of size $|S| = n-k$ such that
$\xx(S) \mid m'$.
\item  If $m' \notin \MMM_{n-1,k}$, calculate $\frac{m'}{\mm(S)}$.
\item  Finally, insert $n$ into the partial ordered set partition $\Phi(m)$ from the above row.
If the answer to the query in the current row was `yes', insert $n$ into the $i^{th}$ unfrozen block
from the left.  If the answer to the query in the current row was `no', insert $n$ into the $i^{th}$ unfrozen 
block from the left and freeze that block.
\end{itemize}
The above table shows that 
$\Phi: x_1^2 x_2^4 x_3^4 x_6 x_7^3 x_8^2 \mapsto (5 \mid 146 \mid 8 \mid 23 \mid 7)$.

We leave it to the reader to check that the recursions $\Phi$ and $\Psi$ undo each other, so that
we have the compositions
$\Phi \circ \Psi = \mathrm{id}_{\OP_{n,k}}$ and
$\Psi \circ \Phi = \mathrm{id}_{\MMM_{n,k}}$.  The definition of $\coinv$ makes it clear that
$\coinv(\sigma) = \deg(\Phi(\sigma))$ for all $\sigma \in \OP_{n,k}$.
\end{proof}

\subsection{Gr\"obner basis and Hilbert series}  We exploit the bijection $\Psi$ of 
Theorem~\ref{psi-is-bijection}.  First, we determine $\Hilb(R_{n,k}; q)$.

\begin{theorem}
\label{hilbert-series-theorem}
Let $k \leq n$ be positive integers.  The Hilbert series $\Hilb(R_{n,k}; q)$ of the graded vector space 
$R_{n,k}$ is the generating function for $\coinv$ on $\OP_{n,k}$:
\begin{equation}
\Hilb(R_{n,k}; q) = P_{n,k}(q) = \rev_q( [k]!_q \cdot \Stir_q(n,k).
\end{equation}
In particular, we have
\begin{equation}
\dim(R_{n,k}) = |\OP_{n,k}| = k! \cdot \Stir(n,k).
\end{equation}
\end{theorem}

\begin{proof}
Let $<$ be the lexicographic  order on $\QQ[\xx_n]$ and consider the ideals
$I_{n,k}$ and $\TT(Y_{n,k})$.  By Lemma~\ref{I-contained-in-T}, we know 
$I_{n,k} \subseteq \TT(Y_{n,k})$.  From this we get the containment of monomial ideals
\begin{equation}
\initial_<(I_{n,k}) \subseteq \initial_<(\TT(Y_{n,k})).
\end{equation}
Let $\BBB^I_{n,k}$ be the standard monomial basis of $\frac{\QQ[\xx_n]}{I_{n,k}}$ and let 
$\BBB^{\bf T}_{n,k}$ be the standard monomial basis of $\frac{\QQ[\xx_n]}{\TT(Y_{n,k})}$.
The containment above immediately gives 
\begin{equation}
\BBB^{\bf T}_{n,k} \subseteq \BBB^{I}_{n,k}.
\end{equation}
Lemma~\ref{skip-monomials-in-initial} extends this inclusion to the triple
\begin{equation}
\BBB^{\bf T}_{n,k} \subseteq \BBB^{I}_{n,k} \subseteq \MMM_{n,k}.
\end{equation}

Since $\frac{\QQ[\xx_n]}{{\bf T}(Y_{n,k})} \cong \frac{\QQ[\xx_n]}{{\bf I}(Y_{n,k})} \cong \QQ[\OP_{n,k}]$, we know
$|\BBB^{\bf T}_{n,k}| = |\OP_{n,k}|$.  On the other hand, Theorem~\ref{psi-is-bijection} says 
$|\MMM_{n,k}| = |\OP_{n,k}|$.  Therefore, the above inclusions are actually equalities and we have
\begin{equation}
\label{normal-basis-equality}
\BBB^{\bf T}_{n,k} = \MMM_{n,k}.
\end{equation}
This equality and the containment $I_{n,k} \subseteq \TT(Y_{n,k})$ imply that we have the equality of ideals
\begin{equation}
\label{ideal-equality}
I_{n,k} = \TT(Y_{n,k}).
\end{equation}
Consequently, we have 
\begin{equation}
\dim(R_{n,k}) = |\BBB^{\bf T}_{n,k}| = |\MMM_{n,k}| = |\OP_{n,k}|.
\end{equation}
Moreover, we have 

\begin{align}
\Hilb(R_{n,k}; q) &= \Hilb \left( \frac{\QQ[\xx_n]}{I_{n,k}}; q \right) & & \text{(definition of $R_{n,k}$)}\\
&=  \Hilb \left( \frac{\QQ[\xx_n]}{\TT(Y_{n,k})}; q \right) & & \text{(Equation~\ref{ideal-equality})}  \\
&= \sum_{m \in \BBB^{\bf T}_{n,k}} q^{\deg(m)} & & \text{(definition of $\BBB^{\TT}_{n,k}$)} \\
&= \sum_{m \in \MMM_{n,k}} q^{\deg(m)} & & \text{(Equation~\ref{normal-basis-equality})} \\
&= \sum_{\sigma \in \OP_{n,k}} q^{\coinv(\sigma)} & & \text{(Theorem~\ref{psi-is-bijection})}.
\end{align}

The proof is complete.
\end{proof}

We are also in a position to identify the {\em ungraded} Frobenius character of $R_{n,k}$.

\begin{theorem}
\label{ungraded-frobenius-theorem}
Let $k \leq n$ be positive integers.  As {\em ungraded} $\symm_n$-modules, we have
\begin{equation}
R_{n,k} \cong_{\symm_n} \QQ[\OP_{n,k}].
\end{equation}
Equivalently, we have the {\em ungraded} Frobenius character
\begin{equation}
\Frob(R_{n,k}) = \sum_{\substack{\lambda \vdash n \\ \ell(\lambda) = k}} 
{k \choose m_1(\lambda), \dots, m_k(\lambda)} h_{\lambda}(\xx),
\end{equation}
where ${k \choose m_1(\lambda), \dots, m_k(\lambda)} = \frac{k!}{m_1(\lambda)! \cdots m_k(\lambda)!}$ 
is the multinomial coefficient.
\end{theorem}

\begin{proof}
We know that $\frac{\QQ[\xx_n]}{\TT(Y_{n,k})} \cong \frac{\QQ[\xx_n]}{\II(Y_{n,k})} \cong \QQ[\OP_{n,k}]$
as ungraded $\symm_n$-representations.  The proof of Theorem~\ref{hilbert-series-theorem} shows
$\TT(Y_{n,k}) = I_{n,k}$, so that $R_{n,k} = \frac{\QQ[\xx_n]}{I_{n,k}} = \frac{\QQ[\xx_n]}{\TT(Y_{n,k})}$.
\end{proof}

Recall that an $(n,k)$-staircase is a shuffle of the sequences
$(0, 1,  \dots, k-1)$ and $(k-1,  \dots, k-1)$, where there are $n-k$ copies of $k-1$ in the second sequence.

\begin{defn}
The  {\em $(n,k)$-Artin monomials $\AAA_{n,k}$} are those monomials in
$\QQ[\xx_n]$ whose exponent vectors fit under at least one $(n,k)$-staircase.
\end{defn}

In particular, the set $\AAA_{n,n}$ consists of the usual Artin monomials.

\begin{theorem}
\label{artin-basis-theorem}
Let $k \leq n$ be positive integers.  We have $\AAA_{n,k} = \MMM_{n,k}$.  Moreover, the set $\AAA_{n,k}$ descends
to a monomial basis of $R_{n,k}$.
\end{theorem}

The images in $R_{n,k}$ of the monomials in $\AAA_{n,k}$ will be called the {\em generalized Artin basis}
of $R_{n,k}$.

\begin{proof}
The proof of Theorem~\ref{hilbert-series-theorem} implies that $\MMM_{n,k}$ is the standard monomial basis for 
$R_{n,k}$ with respect to lexicographic order.  Therefore, we need only show that $\AAA_{n,k} = \MMM_{n,k}$.

If $(a_1, \dots, a_n)$ is any $(n,k)$-staircase, a direct check shows that the monomial
$x_1^{a_1} \cdots x_n^{a_n} \in \QQ[\xx_n]$ is $(n,k)$-nonskip.
Therefore, we have $\AAA_{n,k} \subseteq \MMM_{n,k}$.

To verify the reverse containment, we show that the bijection
$\Psi: \OP_{n,k} \rightarrow \MMM_{n,k}$ of Theorem~\ref{psi-is-bijection} satisfies 
$\Psi(\OP_{n,k}) \subseteq \AAA_{n,k}$.
Let $\sigma \in \OP_{n,k}$ and let $\overline{\sigma}$ be the ordered set partition of size $n-1$ obtained
by removing $n$ from $\sigma$.

{\bf Case 1:}  {\em $\{n\}$ is not a singleton block of $\sigma$.}

In this case $\overline{\sigma} \in \OP_{n-1,k}$.
We may inductively assume that $\Psi(\overline{\sigma}) \in \AAA_{n-1,k}$.  Therefore, there exists an 
$(n-1,k)$-staircase $(a_1, \dots, a_{n-1})$ such that $\Psi(\overline{\sigma}) \mid x_1^{a_1} \cdots x_{n-1}^{a_{n-1}}$.
We have $\Psi(\sigma) = \Psi(\overline{\sigma}) \cdot x_n^i$ where $0 \leq i \leq k-1$.  Since 
$(a_1, \dots, a_{n-1}, k-1)$ is an $(n,k)$-staircase and 
$\Psi(\sigma) \mid x_1^{a_1} \cdots x_{n-1}^{a_{n-1}} x_n^{k-1}$, we have
$\Psi(\sigma) \in \AAA_{n,k}$.

{\bf Case 2:}  {\em $\{n\}$ is  a singleton block of $\sigma$.}

In this case $\overline{\sigma} \in \OP_{n-1,k-1}$.
We have $\Psi(\sigma) = \Psi(\overline{\sigma}) \cdot \mm(S) \cdot x_n^i$, where $S \subseteq [n-1]$ satisfies
$|S| = n-k$, $0 \leq i \leq k-1$, and $\xx(S) \mid (\Psi(\overline{\sigma}) \cdot \mm(S))$.  Consider the $(n,k)$-staircase
$(a_1, \dots, a_n)$ defined by $a_j = k-1$ if $j \in S$ or $j = n$.  

We claim $\Psi(\sigma) \mid x_1^{a_1} \cdots x_n^{a_n}$, so that $\Psi(\sigma) \in \AAA_{n,k}$.  To see this, write
$\Psi(\sigma) = x_1^{b_1} \cdots x_n^{b_n}$.  Since $\Psi(\sigma) \in \MMM_{n,k}$ we know $0 \leq b_j \leq k-1$
for all $j$, so that $b_j \leq a_j$ if $j \in S$ or $j = n$.  If $j \in [n-1] - S$ is such that $b_j > a_j$, the fact
that $\xx(S) \mid \Psi(\sigma)$ would imply that $\xx(S \cup \{j\}) \mid \Psi(\sigma)$, contradicting the fact that
$\Psi(\sigma) \in \MMM_{n,k}$.  We conclude that $\Psi(\sigma) \mid x_1^{a_1} \cdots x_n^{a_n}$, as desired.
\end{proof}

As an example of Theorem~\ref{artin-basis-theorem}, 
consider the case $(n,k) = (3,2)$.  The $(3,2)$-staircases are the shuffles
of the sequences $(0,1)$ and $(1)$, i.e. the sequences
$(0,1,1)$ and $(1,0,1)$.  The $(3,2)$-Artin monomials are therefore
\begin{equation*}
\AAA_{3,2} = \{x_2 x_3, x_1 x_3, x_1, x_2, x_3, 1\}.
\end{equation*}
This is precisely the set $\MMM_{3,2}$ of monomials in $\QQ[\xx_3]$ which are divisible by none of the monomials
in the list
\begin{equation*}
x_1^2, \, x_2^2, \, x_3^2, \, \xx(12) = x_1 x_2, \, \xx(13) = x_1 x_3^2, \, \xx(23) = x_2^2 x_3^2.
\end{equation*}
We conclude that $\AAA_{3,2} = \MMM_{3,2}$ descends to a basis of $R_{3,2}$, yielding the Hilbert series
\begin{equation*}
\Hilb(R_{3,2}; q) = 2q^2 + 3q + 1.
\end{equation*}
This agrees with the $\coinv$ distribution on $\OP_{3,2}$.
\begin{center}
$\begin{array}{c | c c c c c c}
\sigma & (12 \mid 3) & (1 \mid 23) & (13 \mid 2) & (2 \mid 13) & (3 \mid 12) & (23 \mid 1) \\ \hline
\coinv & 2 & 2 & 1 & 1 & 1 & 0
\end{array}$
\end{center}

We can also determine the reduced Gr\"obner basis of $I_{n,k}$.

\begin{theorem}
\label{reduced-groebner-basis-theorem}
Let $k \leq n$ be positive integers and let $<$ be the lexicographic monomial order.
  A Gr\"obner basis for $I_{n,k}$ with respect to $<$ is given by the variable powers
\begin{equation*}
x_1^k, x_2^k, \dots, x_n^k
\end{equation*}
together with the reverse Demazure characters
\begin{equation*}
\kappa_{\gamma(S)^*}(\xx_n^*)
\end{equation*}
for $S \subseteq [n-1]$ satisfying $|S| = n-k+1$.

If $k < n$, this Gr\"obner basis is the reduced Gr\"obner basis for $I_{n,k}$ with respect to $<$.
\end{theorem}

When $k = n$,
the reduced Gr\"obner basis for the classical invariant ideal $I_{n,n}$ 
is $\{h_i(x_i, x_{i+1}, \dots, x_n) \,:\, 1 \leq i \leq n\}$ (see \cite[Sec. 7.2]{Bergeron}).
Since 
$h_i(x_i, x_{i+1}, \dots, x_n) = \kappa_{(0, \dots, 0, i, 0, \dots , 0)^*}(\xx_n^*)$
(where the $i$ in the vector $(0, \dots, 0, i, 0, \dots , 0)$ is in position $n-i$),
the last sentence of
Theorem~\ref{reduced-groebner-basis-theorem} is almost correct for 
$k = n$; one simply needs to throw out the variable powers $x_1^n, \dots x_{n-1}^n$.

\begin{proof}
By Lemma~\ref{demazure-identity} (and Equation~\ref{magical-demazure-equations-friend}) we know that 
the polynomials listed lie in $I_{n,k}$.  Lemma~\ref{reduced-demazure-lemma} and 
Theorem~\ref{artin-basis-theorem} tell us that the number of monomials in $\QQ[\xx_n]$
 which do not divide the leading terms of any of
the polynomials listed in the statement equals $\dim(R_{n,k})$.  It follows that the polynomials listed in the statement
form a Gr\"obner basis for $I_{n,k}$.

When $k < n$, Lemma~\ref{reduced-demazure-lemma} implies that, for any distinct polynomials $f, g$ listed in the 
statement, the monomial $\initial_<(f)$ has coefficient $1$ and does not divide any monomial in  $g$.  This 
implies the claim about reducedness.
\end{proof}

For example, consider the case $(n,k) = (6,4)$.  The reduced Gr\"obner basis of $I_{6,4} \subseteq \QQ[\xx_6]$
is given by 
the variable powers
\begin{equation*}
x_1^4, x_2^4, x_3^4, x_4^4, x_5^4, x_6^4
\end{equation*}
together with the reverse Demazure characters
\begin{align*}
&\kappa_{(0,0,0,1,1,1)}(\xx_6^*), \kappa_{(0,0,2,0,1,1)}(\xx_6^*),
\kappa_{(0,3,0,0,1,1)}(\xx_6^*), \kappa_{(0,0,2,2,0,1)}(\xx_6^*), \kappa_{(0,3,0,2,0,1)}(\xx_6^*), \\
&\kappa_{(0,3,3,0,0,1)}(\xx_6^*),
\kappa_{(0,0,2,2,2,0)}(\xx_6^*), \kappa_{(0,3,0,2,2,0)}(\xx_6^*), \kappa_{(0,3,3,0,2,0)}(\xx_6^*),
\kappa_{(0,3,3,3,0,0)}(\xx_6^*).
\end{align*}

The authors find 
Theorem~\ref{reduced-groebner-basis-theorem}  mysterious and know of no 
conceptual reason to expect Demazure characters to appear as Gr\"obner basis elements of $I_{n,k}$.

\section{Generalized Garsia-Stanton basis}
\label{Garsia-Stanton}

\subsection{The classical Garsia-Stanton basis for $R_n$}
The {\em Lehmer code} of a permutation $\pi = \pi_1 \dots \pi_n \in \symm_n$
is the word $c(\pi) = c(\pi)_1 \dots c(\pi)_n$, where 
$c(\pi)_i = | \{j < i \,:\, \pi_i < \pi_j \}$.  The map $\pi \mapsto c(\pi)$ provides a bijection from $\symm_n$
to the set of words which are componentwise less than $(0, 1, \dots, n-1)$.  The classical Artin monomial basis
of $R_n$ therefore witnesses the fact that
$\Hilb(R_n; q) = \sum_{\pi \in S_n} q^{c(\pi)_1 + \cdots + c(\pi)_n} = \sum_{w \in \symm_n} q^{\inv(\pi)}$.

Just as the Artin basis for $R_n$ is related to the Lehmer code sum or $\inv$ statistic on $\symm_n$,
the Garsia-Stanton basis for $R_n$ is related to the major index statistic $\maj$ on $\symm_n$.
Given a permutation $\pi = \pi_1 \dots \pi_n \in \symm_n$, the {\em Garsia-Stanton monomial}
$gs_{\pi}$ is the product
\begin{equation}
gs_{\pi} := \prod_{i \in \Des(\pi)} x_{\pi_1} x_{\pi_2} \cdots x_{\pi_i} \in \QQ[\xx_n].
\end{equation}
For example, if $\pi = 34256187$ we have
\begin{equation*}
gs_{\pi} = (x_3 x_4) \cdot (x_3 x_4 x_2 x_5 x_6) \cdot (x_3 x_4 x_2 x_5 x_6 x_1 x_8) \in \QQ[\xx_8].
\end{equation*}
It is clear that $\deg(gs_{\pi}) = \maj(\pi)$ for any permutation $\pi \in \symm_n$.

Garsia used Stanley-Reisner theory to show that $\{gs_{\pi} \,:\, \pi \in \symm_n\}$ descends to
a basis of $R_n$ \cite{Garsia}; this basis was then studied by Garsia and Stanton in the context of invariant
theory \cite{GS}.  
Adin, Brenti, and Roichman later gave a different proof of this fact using a straightening argument \cite{ABR}.
This witnesses the fact that
$\Hilb(R_n; q) = \sum_{\pi \in \symm_n} q^{\maj(\pi)}$.
The GS monomials were also studied by 
E. E. Allen under the name `descent monomials' \cite{Allen}.

\subsection{The generalized Garsia-Stanton basis for $R_{n,k}$}
Remmel and Wilson proved that $\inv$ and $\maj$ share the same distribution on $\OP_{n,k}$ \cite{RW}.  
Taking complementary
statistics, we get
\begin{equation*}
\sum_{\sigma \in \OP_{n,k}} q^{\coinv(\sigma)} = \sum_{\sigma \in \OP_{n,k}} q^{\comaj(\sigma)}.
\end{equation*}
Theorem~\ref{artin-basis-theorem} shows that the $(n,k)$-Artin basis $\AAA_{n,k}$ of $R_{n,k}$ witnesses the fact that
the left hand side is $\Hilb(R_{n,k}; q)$.  It is natural to ask for a generalized Garsia-Stanton basis which witnesses the 
fact that the right hand side is also $\Hilb(R_{n,k}; q)$.  We will provide such a basis.

To start, we will recast the right hand side in a more illuminating form.

\begin{lemma}
\label{comaj-generating-function}
We have 
\begin{equation}
\sum_{\sigma \in \OP_{n,k}} q^{\comaj(\sigma)} = \sum_{\pi \in \symm_n} q^{\maj(\pi)} \cdot {\asc(\pi) \brack n-k}_q.
\end{equation}
\end{lemma}

\begin{proof}
We consider the model for ordered set partitions in $\OP_{n,k}$ as ascent starred permutations 
$(\pi, S)$, where $\pi = \pi_1 \cdots \pi_n \in \symm_n$ and $S$ is a set of ascents of $w$ with $|S| = n-k$.  
For a fixed permutation $\pi$, there are ${\asc(\pi) \choose n-k}$ choices for the set of stars $S$.
For example, if $(n,k) = (6,4)$ and $\pi = 245136 \in \symm_6$ we have $n-k = 2$ and the ${4 \choose 2}$
elements of $\OP_{6,4}$ shown below, together with the values of $\maj$ and $\comaj$. 
\begin{center}
$\begin{array}{c|cccccc} \sigma &
2_* 4_* 5 \, \, 1 \, \, 3 \, \, 6 & 2_* 4 \, \, 5 \, \, 1_* 3 \, \, 6 & 2_* 4 \, \, 5 \, \, 1 \, \, 3_* 6 &
2 \, \, 4_* 5 \, \, 1_* 3 \, \, 6 & 2 \, \, 4_* 5 \, \, 1 \, \, 3_* 6 & 2 \, \, 4 \, \, 5 \, \, 1_* 3_* 6  \\ \hline
\maj & 5 & 6 & 7 & 7 & 8 & 9 \\ \hline
\comaj & 7 & 6 & 5 & 5 & 4 & 3
\end{array}$
\end{center}

We claim that, for $\pi \in \symm_n$ fixed, we have 
\begin{equation}
\sum_{\substack{S \subseteq \Asc(\pi) \\ |S| = n-k}} q^{\comaj(\pi,S)} = q^{\maj(\pi)} \cdot {\asc(\pi) \brack n-k}_q.
\end{equation}
Indeed, the standard combinatorial 
interpretation for ${\asc(\pi) \brack n-k}_q$ is as the generating function for size on partitions
$\lambda$ fitting inside a box of dimensions $(n-k) \times (\asc(\pi) - n + k)$.  For $\pi$ fixed, the choice
of $\lambda$ corresponds to the choice of $S \subseteq \Asc(\pi)$ in the ordered set partition $(\pi, S)$.
The definition of $\comaj$ yields the factor of $q^{\maj(\pi)}$.
\end{proof}

We introduce the following generalization of the GS monomials.

\begin{defn}
Let $k \leq n$ be positive integers.  The {\em $(n,k)$-Garsia-Stanton monomials} $\GS_{n,k}$
are given by
\begin{equation*}
\GS_{n,k} := \{ gs_{\pi} \cdot x_{\pi_1}^{i_1} x_{\pi_2}^{i_2} \cdots x_{\pi_{n-k}}^{i_{n-k}} \,:\, 
\pi \in \symm_n, \des(\pi) < k, \text{ and } (k - \des(\pi)) > i_1 \geq i_2 \geq \cdots \geq i_{n-k} \geq 0\}.
\end{equation*}
\end{defn} 

When $k = n$ we recover the usual GS monomials.

As an example of these monomials, take $(n,k) = (6,3)$ and consider the permutation
$\pi = 356124 \in \symm_6$.  We have $n-k = 3$, 
$k - \des(\pi) = 2$ and $gs_{\pi} = x_3 x_5 x_6$.  This gives rise to the following four elements of $\GS_{6,3}$:
\begin{equation*}
(x_3 x_5 x_6) \cdot 1, \,\,\,\,\,\,  (x_3 x_5 x_6) \cdot (x_3), \,\,\,\,\,\,
 (x_3 x_5 x_6) \cdot (x_3 x_5), \,\,\,\,\,\, (x_3 x_5 x_6) \cdot (x_3 x_5 x_6).
\end{equation*}

In general, if $\pi \in \symm_n$ satisfies $\des(\pi) < k$, we have 
\begin{align}
\sum_{(k - \des(\pi)) > i_1 \geq i_2 \geq \cdots \geq i_{n-k} \geq 0} 
q^{\mathrm{deg}(gs_{\pi} \cdot x_{\pi_1}^{i_1} x_{\pi_2}^{i_2} \cdots x_{\pi_{n-k}}^{i_{n-k}})} &=
q^{\maj(\pi)} \cdot {(k - \des(\pi) - 1) + (n-k) \brack n-k}_q \\
&= q^{\maj(\pi)} \cdot {n - \des(\pi) - 1 \brack n-k}_q \\
&= q^{\maj(\pi)} \cdot {\mathrm{asc}(\pi) \brack n-k}_q,
\end{align}
which agrees with the summand corresponding to $\pi$ in Lemma~\ref{comaj-generating-function}.
In particular, we have $|\GS_{n,k}| \leq |\OP_{n,k}|$.  We cannot yet assert equality because two  
$(n,k)$-GS monomials could {\em a priori} coincide for different choices of $\pi$.
The following theorem guarantees that this does not actually happen.

\begin{theorem}
\label{garsia-stanton-basis-theorem}
Let $k \leq n$ be positive integers.  The set $\GS_{n,k}$ of $(n,k)$-Garsia-Stanton monomials descends 
to a basis for $R_{n,k}$.
\end{theorem}

\begin{proof}

The strategy for proving $\GS_{n,k}$ is a basis of $R_{n,k}$ is to apply the 
{\em straightening algorithm} of Adin, Brenti, and Roichman \cite{ABR}.  To describe this algorithm,
we will need a partial order on monomials in $\QQ[\xx_n]$.  

Given any  monomial
$m = x_1^{a_1} \cdots x_n^{a_n} \in \QQ[\xx_n]$, 
let $\lambda(m) = \mathrm{sort}(a_1, \dots, a_n)$ be the 
partition obtained by sorting the exponent sequence of $m$ into weakly decreasing order.
Let $\pi(m) = \pi_1 \dots \pi_n \in \symm_n$ be the permutation obtained by listing the indices
of variables in decreasing order of exponents in $m$, breaking ties by choosing  smaller indexed
variables first.  For example, if $m = x_1^3 x_2^4 x_3^0 x_4^2 x_5^2 x_6^0 x_7^0$,
we have 
$\lambda(m) = (4,3,2,2)$ and $\pi(m) = 2145367$.  If $m, m'$ are any two monomials in 
$\QQ[\xx_n]$, we say $m \prec m'$ if $\deg(m) = \deg(m')$ and 
\begin{itemize}
\item  $\lambda(m) < \lambda(m')$ in dominance order, or
\item  $\lambda(m) = \lambda(m')$ but $\inv(\pi(m')) > \inv(\pi(m))$.
\end{itemize}

If $\pi = \pi_1 \dots \pi_n$ is any permutation in $\symm_n$, let $d(\pi) = (d_1, \dots, d_n)$ be the length
$n$ vector given by
$$
d_j = |\Des(\pi) \cap \{j, j+1, \dots, n\}|.
$$
For example, we have $d(2145367) = (2,1,1,1,0,0,0)$. 
We have $d_1 = \des(\pi)$.
 If $m$ is a monomial in $\QQ[\xx_n]$,
let $d(m) := d(\pi(m))$.  It is shown in \cite{ABR} that the componentwise difference
$\lambda(m) - d(m)$ is an integer partition (i.e., has weakly decreasing components).
Let $\mu(m)$ be the {\em conjugate} of this difference:  $\mu(m)' = \lambda(m) - d(m)$.
In our example, we have $\mu(m)' = (2,2,1,1,0,0,0)$ so that $\mu(m) = (4,2)$.

Adin, Brenti, and Roichman derive the following {\em straightening lemma}
\cite[Lem. 3.5]{ABR}.  If $m$ is any monomial in $\QQ[\xx_n]$ we may write
\begin{equation}
\label{straighten}
m = gs_{\pi(m)} \cdot e_{\mu(m)}(\xx_n) - \Sigma,
\end{equation}
where $\Sigma$ is a linear combination of monomials $m' \prec m$.
This implies that the classical GS monomials span $R_{n,n}$.
Together with the fact that $\dim(R_{n,n}) = n!$, we get that the classical GS monomials 
are a basis for $R_{n,n}$.  

We want to apply the straightening relation
$(\ref{straighten})$ to prove that $\GS_{n,k}$ spans $R_{n,k}$.  Since we know 
$|\GS_{n,k}| \leq \dim(R_{n,k})$, this would prove that $\GS_{n,k}$ is a basis 
of $R_{n,k}$.

Consider a monomial $m$ in $\QQ[\xx_n]$.  We argue that the class 
$m + I_{n,k}$ lies in the span of $\GS_{n,k}$.
All of the equivalences in the following paragraphs are modulo the ideal $I_{n,k}$.

If $m$ is minimal with respect to
$\prec$, then $m$ must have the form $m = x_1^0 \cdots x_{r-1}^0 x_r^1 \cdots x_n^1$
for some $r$.  If $r = 1$, then $m = e_n(\xx_n) \equiv 0$.  If $r > 1$,
then $m = gs_{r (r+1) \dots n 1 2 \dots (r-1)}$, so that $m+I_{n,k}$ lies in the span of 
$\GS_{n,k}$.

By the last paragraph, we may inductively assume that if $m'$ is any monomial such that
$m' \prec m$,
the coset $m' + I_{n,k}$ 
lies in the span of the monomials in $\GS_{n,k}$.
Apply the straightening relation (\ref{straighten}) to $m$, yielding
\begin{equation}
m = gs_{\pi(m)} \cdot e_{\mu(m)}(\xx_n) - \Sigma.
\end{equation}
By induction, the element $\Sigma + I_{n,k}$ lies in the span of $\GS_{n,k}$.

For any $\pi \in \symm_n$, the definition of GS monomials implies that $\des(\pi)$ is the exponent
of $x_{\pi_1}$ in the GS monomial $gs_{\pi}$.
In particular,
if $\des(\pi(m)) \geq k$, then 
$gs_{\pi(m)} \cdot e_{\mu(m)}(\xx_n) \equiv gs_{\pi(m)} \equiv 0$ and
$m + I_{n,k}$  lies in the span of $\GS_{n,k}$.
If $\mu(m)_1 > n-k$ then 
$gs_{\pi(m)} \cdot e_{\mu(m)}(\xx_n) \equiv e_{\mu(m)}(\xx_n) \equiv e_{\mu(m)_1}(\xx_n) \equiv 0$, 
so that $m + I_{n,k}$ lies in the span of $\GS_{n,k}$.

We are therefore reduced to the situation where $m$ is a monomial in $\QQ[\xx_n]$
with 
\begin{equation*}
\des(\pi(m)) < k \text{ and } \mu(m)_1 \leq n-k.
\end{equation*}
We claim that any such monomial $m$
 actually lies in $\GS_{n,k}$, or satisfies $m \equiv 0$.
Indeed, write $\pi(m) = \pi_1 \dots \pi_n$, so that
\begin{equation*}
m = x_{\pi_1}^{\lambda(m)_1} x_{\pi_2}^{\lambda(m)_2} \cdots x_{\pi_n}^{\lambda(m)_n}.
\end{equation*}
The GS monomial $gs_{\pi(m)}$ is equal to
\begin{equation*}
gs_{\pi(m)} = x_{\pi_1}^{d_1(m)} x_{\pi_2}^{d_2(m)} \cdots x_{\pi_n}^{d_n(m)}.
\end{equation*}
By the definition of $\mu(m)$ we have
\begin{equation*}
m = gs_{\pi(m)} \cdot x_{\pi_1}^{\mu(m)'_1} x_{\pi_2}^{\mu(m)'_2} \cdots x_{\pi_n}^{\mu(m)'_n}.
\end{equation*}
Since $\mu(m)_1 \leq n-k$, we know that 
\begin{equation*}
m = gs_{\pi(m)} \cdot x_{\pi_1}^{\mu(m)'_1} x_{\pi_2}^{\mu(m)'_2} \cdots x_{\pi_{n-k}}^{\mu(m)'_{n-k}},
\end{equation*}
where $\mu(m)'_1 \geq \cdots \geq \mu(m)'_{n-k} \geq 0$.  As long as 
$\mu(m)'_1 \leq k - \des(\pi) = k - d_1(m)$, we have $m \in \GS_{n,k}$.
If $\mu(m)'_1 > k - \des(\pi)$, then $x_{\pi_1}^k \mid m$, so that $m \equiv 0$.
\end{proof}

The proof of Theorem~\ref{garsia-stanton-basis-theorem} is an extension of the proof of Adin, Brenti,
and Roichman that the classical GS monomials form a basis for $R_n$ \cite{ABR}.
It could be interesting to find an alternative proof of Theorem~\ref{garsia-stanton-basis-theorem}
which extends the original poset-theoretic proof of Garsia \cite{Garsia}.

\section{Frobenius series}
\label{Frobenius series}

\subsection{The rings $R_{n,k,s}$}
The goal of this section is to prove our formula for the {\em graded}
Frobenius character of $R_{n,k}$.  To do this, we will need a one-parameter extension
$R_{n,k,s}$ of these rings.

\begin{defn}
Let $s \leq k \leq n$ be positive integers.  Define $I_{n,k,s} \subseteq \QQ[\xx_n]$
to be the ideal 
\begin{equation*}
I_{n,k,s} := \langle x_1^k , x_2^k, \dots, x_n^k, e_n(\xx_n), e_{n-1}(\xx_n), \dots, e_{n-s+1}(\xx_n) \rangle.
\end{equation*}
Let $R_{n,k,s} := \frac{\QQ[\xx_n]}{I_{n,k,s}}$ be the corresponding quotient ring.
\end{defn}

We have $I_{n,k,1} \subseteq I_{n,k,2} \subseteq \cdots \subseteq I_{n,k,k} = I_{n,k}$, so that 
$R_{n,k,k} = R_{n,k}$.  The ideals $I_{n,k,s}$ are homogeneous and stable under the action of 
$\symm_n$, so that the $R_{n,k,s}$ are graded $\symm_n$-modules.

We want to give a combinatorial model for the rings $R_{n,k,s}$.  To do so, for $s \leq k \leq n$
define $\OP_{n,k,s}$ to be the collection of $k$-block ordered set partitions 
$\sigma = (B_1 \mid \cdots \mid B_k)$
of $[n+(k-s)]$
such that, for $1 \leq i \leq k-s$, we have
$n + i \in B_{s+i}$.  For example, we have
\begin{equation*}
(46 \mid 1 \mid 23{\bf 7} \mid {\bf 8} \mid 5 {\bf 9}) \in \OP_{6,5,2}.
\end{equation*}
We will refer to the letters $n+1, n+2, \dots, n+k-s$ of $\sigma \in \OP_{n,k,s}$ as {\em big}; these 
are shown in bold above.  The other letters $1, 2, \dots, n$ will be called {\em small}.

The symmetric group $\symm_n$ acts on $\OP_{n,k,s}$ by permuting the small letters.  
We will identify this action with $R_{n,k,s}$ in a manner completely analogous to the analysis in 
Section~\ref{Hilbert series}; we will be more brief this time.
Let us model this action with
a point set.  Fix distinct rational numbers $\alpha_1, \dots, \alpha_k$.

\begin{defn}
Let $Y_{n,k,s} \subseteq \QQ^{n + k-s}$ be the set of points $(y_1, \dots, y_n, y_{n+1}, \dots , y_{n+k-s})$ such that
\begin{itemize}
\item  $\{y_1, \dots, y_n, y_{n+1}, \dots, y_{n+k-s}\} = \{\alpha_1, \dots, \alpha_k\}$, and
\item  $y_{n+i} = \alpha_{s+i}$ for all $1 \leq i \leq k-s$.
\end{itemize}
\end{defn}

It is evident that the action of $\symm_n$ on the small coordinates gives an identification of 
$\symm_n$-modules
$\QQ[\OP_{n,k,s}] = \QQ[Y_{n,k,s}]$.

Let $\II(Y_{n,k,s})$ be the ideal of polynomials in $\QQ[\xx_{n+k-s}]$ which vanish on $Y_{n,k,s}$ and let 
$\TT(Y_{n,k,s})$ be the corresponding top component ideal.  Since $x_{n+i} - \alpha_{n+i} \in \II(Y_{n,k,s})$ 
for all $1 \leq i \leq k-s$, we know that $x_{n+i} \in \TT(Y_{n,k,s})$.  Let 
$\zeta: \QQ[\xx_{n+k-s}] \twoheadrightarrow \QQ[\xx_n]$ be the evaluation map which sets $\zeta(x_{n+i}) = 0$
for all $1 \leq i \leq k-s$.  Let $T_{n,k,s} := \zeta(\TT(Y_{n,k,s}))$.  It follows that $T_{n,k,s} \subseteq \QQ[\xx_n]$
is an ideal in $\QQ[\xx_n]$ and we have  isomorphisms of $\symm_n$-modules 
\begin{equation*}
\QQ[\OP_{n,k,s}] \cong
\frac{\QQ[\xx_{n+k-s}]}{\II(Y_{n,k,s})} \cong \frac{\QQ[\xx_{n+k-s}]}{\TT(Y_{n,k,s})} \cong \frac{\QQ[\xx_n]}{T_{n,k,s}}.
\end{equation*}

We prove a generalization of Lemma~\ref{I-contained-in-T} to the ideals $I_{n,k,s}$.

\begin{lemma}
\label{I-contained-in-T-generalized}
Let $s \leq k \leq n$.
We have $I_{n,k,s} \subseteq T_{n,k,s}$.
\end{lemma}

\begin{proof}
We show that every generator of $I_{n,k,s}$ lies in $T_{n,k,s}$.  

For $1 \leq i \leq n$, we have
$(x_i - \alpha_1) \cdots (x_i - \alpha_k) \in \II(Y_{n,k,s})$, so that $x_i^k \in T_{n,k,s}$.

Lemma~\ref{vanishing-lemma} implies that 
$e_r(\xx_{n+k-s}) \in \TT(Y_{n,k,s})$ for all $r \geq n-s+1$.  Applying the evaluation map $\zeta$ gives
$\zeta: e_r(\xx_{n+k-s}) \mapsto e_r(\xx_n) \in T_{n,k,s}$.
\end{proof}

Next, we prove that certain reverse Demazure characters are contained in $I_{n,k,s}$. These 
polynomials will ultimately be members of a Gr\"obner basis of $I_{n,k,s}$.

\begin{lemma}
\label{demazures-contained-in-I-generalized}
Suppose $S \subseteq [n]$ satisfies $|S| = n-s+1$.  The reverse Demazure character
$\kappa_{\gamma(S)^*}(\xx_n^*)$ lies in $I_{n,k,s}$.
\end{lemma}

\begin{proof}
Apply Lemma~\ref{demazure-identity} and Equation~\ref{magical-demazure-equations-friend}.
\end{proof}

Let $\MMM_{n,k,s}$ be the collection of monomials $m  \in \QQ[\xx_n]$ 
satisfying the following two conditions:
\begin{enumerate}
\item  we have $x_i^k \nmid m$ for all $1 \leq i \leq n$, and 
\item  for any $S \subseteq [n]$ with $|S| = n-s+1$, we have $\xx(S) \nmid m$.
\end{enumerate}

The monomials in $\MMM_{n,k,s}$ will turn out to be the standard monomial basis for 
the ring $R_{n,k,s}$.  When $k = s$, we have $\MMM_{n,k,k} = \MMM_{n,k}$.  
A more
precise relationship between these two families of monomials is as follows.

\begin{lemma}
\label{augment}
Let $s \leq k \leq n$.
If 
$x_1^{a_1} \cdots x_n^{a_n} x_{n+1}^{a_{n+1}} \cdots x_{n+k-s}^{a_{n+k-s}} \in \MMM_{n+k-s,k}$,
then
$x_1^{a_1} \cdots x_n^{a_n} \in \MMM_{n,k,s}$.

On the other hand,
if $x_1^{a_1} \cdots x_n^{a_n} \in \MMM_{n,k,s}$ and
$0 \leq a_{n+1} < a_{n+2} < \cdots < a_{n+k-s} \leq k-1$, then
$x_1^{a_1} \cdots x_n^{a_n} x_{n+1}^{a_{n+1}} \cdots x_{n+k-s}^{a_{n+k-s}} \in \MMM_{n+k-s,k}$.
\end{lemma}

\begin{proof}
The first statement is clear from the definitions.  For the second statement, suppose 
$$x_1^{a_1} \cdots x_n^{a_n} \in \MMM_{n,k,s}$$
and
$$0 \leq a_{n+1} < a_{n+2} < \cdots < a_{n+k-s} \leq k-1.$$
We need to show
$$
m := x_1^{a_1} \cdots x_n^{a_n} x_{n+1}^{a_{n+1}} \cdots x_{n+k-s}^{a_{n+k-s}} \in \MMM_{n+k-s,k}.
$$
This amounts to showing that $\xx(S) \nmid m$ for any $S \subseteq [n+k-s]$ with $|S| = n-s+1$.
Certainly $\xx(S) \nmid m$ if $S \subseteq [n]$.  On the other hand, if 
$\xx(S) \mid m$ and 
$n+i \in S - [n]$, the 
exponent $e_{n+i}$ of $x_{n+i}$ in the skip monomial $\xx(S)$ is $\geq s + i$.  However, the inequalities
$0 \leq a_{n+1} < a_{n+2} < \cdots < a_{n+k-s} \leq k-1$ and the divisibility
$\xx(S) \mid m$ force $e_{n+i} \leq a_{n+i} < k-(s-i) \leq s + i$, which contradicts $e_i \geq s + i$.
\end{proof}

We apply our bijection $\Psi$ from Section~\ref{Hilbert series} to show that the monomials in
$\MMM_{n,k,s}$ are equinumerous with the ordered set partitions in $\OP_{n,k,s}$.

\begin{lemma}
\label{two-parameter-monomial-count}
For any $s \leq k \leq n$ we have
$|\MMM_{n,k,s}| = |\OP_{n,k,s}|$.
\end{lemma}

\begin{proof}
Consider the bijection $\Psi: \OP_{n+k-s,k} \rightarrow \MMM_{n+k-s,k}$ of Theorem~\ref{psi-is-bijection}.
What is the image of $\OP_{n,k,s}$ under $\Psi$?  Let $\MMM'_{n,k,s}$ be the set of monomials
$x_1^{a_1} \cdots x_{n+k-s}^{a_{n+k-s}} \in \MMM_{n+k-s,k}$ which satisfy
\begin{equation*}
(a_{n+1}, a_{n+2}, \dots, a_{n+k-s}) = (s , s+1 , \dots , k-1).
\end{equation*}
It follows from the definition of $\Psi$ that
\begin{equation*}\Psi(\OP_{n,k,s}) = \MMM'_{n,k,s}.\end{equation*}
On the other hand, Lemma~\ref{augment} guarantees that
$|\MMM'_{n,k,s}| = |\MMM_{n,k,s}|.$ 
\end{proof}

We generalize Theorem~\ref{reduced-groebner-basis-theorem} to get a Gr\"obner basis for the 
ideals $I_{n,k,s}$.

\begin{lemma}
\label{two-parameter-groebner}
Let $s \leq k \leq n$.
A Gr\"obner basis for the ideal $I_{n,k,s}$ with respect to $<_{lex}$ consists of the variable powers
$$
x_1^k, \dots, x_n^k
$$
together with the reverse  Demazure characters
$$
\{ \kappa_{\gamma(S)^*}(\xx_n^*) \,:\, S \subseteq [n],  |S| = n-s+1\}.
$$
If $s < k$, this is the reduced Gr\"obner basis for this term ordering.

In particular, we have $\dim(R_{n,k,s}) = |\OP_{n,k,s}|$.
\end{lemma}

\begin{proof}
The polynomials in question lie in the ideal $I_{n,k,s}$ by Lemma~\ref{demazures-contained-in-I-generalized}. 
By Lemma~\ref{two-parameter-monomial-count}, the number of monomials which do not divide
any leading terms of the polynomials listed here equals $|\OP_{n,k,s}|$.
By Lemma~\ref{I-contained-in-T-generalized}, we have
$\dim(R_{n,k,s}) \geq | \OP_{n,k,s} |$.  This forces the set of polynomials here to be 
a Gr\"obner basis for $I_{n,k,s}$.  If $s < k$, reducedness follows from an argument
similar to the case of Theorem~\ref{reduced-groebner-basis-theorem}.
\end{proof}

Let us remark that the number $\dim(R_{n,k,s}) = |\OP_{n,k,s}|$ has combinatorial significance.
The collection $\OP_{n,k,s}$ bijects with the collection of functions $f: [n] \rightarrow [k]$ whose image contains
$[s]$.  The number of such functions is
\begin{equation}
\label{difference-equation}
\sum_{m = 0}^n {n \choose m} s! \cdot \Stir(m,s) \cdot (k-s)^{n-m},
\end{equation}
where $m$ parametrizes the size of the preimage of $[s]$ under $f$.
This formula makes sense (and gives the correct value $k^n$) when $s = 0$ and we set
$I_{n,k,0} := \langle x_1^k, \dots, x_n^k \rangle$.  For general $s$, the expression (\ref{difference-equation})
is equal to the $s^{th}$ difference of the sequence $k^n$.  When $s = k$, we recover the 
$k^{th}$ difference $k! \cdot \Stir(n,k)$.
\footnote{The authors thank Dennis Stanton for pointing this out.}

We have a surjection
\begin{equation*}
\frac{\QQ[\xx_n]}{T_{n,k,s}} \twoheadrightarrow R_{n,k,s}
\end{equation*}
and an identification of the $\symm_n$-module on the left hand side with $\OP_{n,k,s}$.
Lemma~\ref{two-parameter-groebner} tells
us that these modules have the same dimension, so we have
  an isomorphism of {\em ungraded} $\symm_n$-modules
\begin{equation}
\label{ungraded-two-parameter-identification}
R_{n,k,s} \cong \QQ[\OP_{n,k,s}].
\end{equation}

\subsection{Antisymmetrization and $e_j(\xx)^{\perp}$}
Consider the parabolic subgroup $\symm_{n-j} \times \symm_j$ of $\symm_n$ and let $\epsilon_j$ be the 
antisymmetrization operator with respect to the last $j$ variables.
In other words, we have that $\epsilon_j \in \QQ[\symm_n]$ is the group algebra idempotent
\begin{equation}
\epsilon_j := \frac{1}{j!} \sum_{\pi \in \symm_{\{n-j+1, \dots, n\}}} \sign(\pi) \cdot \pi.
\end{equation}

Let us consider the action of $\epsilon_j$ on $\OP_{n,k}$.
Since $\epsilon_j$ kills any  $\sigma \in \OP_{n,k}$ with any of the $j$ letters 
$n-j+1, n-j+2, \dots, n$ in the same block,
we have
\begin{equation}
\dim( \epsilon_j \QQ[\OP_{n,k}] ) = {k \choose j} \cdot | \OP_{n-j,k,k-j} |.
\end{equation}
Applying the isomorphism (\ref{ungraded-two-parameter-identification}) of ungraded
$\symm_n$-modules, we have 
\begin{equation}
\label{dimension-alternants}
\dim( \epsilon_j R_{n,k} ) = {k \choose j} \cdot | \OP_{n-j,k,k-j} |.
\end{equation}
Our next goal is to bootstrap Equation~\ref{dimension-alternants}
to a statement involving the rings $R_{n-j,k,k-j}$.

If $V$ is any $\symm_j$-module, recall that the space of {\em alternants} is 
\begin{equation}
\{ v \in V \,:\, \pi.v = \sign(\pi) \cdot v \text{ for all $\pi \in \symm_j$} \}.
\end{equation}
The symmetric group $\symm_j$ acts on the quotient ring
$\frac{\QQ[x_{n-j+1}, \dots, x_n]}{\langle x_{n-j+1}^k, \dots, x_n^k \rangle}$
by variable permutation; let $A_{n,k,j}$ be the space of alternants for this module.
The set
\begin{equation}
\{\epsilon_j \cdot (x_{n-j+1}^{a_{n-j+1}} \cdots x_n^{a_n}) \,:\, 0 \leq a_{n-j+1} < \cdots < a_n \leq k-1 \}
\end{equation}
descends to a  basis for $A_{n,k,j}$; it follows that
\begin{equation}
\label{equation-a-hilbert}
\Hilb(A_{n,k,j}) = q^{{j \choose 2}} {k \brack j}_q.
\end{equation}

Observe that $\epsilon_j R_{n,k}$ is a graded $\symm_{n-j}$-module.  The group 
$\symm_{n-j}$ also acts on the first component of the tensor product
$R_{n-k,k,k-j} \otimes A_{n-j}$.  The next result states that the natural multiplication map
induces an isomorphism between these graded modules.

\begin{lemma}
\label{mu-is-isomorphism}
As graded $\symm_{n-j}$-modules we have $\epsilon_j R_{n,k} \cong  R_{n-j,k,k-j} \otimes A_{n,k,j}$.
\end{lemma}

\begin{proof}
Consider the direct sum decomposition 
$\QQ[\zz_j] = \QQ[\zz_j]_0 \oplus \QQ[\zz_j]_1 \oplus \QQ[\zz_j]_2 \oplus \cdots$ where
$\QQ[\zz_j]_d$ is the vector space of polynomials in $\zz_j$ which are homogeneous of degree $d$.
For any $d \geq 0$, define 
$\QQ[\zz_j]_{\geq d} := \QQ[\zz_j]_d \oplus \QQ[\zz_j]_{d + 1} \oplus \QQ[\zz_j]_{d + 2} \oplus \cdots $ and
let $\epsilon_j \QQ[\zz_j]_{\geq d}$ be the image of $\QQ[\zz_j]_{\geq d}$ under $\epsilon_j$.

We will use the spaces $\QQ[\zz_j]_{\geq d}$ to obtain descending
 filtrations of our modules of interest.  For any $d \geq 0$,
define
\begin{align}
U_{\geq d} &:= \text{image of $\QQ[\yy_{n-j}] \otimes \epsilon_j \QQ[\zz_j]_{\geq d}$ 
in $\QQ[\yy_{n-j}] \otimes A_{n,k,j}$,} \\
V_{\geq d} &:= \text{image of $\QQ[\yy_{n-j}] \otimes \epsilon_j \QQ[\zz_j]_{\geq d}$  in
$R_{n-j,k,k-j} \otimes A_{n,k,j}$,} \\
W_{\geq d} &:= \text{image of $\QQ[\yy_{n-j}] \cdot \epsilon_j \QQ[\zz_j]_{\geq d}$ in $\epsilon_j R_{n,k}$}.
\end{align}
Here we have suppressed dependence on $n, k, j$ to reduce notational clutter.
Each of the spaces $U_{\geq d}, V_{\geq d},$ and $W_{\geq d}$ is closed under the action of $\symm_{n-j}$
on the $\yy$-variables and is a $\QQ[\yy_{n-j}]$-module.  For any $d \geq 0$, set
\begin{equation}
U_d := U_{\geq d}/U_{\geq d+1}, \quad V_d := V_{\geq d}/V_{\geq d+1}, \quad
W_d := W_{\geq d}/W_{\geq d+1}.
\end{equation}

Fix $d \geq 0$ and consider the multiplication map $\widetilde{\mu}_d: U_{\geq d} \rightarrow W_{d}$
induced by $f(\yy_{n-j}) \otimes g(\zz_j) \mapsto f(\yy_{n-j}) \cdot g(\zz_j)$.
The map $\widetilde{\mu}_d$ is both a graded map of $\symm_{n-j}$-modules and a 
$\QQ[\yy_{n-j}]$-module homomorphism.

For any $r > n - k$ and any $g(\zz_j) \in \epsilon_j \QQ[\zz_j]_{\geq d}$ we have 
\begin{equation}
\widetilde{\mu}_d: e_r(\yy_{n-j}) \otimes g(\zz_j) \mapsto
e_r(\yy_{n-j}) g(\zz_j) = \sum_{a + b = r} e_a(\yy_{n-j}) e_b(\zz_j) g(\zz_j) = e_r(\xx_n) g(\zz_j) \in 
\epsilon_j I_{n,k}.
\end{equation}
The first equality uses the fact that $W_{\geq d+1} = 0$ inside $W_d$.  For $1 \leq i \leq n-j$ one also has
\begin{equation}
\widetilde{\mu}_d: y_i^k \otimes g(\zz_j) \mapsto y_i^k g(\zz_j) \in \epsilon_j I_{n,k}.
\end{equation}
Since $\widetilde{\mu}_d$ is a $\QQ[\yy_{n-j}]$-module homomorphism we have 
\begin{equation}
(I_{n-j,k,k-j} \otimes \epsilon_j A_{n,k,j}) \cap U_{\geq d} \subseteq \mathrm{Ker}(\widetilde{\mu}_d).
\end{equation}
It follows that $\widetilde{\mu}_d$ induces a map $\widehat{\mu}_d: V_{\geq d} \rightarrow W_d$.

Since $\widehat{\mu}_d(V_{\geq d+1}) \subseteq W_{\geq d+1}/W_{\geq d+1} = 0$, the map $\widehat{\mu}_d$
in turn induces a map $\mu_d: V_d \rightarrow W_d$.  Let 
\begin{equation}
\mu: \bigoplus_{d \geq 0} V_d \rightarrow \bigoplus_{d \geq 0} W_d 
\end{equation}
be the direct 
sum of the maps $\mu_d$.

We claim that $\mu$ is an isomorphism of graded $\symm_{n-j}$-modules. The map $\mu$ is a
graded $\symm_{n-j}$-module map because it is a direct sum of such maps, so it suffices to 
verify that $\mu$ is bijective.  To do this, we demonstrate that $\mu$ sends a basis to a basis.

Consider the subset of $\epsilon_j R_{n,k}$ given by the 
set $\CCC_{n,k,j}$ of images under $\epsilon_j$ of 
monomials 
$m(\xx_n) = m(\yy_{n-j}) \cdot m(\zz_j) \in \MMM_{n,k}$ with the property that the exponent sequence in 
$m(\zz_j)$ is strictly increasing.   In other words, we have
\begin{equation}
\CCC_{n,k,j} := \left\{ \epsilon_j m(\xx_n) \,:\, 
\begin{array}{c}
m(\xx_n) = m(\yy_{n-j}) \cdot m(\zz_j) \in \MMM_{n,k} \text{ and } \\
m(\zz_j) = z_1^{a_1} \cdots z_j^{a_j} \text{ with } a_1 < \cdots < a_j \end{array} \right\}.
\end{equation} 
We claim that $\CCC_{n,k,j}$ is a basis of $\epsilon_j R_{n,k}$.
Since the elements of $\CCC_{n,k,j}$ are homogeneous in the $\zz$-variables, this will
also show that $\CCC_{n,k,j}$ gives a basis of the codomain
$\bigoplus_{d \geq 0} W_d$ of $\mu$.

Since $\MMM_{n,k}$ is a basis of $R_{n,k}$, it is immediate that 
$\{ \epsilon_j m(\xx_n) \,:\, m(\xx_n) \in \MMM_{n,k}\}$ spans $\epsilon_j R_{n,k}$.
Let $m(\xx_n) = m(\yy_{n-j}) \cdot m(\zz_j) \in \MMM_{n,k}$ with 
$m(\zz_j) = z_1^{a_1} \cdots z_j^{a_j}$.  If any of the numbers $(a_1, \dots, a_j)$ coincide,
then $\epsilon_j m(\xx_n) = m(\yy_{n-j}) \cdot \epsilon_j m(\zz_j) = m(\yy_{n-j}) \cdot 0 = 0$.
Moreover, if $m(\zz_j)' = z_1^{a_1'} \cdots z_j^{a_j'}$ and $(a_1', \dots, a_j')$ is any permutation
of $(a_1, \dots, a_j)$, we have $\epsilon_j m(\xx_n) = \pm \epsilon_j (m(\yy_{n-j}) \cdot m(\zz_j)')$.
It follows that $\CCC_{n,j}$ spans $\epsilon_j R_{n,k}$.

To show that the set $\CCC_{n,k,j}$ 
actually is a basis of $\epsilon_j R_{n,k}$, we perform a dimension count.
In particular,
Lemma~\ref{augment} and Proposition~\ref{two-parameter-groebner}
tell us that 
\begin{equation}
| \CCC_{n,k,j} | = {k \choose j} \cdot |\OP_{n-j,k,k-j}| = \dim(\epsilon_j R_{n,k}).
\end{equation}
Therefore, the set
$\CCC_{n,k,j}$ is a basis for $\epsilon_j R_{n,k}$ and induces a basis of $\bigoplus_{d \geq 0} W_d$.

In order to get a basis for the domain $\bigoplus_{d \geq 0} V_d$ of $\mu$, consider
the following set of simple tensors in $\QQ[\xx_n] = \QQ[\yy_{n-j}] \otimes \QQ[\zz_j]$.
\begin{equation}
\DDD_{n,k,j} := \left\{ m(\yy_{n-j}) \otimes \epsilon_j m(\zz_j) \,:\, 
\begin{array}{c}
m(\yy_{n-j}) \in  \MMM_{n-j,k,k-j} \text{ and } \\
m(\zz_j) = z_1^{a_1} \cdots z_j^{a_j} \text{ with } 0 \leq a_1 < \cdots < a_j  < k \end{array} \right\}.
\end{equation} 
Since $\MMM_{n-j,k,k-j}$ is the standard monomial basis of $R_{n-j,k,k-j}$ with respect to the lexicographical
term ordering, this is a basis for the tensor product $R_{n-j,k,k-j} \otimes A_{n,k,j}$.
Since the elements of $\DDD_{n,k,j}$ are homogeneous with respect to the $\zz$-variables,
the set $\DDD_{n,k,j}$ yields a basis of $\bigoplus_{d \geq 0} V_d$.

By definition, the multiplication map $\mu$ carries the basis element 
$m(\yy_{n-j}) \otimes \epsilon_j m(\zz_j) \in \DDD_{n,k,j}$ to the corresponding basis element
$\epsilon_j m(\xx_n) \in \CCC_{n,k,j}$ where $m(\xx_n) = m(\yy_{n-j}) \cdot m(\zz_j)$. This proves that 
$\mu: \bigoplus_{d \geq 0} V_d \rightarrow \bigoplus_{d \geq 0} W_d$ is an isomorphism of 
graded $\symm_{n-j}$-modules.  By standard properties of filtrations we also have 
isomorphism of graded $\symm_{n-j}$-modules 
$\epsilon_j R_{n,k} \cong \bigoplus_{d \geq 0} V_d$ and
$R_{n-j,k,k-j} \otimes A_{n,k,j} \cong \bigoplus_{d \geq 0} W_d$.
\end{proof}

Since we ultimately want to determine the graded Frobenius image $\grFrob(R_{n,k}; q)$, we need
to relate the antisymmetrization operator $\epsilon_j$ to symmetric function theory.  
Let $V$ be any $\symm_n$-module.  The image $\epsilon_j V$ is a subspace of $V$ and carries an action 
of the subgroup $\symm_{n-j}$ in the parabolic decomposition 
$\symm_{n-j} \times \symm_j$.  We may therefore speak of the Frobenius image
$\Frob(\epsilon_j V) \in \Lambda_{n-j}$.  The symmetric functions 
$\Frob(\epsilon_j V)$ and $\Frob(V)$ are related by
\begin{equation}
\label{general-e-perp-equation}
\Frob(\epsilon_j V) = e_j(\xx)^{\perp} \Frob(V).
\end{equation}

Equation~\ref{general-e-perp-equation} was used extensively by Garsia and Procesi in their study 
of the cohomology of Springer fibers \cite{GP}.  The fastest way to prove Equation~\ref{general-e-perp-equation}
is to use Frobenius reciprocity.  If $V$ is a graded $\symm_n$-module, we restate Equation~\ref{general-e-perp-equation}
for emphasis:
\begin{equation}
\label{general-e-perp-equation-graded}
\grFrob(\epsilon_j V; q) = e_j(\xx)^{\perp} \grFrob(V; q).
\end{equation}

We want to prove $\grFrob(R_{n,k}; q) = D_{n,k}(\xx; q)$.
By Lemma~\ref{e-perp-determines-symmetric-function}, it is enough to show 
\begin{equation*}
e_j(\xx_n)^{\perp} \grFrob(R_{n,k};q) = e_j(\xx_n)^{\perp} D_{n,k}(\xx;q)
\end{equation*}
for all $j \geq 1$.
Lemma~\ref{e-perp-image} gives an expression for $e_j(\xx_n)^{\perp} D_{n,k}(\xx;q)$ in terms of smaller 
$D$-functions.  
We know that 
\begin{align}
e_j(\xx)^{\perp} \grFrob(R_{n,k}; q) &= \grFrob(\epsilon_j R_{n,k}; q) & 
\text{(by Equation~\ref{general-e-perp-equation-graded})} \\
&= \Hilb(A_{n,k,j};q) \cdot \grFrob(R_{n-j,k,k-j}; q) & \text{(by Lemma~\ref{mu-is-isomorphism})} \\
&= q^{{j \choose 2}} {k \brack j}_q \cdot \grFrob(R_{n-j,k,k-j};q). & \text{(by Equation~\ref{equation-a-hilbert})}
\end{align}
If we want $\grFrob(R_{n,k};q)$ to satisfy the same recursion as in Lemma~\ref{e-perp-image}, we must have
\begin{equation}
\label{target-equation}
 \grFrob(R_{n-j,k,k-j};q) = 
 \sum_{m = \max(1,k-j)}^{\min(k,n-j)}  
q^{(k-m) \cdot (n-j-m)}  {j \brack k-m}_q \grFrob(R_{n-j,m};q).
\end{equation}

\subsection{A short exact sequence}  The goal of the rest of this section is to prove 
Equation~\ref{target-equation}.  Its proof will rely on the following short exact sequence of $\symm_n$-modules.

\begin{lemma}
\label{two-parameter-short-exact-sequence}
Let $s < k \leq n$.  There is a short exact sequence
\begin{equation}
0 \rightarrow R_{n,k-1,s} \rightarrow R_{n,k,s} \rightarrow R_{n,k,s+1} \rightarrow 0
\end{equation}
of $\symm_n$-modules, where the first map is homogeneous of degree $n-s$ and the second
map is homogeneous of degree $0$.

Equivalently, we have
\begin{equation}
\label{pascal-recursion}
\grFrob(R_{n,k,s};q) = \grFrob(R_{n,k,s+1};q) + q^{n-s} \cdot \grFrob(R_{n,k-1,s};q).
\end{equation}
\end{lemma}

\begin{proof}
As we have an inclusion of ideals $I_{n,k,s} \subseteq I_{n,k,s+1}$, we take the second map
to be the canonical projection
\begin{equation}
\pi: R_{n,k,s} \twoheadrightarrow R_{n,k,s+1} \rightarrow 0.
\end{equation}

To build the first map, consider the multiplication 
\begin{equation}
\widetilde{\phi}: \QQ[\xx_n] \xrightarrow{\cdot e_{n-s}(\xx_n)} R_{n,k,s}.
\end{equation}
We claim that $\widetilde{\phi}(I_{n,k-1,s}) = 0$.
This amounts to showing that $\widetilde{\phi}(x_i^{k-1}) = 0$ in $R_{n,k,s}$ for all
$1 \leq i \leq n$.  
To ease notation, we will only handle the case $i = 1$ (the other cases are similar).
We have (where all congruences are modulo $I_{n,k,s}$)
\begin{align}
\widetilde{\phi}(x_1^{k-1}) &= x_1^{k-1} \cdot e_{n-s}(x_1, \dots, x_n) \\
&= x_1^k e_{n-s-1}(x_2, x_3, \dots, x_n) + x_1^{k-1} e_{n-s}(x_2, x_3, \dots, x_n) \\
&\equiv x_1^{k-1} e_{n-s}(x_2, x_3, \dots, x_n) \\
&= x_1^{k-2} e_{n-s+1}(x_1, x_2, \dots, x_n) - x_1^{k-2} e_{n-s+1}(x_2, x_3, \dots, x_n) \\
&\equiv - x_1^{k-2} e_{n-s+1}(x_2, x_3, \dots, x_n) \equiv \cdots \\
&\equiv \pm x_1^{k-s} e_{n-1}(x_2, \dots x_n) = \pm x_1^{k-s-1} e_n(x_1, \dots, x_n) \equiv 0.
\end{align}
Therefore, the map $\widetilde{\phi}$ induces a map
$$
\phi: R_{n,k-1,s} \rightarrow R_{n,k,s}
$$
of homogeneous degree $n-s$ which surjects onto the kernel of $\pi$.

By considering whether the final block of an element of $\OP_{n,k,s}$ is the singleton 
$\{n+k-s\}$, we get the identity
\begin{equation}
|\OP_{n,k,s}| = |\OP_{n,k,s-1}| + |\OP_{n,k-1,s}|.
\end{equation}
 By  Lemma~\ref{two-parameter-groebner}, this implies
that 
\begin{equation}
\dim(R_{n,k,s}) = \dim(R_{n,k,s-1}) + \dim(R_{n,k-1,s}).
\end{equation}
This forces the complex
\begin{equation}
0 \rightarrow R_{n,k-1,s} \xrightarrow{\phi} R_{n,k,s} \xrightarrow{\pi} R_{n,k,s+1} \rightarrow 0
\end{equation}
to be exact.  To finish the proof, simply observe that the maps $\phi$ and $\pi$ commute 
with the action of $\symm_n$.
\end{proof}

We are ready to prove Equation~\ref{target-equation}.  The proof will rely on the short exact sequence of 
Lemma~\ref{two-parameter-short-exact-sequence} and the $q$-analog of the Pascal's Triangle recursion.

\begin{lemma}
\label{goal-equation-lemma}
Let $1 \leq j \leq n$.
We have
\begin{equation}
 \grFrob(R_{n-j,k,k-j}; q) = 
 \sum_{m = \max(1,k-j)}^{\min(k,n-j)}  
q^{(k-m) \cdot (n-j-m)}  {j \brack k-m}_q \grFrob(R_{n-j,m}; q).
\end{equation}
\end{lemma}

\begin{proof}
Reindexing our variables, this is equivalent to
\begin{equation}
\label{goal-equation}
 \grFrob(R_{n,k,s}; q) = \sum_{m = 0}^{k-s} q^{m \cdot (n-k+m)}  {k-s \brack m}_q  \grFrob(R_{n,k-m}; q),
\end{equation}
for any $1 \leq s \leq k \leq n$,
where we adopt the convention that $\grFrob(R_{n',k'}; q) = 0$ if $k' > n'$.  
 If $k = s$, we have  $R_{n,k,s} = R_{n,k,k} = R_{n,k}$ and we are done.  To prove the general
 case, we use the $q$-Pascal recursion.

Suppose $s < k$ and 
let $E_{n,k,s}$ be the right hand side of Equation~\ref{goal-equation}.  Then
$E_{n,k,s+1} + q^{n-s} \cdot E_{n,k-1,s}$ is the expression

\begin{align*}
 \sum_{m = 0}^{k-s-1} q^{m \cdot (n-k+m)}  &{k-s-1 \brack m}_q  \grFrob(R_{n,k-m};q) + \\
&q^{n-s} \cdot \sum_{m' = 0}^{k-s-1} q^{m' \cdot (n-k+m'-1)}  {k-s-1 \brack m'}_q  \grFrob(R_{n,k-m'-1};q),
\end{align*}
where we changed the dummy variable from $m$ to $m'$ in the second summation.
Grouping like terms gives 
\begin{equation}
\sum_{m = 0}^{k-s} f_q(n,k,m,s) \cdot  \grFrob(R_{n,k-m}), 
\end{equation}
where
\begin{align}
f_q(n,k,m,s) &= q^{m \cdot (n-k+m)} {k-s-1 \brack m}_q + q^{(n-s) + (m-1) \cdot (n-k+m)} 
{k-s-1 \brack m-1}_q \\
&= q^{m \cdot (n-k+m)} \cdot \left( {k-s-1 \brack m}_q + q^{k-s-m} {k-s-1 \brack m-1}_q \right) \\
&= q^{m \cdot (n-k+m)} \cdot {k-s \brack m}_q
\end{align}
and the third equality used the $q$-Pascal recursion.

This proves that 
\begin{equation}
E_{n,k,s} = E_{n,k,s+1} + q^{n-s} \cdot E_{n,k-1,s}.
\end{equation}
Applying Lemma~\ref{two-parameter-short-exact-sequence} we see that the 
left hand side of Equation~\ref{goal-equation} satisfies the same recursion, finishing the proof.
\end{proof}

We are finally ready to prove $\grFrob(R_{n,k};q) = D_{n,k}(\xx;q)$.  With the machinery we have so far,
this is just a chain of lemma invocations.

\begin{theorem}
\label{graded-frobenius-theorem}
The graded Frobenius character of the module $R_{n,k}$ is given by
\begin{equation*}
\grFrob(R_{n,k};q) = D_{n,k}(\xx;q).
\end{equation*}
\end{theorem}

\begin{proof}
Let $j \geq 1$.
By Lemma~\ref{e-perp-determines-symmetric-function}, it is enough to show 
$e_j(\xx)^{\perp} \grFrob(R_{n,k};q) = e_j(\xx)^{\perp} D_{n,k}(\xx;q)$.
We  inductively assume that this identity  holds for all $n' < n$.

We have  that 
\begin{align}
e_j(\xx)^{\perp} \grFrob(R_{n,k};q) 
&= \grFrob(\epsilon_j R_{n,k};q) \\
&= \Hilb(A_{n,k,j};q) \cdot \grFrob(R_{n-j,k,k-j};q) \\ 
&= q^{j \choose 2} {k \brack j}_q \cdot \grFrob(R_{n-j,k,k-j};q) \\
&= q^{j \choose 2} {k \brack j}_q \cdot  \sum_{m = \max(1,k-j)}^{\min(k,n-j)}  
q^{(k-m) \cdot (n-j-m)}  {j \brack k-m}_q \grFrob(R_{n-j,m};q) \\
&=  q^{j \choose 2} {k \brack j}_q \cdot
\sum_{m = \max(1,k-j)}^{\min(k,n-j)}  
q^{(k-m) \cdot (n-j-m)}  {j \brack k-m}_q D_{n-j,m}(\xx;q) \\
&= e_j(\xx)^{\perp} D_{n,k}(x;q).
\end{align}
The first equality is the effect of $e_j(\xx)^{\perp}$ on (graded) Frobenius characters.
The second equality follows from Lemma~\ref{mu-is-isomorphism}.
The third equality follows from Equation~\ref{equation-a-hilbert}.
The fourth equality follows from Lemma~\ref{goal-equation-lemma}.
The fifth equality uses induction.
The final equality follows from Lemma~\ref{e-perp-image}.
\end{proof}

We therefore have the following combinatorial interpretation of $\grFrob(R_{n,k};q)$ 
in terms of ordered multiset partition statistics.

\begin{corollary}
\label{combinatorial-expansion}
The graded Frobenius character $\grFrob(R_{n,k})$ is equal to any of the following four expressions,
after applying $\rev_q$ and $\omega$:
\begin{equation}
\sum_{|\gamma| = n} \sum_{\mu \in \OP_{\gamma,k}} q^{\inv(\mu)} \xx^{\gamma} =
\sum_{|\gamma| = n} \sum_{\mu \in \OP_{\gamma,k}} q^{\maj(\mu)} \xx^{\gamma} =
\sum_{|\gamma| = n} \sum_{\mu \in \OP_{\gamma,k}} q^{\dinv(\mu)} \xx^{\gamma} =
\sum_{|\gamma| = n} \sum_{\mu \in \OP_{\gamma,k}} q^{\minimaj(\mu)} \xx^{\gamma}.
\end{equation}
\end{corollary}

For example, consider $(n,k) = (3,2)$.  We have the inversion counts
\begin{center}
$\begin{array}{c | c c c c c c c c}
\mu   & (1 \mid 12) & (12 \mid 1) & (1 \mid 23) & (2 \mid 13) & (3 \mid 12) & (12 \mid 3) & (13 \mid 2) & (23 \mid 1) \\ \hline
\inv(\mu) & 0 & 1 & 0 & 1 & 1 & 0 & 1 & 2
\end{array}$
\end{center}
which implies 
\begin{align*}
\sum_{|\gamma| = 3} \sum_{\mu \in \OP_{\gamma,2}} q^{\inv(\mu)} \xx^{\gamma} &= (1+q) m_{(2,1)}(\xx) 
+ (2+3q+q^2) m_{(1,1,1)}(\xx) \\
&= (1+q) s_{(2,1)}(\xx) + (q+q^2) s_{(1,1,1)}(\xx).
\end{align*}
Applying $\rev_q$ and $\omega$ we get
\begin{equation*}
\grFrob(R_{3,2};q) = (q + q^2) s_{(2,1)}(\xx) + (1 + q) s_{(3)}(\xx).
\end{equation*}

An explicit expansion of $\grFrob(R_{n,k};q)$ in the Schur basis
follows from the work of A. T. Wilson \cite{WMultiset}.

\begin{corollary}
\label{schur-expansion}
We have 
\begin{equation}
\label{schur-equation}
\grFrob(R_{n,k};q) =  \sum_{T \in \SYT(n)}
q^{\maj(T)} {n-\des(T)-1 \brack n-k}_q s_{\shape(T)}(\xx).
\end{equation}
\end{corollary}

\begin{proof}
If we let $m = 0$ and take the coefficient of $u^{n-k}$ in \cite[Thm. 5.0.1]{WMultiset}, 
and then apply $\omega$ and $\rev_q$,
we see that
\begin{equation}
\label{wilson-equation}
\grFrob(R_{n,k};q) = \rev_q  \left[ \sum_{T \in \SYT(n)}
q^{\maj(T) + {n-k \choose 2} - (n-k) \cdot \des(T)} {\des(T) \brack n-k}_q s_{\shape(T)'}(\xx) \right].
\end{equation}
Since the maximum power of $q$ appearing in ${\des(T) \brack n-k}_q$
is $(\des(T) - n + k) \cdot (n-k)$ and the maximum power of $q$ occurring in
$\grFrob(R_{n,k};q)$ is $(n-k) \cdot (k-1) + {k \choose 2}$, applying $\rev_q$ to the right hand side of 
Equation~\ref{wilson-equation} gives
\begin{equation}
\label{new-wilson-equation}
\grFrob(R_{n,k};q) = \sum_{T \in \SYT(n)}
q^{{n \choose 2} - \maj(T)} {\des(T) \brack n-k}_q s_{\shape(T)'}(\xx).
\end{equation}
Since $\maj(T') = {n \choose 2} - \maj(T)$ and $\des(T') = n-\des(T)-1$, the corollary follows.
\end{proof}

For example, consider $(n,k) = (4,2)$.  The tableaux with contribute to the right hand side of 
Equation~\ref{schur-equation} are those elements in $\SYT(4)$ with $\leq 1$ descent:
\begin{center}
\begin{small}
\begin{Young}
1 & 2 & 3 & 4
\end{Young}, \hspace{0.2in}
\begin{Young}
1 & 2 & 3 \cr 4
\end{Young}, \hspace{0.2in}
\begin{Young}
1 & 2 & 4 \cr 3
\end{Young}, \hspace{0.2in}
\begin{Young}
1 & 3 & 4 \cr 2
\end{Young}, \hspace{0.2in}
\begin{Young}
1 & 2 \cr 3 & 4
\end{Young}
\end{small}
\end{center}
We get that $\grFrob(R_{4,2};q)$ is equal to
\begin{equation*}
q^{0} {3 \brack 2}_q s_{(4)}(\xx) +
q^{3}{2 \brack 2}_q s_{(3,1)}(\xx) +
q^{2}{2 \brack 2}_q s_{(3,1)}(\xx)  
+ q^{1}{2 \brack 2}_q s_{(3,1)}(\xx) +
q^{2}{2 \brack 2}_q s_{(2,2)}(\xx).
\end{equation*}

It is well known that $\grFrob(R_n;q) = Q'_{(1^n)}(\xx;q)$, where $Q'_{\lambda}(\xx;q)$ is the dual 
Hall-Littlewood symmetric function corresponding to $\lambda \vdash n$.
We have the following generalization of this identity to arbitrary $k \leq n$.

\begin{theorem}
\label{hall-littlewood-expansion}
We have 
\begin{equation}
\label{hl-equation}
\grFrob(R_{n,k};q) =
\rev_q \left[
\sum_{\substack{\lambda \vdash n \\ \ell(\lambda) = k}}  q^{\sum (i-1) \cdot (\lambda_i - 1)}
{k \brack m_1(\lambda), \dots , m_n(\lambda)}_q Q'_{\lambda}(\xx; q) \right].
\end{equation}
\end{theorem}

\begin{proof}

By  Corollary \ref{combinatorial-expansion} we have
$\grFrob(R_{n,k};q) = \rev _q \omega C_{n,k}(\xx;q)$, so an equivalent formulation of (\ref{hl-equation}) is
\begin{equation}
\label{hl2-equation}
\omega C_{n,k}(\xx;q) = 
\left[
\sum_{\substack{\lambda \vdash n \\ \ell(\lambda) = k}}  q^{\sum (i-1) \cdot (\lambda_i - 1)}
{k \brack m_1(\lambda), \dots , m_n(\lambda)}_q Q'_{\lambda}(\xx; q) \right].
\end{equation}
Letting $t=0$ in (\ref{riseform}), we see a labelled Dyck path $P \in \mathcal{LD}_n$ will contribute to 
$C_{n,k}(\xx;q,0)$ iff the underlying Dyck path is ``balanced", i.e.
has the property that for every $i$ satisfying $a_i>0$, we also have $a_{i-1}=a_i-1$.  Such paths are in bijection with 
(strong) compositions $\alpha$ of $n$: the path $p (\alpha)$ corresponding to $\alpha$ consists of $\alpha _1$ north steps followed by
$\alpha _1$ east steps, then $\alpha _2$ north steps followed by $\alpha _2$ east steps, etc..   Thus 
\begin{align}
\label{what}
\omega C_{n,k}(\xx;q,0)= \omega \sum_{\alpha}  \sum_{P} q^{\text{dinv}(P)} \xx ^{P},
\end{align}
where the inner sum is over all labellings of the balanced Dyck path $p (\alpha)$.

By  \cite[Thm. 6.8]{Hagbook}, the inner sum from (\ref{what}) can be
expressed as 
\begin{align}
\sum_{\lambda} q^{\text{mindinv}(p (\alpha))} K_{\lambda ^{\prime},\mu(\alpha)}(q) s_{\lambda}(\xx),
\end{align}
where $\text{mindinv}$ is the minimum value taken by $\text{dinv}$ over all labellings of $p (\alpha)$,
$\mu (\alpha)$ is the rearrangement of $\alpha$ into non-increasing order, and $K_{\beta,\mu}(q)$
is the coefficient of $s_{\beta}$ in $Q^{\prime}_{\mu}(\xx;q)$.  It is noted in \cite[p. 98-99]{Hagbook} 
that this theorem follows from
a result of Lascoux, Leclerc, and Thibon on LLT polynomials, and also that $\text{mindinv} (\alpha)$ equals the number of
triples $u,v,w$ of north steps of $\alpha$ where $v$ is directly below $u$, and $w$ is in a row above $v$ and in the same
diagonal (i.e. line of slope $1$) as $v$.  
If $i<j$ and $\alpha _i > \alpha _j$, the squares in the columns 
corresponding to $\alpha _i$ and $\alpha _j$ will therefore contribute $\min (\alpha _i, \alpha _j)$ to $\text{mindinv}$, while if 
$i<j$ and $\alpha _i \le  \alpha _j$, they will contribute $\min (\alpha _i,\alpha _j) -1$ to $\text{mindinv}$.  It follows easily that
the contribution to the right-hand-side of (\ref{what}) from
the set of all balanced paths $\alpha$ for which $\mu (\alpha)$ equals a fixed partition $\lambda$ is the 
$q$-multinomial coefficient occurring in the sum in (\ref{hl2-equation}), times $Q^{\prime}_{\lambda}(\xx;q)$, times
$q^{\text{mindinv}(\text{reverse}(\lambda))}$.  Furthermore
\begin{equation*}
 \text{mindinv}(\text{reverse}(\lambda)) = \sum _{i} (i-1)(\lambda _i -1).
\end{equation*}
\end{proof}

For example, consider $(n,k) = (6,3)$.
The partitions $\lambda \vdash 6$ with $\ell(\lambda) = 3$ are 
$(4,1,1), (3,2,1),$ and $(2,2,2)$.  Consequently, the graded Frobenius character of 
$R_{6,3}$ is the $q$-reversal of
\begin{equation*}
q^0 {3 \brack 2,1}_q Q'_{(4,1,1)}(\xx;q) +
q^1 {3 \brack 1,1,1}_q Q'_{(3,2,1)}(\xx;q) +
q^3 {3 \brack 3}_q Q'_{(2,2,2)}(\xx;q).
\end{equation*}

\section{Conclusion}
\label{Conclusion}

In this paper we studied a generalization $R_{n,k}$ of the coinvariant algebra $R_n$ attached
to $\symm_n$.  The relevant ideal was a deformation
$I_{n,k} = \langle x_1^k, \dots, x_n^k, e_{n-k+1}(\xx_n), \dots, e_n(\xx_n) \rangle$ 
of the usual invariant ideal
$\langle \QQ[\xx_n]^{\symm_n}_+ \rangle = \langle e_1(\xx_n), \dots, e_n(\xx_n) \rangle$.

However, the elementary symmetric functions are not the only interesting choice of generators 
for the ideal $\langle \QQ[\xx_n]^{\symm_n}_+ \rangle$.  For example, we have
\begin{equation*}
\langle \QQ[\xx_n]^{\symm_n}_+ \rangle = \langle h_1(\xx_n), \dots, h_n(\xx_n) \rangle =
\langle p_1(\xx_n), \dots, p_n(\xx_n) \rangle,
\end{equation*}
where $h_d(\xx_n)$ is a homogeneous symmetric function and 
$p_d(\xx_n) = x_1^d + \cdots + x_n^d$ is a power sum symmetric function.
For $k < n$, one can define the ideals $I^h_{n,k}$ and $I^p_{n,k}$  by replacing $e$'s  in the definition of
$I_{n,k}$ with
$h$'s or $p$'s.   

It turns out that none of the ideals $I_{n,k}, I^h_{n,k},$ or $I^p_{n,k}$ coincide.  We do not have a conjecture for
the Hilbert series or Frobenius character of the quotient of $\QQ[\xx_n]$ by $I^h_{n,k}$ or 
$I^p_{n,k}$.  Moreover, we do not have a conceptual understanding of why elementary symmetric functions give
rise to ordered set partitions, but homogeneous or power sums do not.
It would be interesting to recast our construction of $I_{n,k}$ in such a way that makes the presence
of elementary symmetric functions natural.

Our analysis of the quotient rings $R_{n,k}$ used many similar methods to those found in the work
of Garsia and Procesi \cite{GP}.  In particular, Garsia and Procesi studied the quotient
$\frac{\QQ[\xx_n]}{I_{\lambda}}$ of $\QQ[\xx_n]$ by the {\em Tanisaki ideal} 
$I_{\lambda}$ for $\lambda \vdash n$.  They proved that 
\begin{equation}
\label{gp-equation}
\grFrob \left(\frac{\QQ[\xx_n]}{I_{\lambda}}\right) = Q'_{\lambda}(\xx;q).
\end{equation}
If we compare Equations~\ref{hl-equation} and \ref{gp-equation}, we notice that 
$R_{n,k}$ is isomorphic to a direct sum of modules of the form $\frac{\QQ[\xx_n]}{I_{\lambda}}$
for $\ell(\lambda) = k$, with appropriate grading shifts.  

\begin{problem}
Develop a relationship between
the Tanisaki ideals $I_{\lambda}$ and our ideal $I_{n,k}$ (perhaps a sort of filtration) which proves 
Equation~\ref{hl-equation} algebraically.
\end{problem}

In our paper we generalized two important bases of the ring $R_n$ to the ring $R_{n,k}$:
the Artin basis and the Garsia-Stanton basis.
The coinvariant algebra $R_n$ has another important basis: the basis of {\em Schubert polynomials}
$\{ \symm_{\pi} \,:\, \pi \text{ a permutation of $1, 2, \dots, n$} \}$.  Although the Schubert polynomial
$\symm_{\pi}$ is almost never a monomial, it has a beautiful positive decomposition into monomials via the 
theory of {\em pipe dreams}.  The structure constants of the Schubert basis of $R_n$ are known to coincide
with those for 
multiplication of {\em Schubert classes} 
inside the cohomology ring of the variety of complete flags in $\CC^n$.

\begin{problem}
\label{schubert-problem}
Extend the Schubert basis to obtain a basis $\{ \symm_{\sigma} \,:\, \sigma \in \OP_{n,k} \}$ 
of $R_{n,k}$ indexed by ordered set partitions.  The polynomials $\symm_{\sigma}$ should be a positive
integer combination of monomials and they should have positive structure constants inside $R_{n,k}$.
Furthermore, they should represent classes of a basis of the cohomology ring
$H^{\bullet}(F_{n,k}; \ZZ)$, where $F_{n,k}$ is some $(n,k)$-generalization of the flag variety.
\footnote{While this paper was under review, Pawlowski and Rhoades 
\cite{PR} solved Problem~\ref{schubert-problem}.}
\end{problem}

Let $W$ be a complex reflection group acting on $V = \CC^n$.  Then $W$ acts on the coordinate ring 
$\CC[V]$ and the ideal $\langle \CC[V]^W_+ \rangle$ of $W$-invariants with vanishing constant term
is the $W$-analog of the classical invariant ideal for $\symm_n$.  
Chevalley proved that we have an identification of $W$-modules
$\frac{\CC[V]}{\langle \CC[V]^W_+ \rangle} \cong \CC[W]$ \cite{C}.

\begin{problem}
\label{coxeter-complex-problem}
For $0 \leq k \leq n$, develop a $W$-generalization $I_{W,k} \subseteq \CC[V]$ of our ideals $I_{n,k}$.
When $W$ is well-generated, the action of $W$ on the quotient
$\frac{\CC[V]}{I_{W,k}}$ should give a graded version of the action of $W$ on the $k$-dimensional faces of the 
Coxeter complex $\Delta(W)$ attached to $W$.
\end{problem}

When $W = G(r,1,n)$, Chan and  Rhoades \cite{CR}  have a solution to Problem~\ref{coxeter-complex-problem}.

The modules $R_{n,k}$ have a further refinement as follows.  For a monomial $m \in \QQ[\xx_n]$,
let $\lambda(m)$ be the partition obtained by sorting the exponent sequence of $m$.  For a partition $\mu$,
define two subspaces $P_{\leq \mu}$ and $P_{< \mu}$ of $\QQ[\xx_n]$ by
\begin{equation}
P_{\leq \mu} = \mathrm{span} \{ m \in \QQ[\xx_n] \,:\, \lambda(m) \leq \mu\}, \hspace{0.5in}
P_{< \mu} = \mathrm{span} \{ m \in \QQ[\xx_n] \,:\, \lambda(m) < \mu\}.
\end{equation}
Let $Q_{\leq \mu}$ and $Q_{< \mu}$ be the images of these subspaces in $R_{n,k}$.  
Finally, let $R_{n,k,\mu} := Q_{\leq \mu}/Q_{< \mu}$ be the corresponding quotient.
We have a graded decomposition of $\symm_n$-modules
\begin{equation}
R_{n,k} \cong \bigoplus_{\mu} R_{n,k,\mu},
\end{equation}
where all but finitely many of the modules in the sum are $0$.

When $k = n$, Adin, Brenti, and Roichman used the Garsia-Stanton basis of $R_n$ to determine
the isomorphism type of $R_{n,n,\mu}$ for all $\mu$ \cite{ABR}.  
This suggests the following problem.

\begin{problem}
\label{partition-isomorphism-problem}
Determine the Frobenius image $\Frob(R_{n,k,\mu})$ for $k \leq n$ and any partition $\mu$.
\end{problem}

The solution of Problem~\ref{partition-isomorphism-problem} in the case $k = n$ found in \cite{ABR}
makes heavy use of the fact that $\QQ[\xx_n]$ is a free module over $\QQ[\xx_n]^{\symm_n}$, which 
in turn comes from the regularity of the sequence $e_1(\xx_n), e_2(\xx_n), \dots, e_n(\xx_n)$.
Since no such regularity holds for the ideal $I_{n,k}$, the solution of Problem~\ref{partition-isomorphism-problem}
for general $k \leq n$ will require different methods.
\footnote{Kyle Meyer (personal communication) has solved 
Problem~\ref{partition-isomorphism-problem} and is preparing a writeup.}

Huang and Rhoades \cite{HR} defined an analog of our quotient $R_{n,k}$ which carries an action of the
 0-Hecke algebra.  Recall that the {\em 0-Hecke algebra} (over any field $\mathbb{F}$) is the unital, associative
 $\mathbb{F}$-algebra with generators $T_1, T_2, \dots, T_{n-1}$ and relations
 \begin{equation*}
 \begin{cases}
 T_i^2 = T_i & 1 \leq i \leq n-1 \\
 T_i T_j = T_j T_i & |i - j| > 1 \\
 T_i T_{i+1} T_i = T_{i+1} T_i T_{i+1} & 1 \leq i \leq n-2.
 \end{cases}
 \end{equation*}
 The algebra $H_n(0)$ acts on the polynomial ring $\mathbb{F}[\xx_n]$ by the rule
 $T_i.f := \frac{x_i f - x_{i+1} (s_i.f)}{x_i - x_{i+1}}$.
 Although the ideal $I_{n,k}$ is not closed under this action, the following ideal $J_{n,k}$ is:
 \begin{equation*}
 J_{n,k} :=
 \langle h_k(x_1), h_k(x_1, x_2), \dots, h_k(x_1, x_2, \dots, x_n), e_n(\xx_n), e_{n-1}(\xx_n), \dots, e_{n-k+1}(\xx_n) \rangle.
 \end{equation*}
 The corresponding quotient $S_{n,k} := \frac{\mathbb{F}[\xx_n]}{J_{n,k}}$ plays the role of the ring $R_{n,k}$ in \cite{HR}.

 The {\em Hecke algebra} $H_n(q)$ (where $q$ is a parameter) 
 interpolates between the symmetric group algebra and the 0-Hecke algebra.
 It has  generators $T_1, T_2, \dots, T_{n-1}$ and the same relations as above, except that 
 $T_i^2 = q + (1-q) T_i$.
 
 \begin{problem}
 \label{hecke-problem}
 Give an analog of the ring $R_{n,k}$ which carries an action of the Hecke algebra $H_n(q)$.
 \end{problem}

In light of \cite{HR}, a solution to Problem~\ref{hecke-problem} would most likely involve defining a new ideal 
to play the role of $I_{n,k}$.
\footnote{While  this paper was under review, Problem~\ref{hecke-problem}
was solved by Huang, Rhoades, and Scrimshaw \cite{HRS}.
The relevant ideal replaces $x_i^k$ with the Hall-Littlewood $P$-function
$P_k(x_1, x_2, \dots, x_i;q)$.}

The modules $R_{n,k}$ of this paper were shown to have graded Frobenius image equal to either of the 
combinatorial expressions in the Delta Conjecture at $t = 0$.
This leads to the following problem.

\begin{problem}
\label{general-delta-conjecture-problem}
Find a nice bigraded $\symm_n$-module whose bigraded Frobenius character equals any of the expressions
$\Delta'_{e_{k-1}} e_n(\xx), \Rise_{n,k}(\xx;q,t),$ or $\Val_{n,k}(\xx;q,t)$ in the Delta Conjecture
(after application of $\omega$ and an appropriate change of variables).
\end{problem}

Problem~\ref{general-delta-conjecture-problem} is probably very difficult.
However, there are two specializations of Problem~\ref{general-delta-conjecture-problem} which could be quite
interesting and tractable.
The first concerns the specialization $t = 1$.
Romero \cite{Romero} has proven that 
$\Delta'_{e_{k-1}} e_n(\xx) \mid_{q = 1, t = q} = \Rise_{n,k}(\xx;1,q)$.

\begin{problem}
\label{t-one-delta-conjecture-problem}
Find a nice (singly) graded $\symm_n$-module whose graded Frobenius character is given by
$\Delta'_{e_{k-1}} e_n(\xx) \mid_{q = 1,t=q} = \Rise_{n,k}(\xx;1,q)$ (after applying $\omega$ and an appropriate 
change of variables).
\end{problem}

In the case $k = n$
Berget and Rhoades \cite{BR} proved that the restriction to $\symm_n$ of a graded $\symm_{n+1}$-module $V(n)$
defined by Postnikov and Shapiro \cite{PS} solves 
Problem~\ref{t-one-delta-conjecture-problem}.
The module $V(n)$ is defined in a matroid-theoretic fashion involving subgraphs of the complete graph $K_{n+1}$
whose complements are connected.

Another interesting specialization in Macdonald polynomial theory is $t = 1/q$.

\begin{problem}
\label{t-1/q-delta-conjecture-problem}
Find a nice (singly) graded $\symm_n$-module whose graded Frobenius character is given by
any of 
$\Delta'_{e_{k-1}} e_n(\xx), \Rise_{n,k}(\xx;q,t),$ or $\Val_{n,k}(\xx;q,t)$ after setting $t = 1/q$
(and applying $\omega$ and an appropriate change of variables).
\end{problem}

Haglund, Remmel, and Wilson have a plethystic formula for $\Delta'_{e_{k-1}} e_n(\xx)$ at $t = 1/q$
\cite[Thm. 5.1, Eqn. 27]{HRW}.  The case $k = n$ was solved by Haiman \cite{HaimanConj}; one considers the 
quotient of $\frac{\QQ[\xx_n]}{\langle x_1 + \cdots + x_n\rangle}$ 
by a homogeneous system of parameters $\theta_1, \theta_2, \dots, \theta_{n-1}$ of degree $n+1$
which carries the dual of the reflection
representation of $\symm_n$.
This construction was generalized to all real reflection groups $W$ and 
{\em $W$-noncrossing parking functions}
in the Parking Conjecture of Armstrong, Reiner, and Rhoades \cite{ARR}.
A solution to Problem~\ref{t-1/q-delta-conjecture-problem} could extend to other reflection groups
to give a marriage of Parking and Delta.

\section{Acknowledgements}
\label{Acknowledgements}

The authors are thankful to Jonathan Chan,
Adriano Garsia, Jia Huang, Kyle Meyer, Satoshi Murai, Vic Reiner, Jeff Remmel,  Dennis
Stanton, Dennis White, and Andy Wilson for many helpful conversations.
In particular, B. Rhoades thanks A. Garsia for the suggestion of modeling the rings $R_{n,k}$ by 
symmetric group actions on point sets.
J. Haglund was partially supported by NSF Grant DMS-1600670.
B. Rhoades was partially supported by NSF Grant DMS-1500838.
M. Shimozono was partially supported by NSF Grant DMS-1600653.
The authors thank the Korea Institute for Advanced Study for organizing the Summer School and
Workshop on Algebraic Combinatorics in June 2016, where this research was developed.
The authors thank an anonymous referee for their careful reading of this paper.
The authors thank Sean Griffin for pointing out a (corrected) mistake in the proof of 
Lemma~\ref{mu-is-isomorphism}.


\begin{thebibliography}{99}
 
 \bibitem{ABR} R. Adin, F. Brenti, and Y. Roichman.  Descent representations and multivariate statisitcs.
 {\it Trans. Amer. Math. Soc.}, {\bf 357} (2005), 3051--3082.
 
 \bibitem{Allen}  E. E. Allen.  The descent monomials and a basis for the diagonally
 symmetric polynomials.  {\it J. Algebraic Combin.}, {\bf 3} (1994), 5--16.
 
 \bibitem{ARR}  D. Armstrong, V. Reiner, and B. Rhoades.  Parking spaces.
 {\it Adv. Math.}, {\bf 269} (2015), 647--706.
 
 \bibitem{Artin}  E. Artin.  {\it Galois Theory,} Second edition.
 Notre Dame Math Lectures, no. 2.  Notre Dame: University of Notre Dame, 1944.
 
 \bibitem{Bergeron}  F. Bergeron.  {\it Algebraic Combinatorics and Coinvariant Spaces.}
 CMS Treatises in Mathematics.
Boca Raton:  Taylor and Francis, 2009.

\bibitem{BR}  A. Berget and B. Rhoades.  Extending the parking space.  {\it J. Comb. Theory, Ser. A,}
{\bf 123 (1)} (2014), 43--56.
 
 \bibitem{CR}  J. Chan and B. Rhoades.  Generalized coinvariant algebras for wreath products.
Submitted, 2017.  {\tt arXiv:1701.06256}.
 
 \bibitem{C}  C. Chevalley.  Invariants of finite groups generated by reflections.
 {\it Amer. J. Math.}, {\bf 77 (4)} (1955), 778--782.
 
 \bibitem{FGP}  S. Fomin, S. Gelfand, and A. Postnikov.  Quantum Schubert polynomials.
 {\it J. Amer. Math. Soc.}, {\bf 10} (1997), no. 3, 565--596.
 
 \bibitem{Garsia}  A. M. Garsia.  Combinatorial methods in the theory of Cohen-Macaulay rings.
 {\it Adv. Math.}, {\bf 38} (1980), 229--266.
 
 \bibitem{GHRY}  A. M. Garsia, J. Haglund, J. Remmel, and M. Yoo.
 A proof of the Delta Conjecture when $q = 0$.
 Preprint, 2017.
 {\tt arXiv:1710.07078}.
 
 \bibitem{GP}  A. M. Garsia and C. Procesi.  On certain graded $S_n$-modules and the $q$-Kostka
 polynomials.  {\it Adv. Math.}, {\bf 94 (1)} (1992), 82--138.
 
 \bibitem{GS}  A. M. Garsia and D. Stanton.  Group actions on Stanley-Reisner rings and invariants of permutation
 groups.  {\it Adv. Math.}, {\bf 51 (2)} (1984), 107--201.
 
  \bibitem{Hagbook} J. Haglund.  {\it The $q$,$t$-Catalan numbers and the space of diagonal harmonics}.
    University Lecture Series, Vol. 41, American Mathematical Society, Providence, RI, (2008).
      With an appendix on the combinatorics of Macdonald polynomials.
 
 \bibitem{HHL}  J. Haglund, M. Haiman, and N. Loehr.  A combinatorial formula for nonsymmetric 
 Macdonald polynomials.  {\it Amer. J. of Math.}, {\bf 103} (2008), 359--383.
 
  \bibitem{HHLRU} J. Haglund, M. Haiman, N. Loehr, J. B. Remmel, and A. Ulyanov.  A combinatorial
 formula for the character of the diagonal coinvariants. {\it Duke Math. J.} {\bf 126} (2005), pp. 195--232.
 
 \bibitem{HLMV}  J. Haglund, K. Luoto, S. Mason, and S. van Willigenburg.
 Refinements of the Littlewood-Richardson Rule.
 {\it Trans. Amer. ath. Soc.}, {\bf 363} (2011), 1665--1686.

\bibitem{HRW}  J. Haglund, J. Remmel, and A. T. Wilson.  The Delta Conjecture.  
Accepted, {\it Trans. Amer. Math. Soc.}, 2016.  {\tt arXiv:1509.07058}.

\bibitem{HaimanConj}  M. Haiman.   Conjectures on the quotient ring by diagonal invariants.
{\it J. Algebraic Combin.}, {\bf 3 (1)} (1994),  17--76

\bibitem{Haiman} M. Haiman.  Vanishing theorems and character formulas for the Hilbert scheme of points in the plane.
{\it Invent. Math.}, {\bf 149 (2)}  (2002), 371--407.

\bibitem{HR}  J. Huang and B. Rhoades.  Ordered set partitions and the 0-Hecke algebra.
Submitted, 2016.
{\tt arXiv:1611.01251}.

\bibitem{HRS}  J. Huang. B. Rhoades, and T.
Scrimshaw.  Hall-Littlewood polynomials and a Hecke action on ordered 
set partitions.
Submitted, 2017.
{\tt arXiv:1709.07995}.


\bibitem{Macdonald}  I. G. Macdonald.  {\it Symmetric Functions and Hall Polynomials,} Second edition.
Oxford Mathematican Monographs.
New York: The Clarendon Press Oxford University Press, 1995.
With contributions by A. Zelevinsky, Oxford Science Publications.

\bibitem{Mason}  S. Mason. An explicit construction of the type A Demazure atoms.
{\it J. Algebraic Combin.}, {\bf 29 (3)} (2009), 295--313.

\bibitem{PR}  B. Pawlowski and B. Rhoades.  A flag variety for the Delta Conjecture.
In preparation, 2017.

\bibitem{PS}  A. Postnikov and B. Shapiro.  Trees, parking functions, syzygies, and deformations of monomial
ideals.  {\it Trans. Amer. Math. Soc.}, {\bf 356} (2004), 3109--3142.

\bibitem{RW}  J. Remmel and A. T. Wilson.  An extension of MacMahon's Equidistribution
Theorem to ordered set partitions.
{\it J. Combin. Theory Ser. A}, {\bf 134} (2015), 242--277.

\bibitem{Rhoades}  B. Rhoades.  Ordered set partition statistics and the Delta Conjecture.
Preprint, 2016. {\tt arXiv:1605.04007}.

\bibitem{Romero}  M. Romero.  The Delta Conjecture at $q = 1$.
Preprint, 2016.  {\tt arXiv:1609.04865}.

%\bibitem{ST}  G. C. Shephard and J. A. Todd.  Finite unitary reflection groups.  {\it Can. J. Math.},
%{\bf 6} (1954), 274--304.



\bibitem{Stanley}  R. P. Stanley.  Invariants of finite groups and their applications to combinatorics.
{\it Bull. Amer. Math. Soc.}, {\bf 1} (1979), 475--511.

\bibitem{Stein} E. Steingr\'imsson.  Statistics on Ordered Partitions of Sets.  Preprint, 2014.
{\tt arXiv:0605670}.

\bibitem{Sturmfels}  B. Sturmfels.  {\it Algorithms in Invariant Theory}.
Springer-Verlag, Berlin, 1993.


\bibitem{WMultiset}  A. T. Wilson.  An extension of MacMahon's Equidistribution Theorem
to ordered multiset partitions. 
{\it Electron. J. Combin.}, {\bf 23 (1)} (2016), P1.5.


  
\end{thebibliography}
\end{document}